\documentclass[10pt,a4paper,reqno]{amsart}
\usepackage[foot]{amsaddr}

\usepackage{amsmath,amsfonts,amsthm,epsfig,graphicx,mathtools,tikz-cd, amssymb,setspace,enumerate,enumitem}
\usepackage[foot]{amsaddr}
\usepackage{caption}
\usepackage{mathrsfs,color}
\usepackage[margin=1in]{geometry}
\usepackage{mathabx}

\usepackage{mathtools}

\DeclarePairedDelimiter\floor{\lfloor}{\rfloor}
\mathtoolsset{showonlyrefs}

\usepackage[page,header]{appendix}
\usepackage{titletoc}

\newcommand{\rd}{\,\mathrm{d}}
\numberwithin{equation}{section}
\newtheorem{theorem}{Theorem}[section]
\newtheorem{lemma}[theorem]{Lemma}
\newtheorem{corollary}[theorem]{Corollary}
\newtheorem{proposition}[theorem]{Proposition}
\newtheorem{definition}[theorem]{Definition}
\newtheorem{remark}[theorem]{Remark}

\def\bu{{\bf u}}
\def\bx{{\bf x}}
\def\by{{\bf y}}
\def\bz{{\bf z}}

\def\cI{\mathcal{I}}
\def\cJ{\mathcal{J}}
\def\cK{\mathcal{K}}
\def\cL{\mathcal{L}}

\def\cR{\mathcal{R}}
\def\cF{\mathcal{F}}

\def\pv{\textnormal{p.v.}}

\def\sgn{\textnormal{sgn}}
\def\supp{\textnormal{supp\,}}
\def\diam{\textnormal{diam\,}}
\def\dist{\textnormal{dist\,}}
\newcommand{\R}{\mathbb{R}}
\newcommand{\Rd}{{\R^d}}
\newcommand{\PP}{\mathcal P}
\newcommand{\MM}{\mathcal M}
\newcommand{\Ptwo}{\PP_2(\R^d)}

\newcommand{\Pd}{\PP(\R^d)}
\newcommand{\Md}{\MM(\R^d)}

\title[From radial symmetry to fractal behavior]{From radial symmetry to fractal behavior of aggregation equilibria for repulsive-attractive potentials}

\author{Jos\'e A. Carrillo$^{\dagger}$, Ruiwen Shu$^{\dagger}$}
\date{\today}
\subjclass[2020]{}
\address[$\dagger$]{Mathematical Institute, University of Oxford, Oxford OX2 6GG, UK}

\begin{document}

\maketitle

\begin{abstract}
   For the interaction energy with repulsive-attractive potentials, we give generic conditions which guarantee the radial symmetry of the local minimizers in the infinite Wasserstein distance.
    As a consequence, we obtain the uniqueness of local minimizers in this topology for a class of interaction potentials. We introduce a novel notion of concavity of the interaction potential allowing us to show certain fractal-like behavior of the local minimizers. We provide a family of interaction potentials such that the support of the associated local minimizers has no isolated points and any superlevel set has no interior points.
\end{abstract}

\hspace{0.83cm}{\footnotesize Keywords: Energy minimization, aggregation equation, radial symmetry, fractal behavior}

\section{Introduction}

Explaining the intricate behavior of coherent structures of active/passive media composed by many interacting agents in mathematical biology and technology has attracted lots of attention in the applied mathematics community. These questions are ubiquitous in collective behavior of animal species, cell aggregates by chemical cues or adhesion forces, granular media and self-assembly of particles, see for instance \cite{MR3143990,To04,Carillo09_Kinetic_Attraction-Repulsion,MR2256869} and the references therein. In many of these models, particular solutions emerge from consensus of movement while their relative positions are determined based only on attraction and repulsion effects~\cite{DCBC,MR3090592}. These equilibrium shapes at the continuum level can be characterized by probability measures $\rho$ for which the balance of attractive and repulsive forces hold. This is equivalent to finding probability measures $\rho$ such that
\begin{equation}\label{eq:basic}
    \nabla W\ast \rho = \int_\Rd \nabla W(\bx-\by)\rho(\by)\rd{\by}= 0 \quad\mbox{ on supp}(\rho) \,,
\end{equation}
with $W\colon\Rd\to (-\infty,\infty]$ being an attractive-repulsive interaction potential between the particles. The richness of the shapes of the support of these aggregation equilibria is quite surprising even for simple potentials~\cite{PhysRevE.84.015203}. Finding particular configurations satisfying \eqref{eq:basic} is a challenging problem due to its highly nonlinear nature since the support of the measure itself is part of the problem and the regularity of the potential plays a key role. These configurations appear naturally as the steady states for the mean-field dynamics associated to the particle system
\begin{equation}\label{1storder}
    \frac{\rd}{\rd{t}}\bx_i=-\frac{1}{N} \sum_{j\neq i}\nabla W(\bx_i-\bx_j), \quad i=1,\cdots,N\,.
\end{equation}
Notice that the system of ODEs \eqref{1storder} is the finite dimensional gradient flow of a discrete interaction energy. Its formal mean-field limit follows the nonlocal partial differential equation 
\begin{equation}\label{eq}
\partial_t \rho + \nabla\cdot (\rho \bu) = 0\,,
\end{equation}
usually referred to as the aggregation equation, where the transport velocity field $\bu(t,\bx)$ is given by
\begin{equation}
\bu(t,\bx) = -\int_\Rd \nabla W(\bx-\by)\rho(t,\by)\rd{\by}\,.
\end{equation}
The aggregation equation \eqref{eq} is the $2$-Wasserstein gradient flow of the \emph{total potential energy}
\begin{equation}\label{energy}
E[\rho] = \frac{1}{2}\int_\Rd\int_\Rd W(\bx-\by)\rho(\by)\rd{\by}\rho(\bx)\rd{\bx}\,.
\end{equation}
In the sequel, we will denote by $V = V[\rho] = W*\rho$ the interaction potential generated by the particle density $\rho$. Notice that 
$\rho(\bx)$ is a steady state of \eqref{eq} if
$\bu(\bx)$ satisfies \eqref{eq:basic},
or equivalently $V[\rho]$ is constant on each of the connected components of $\supp\rho$, modulo regularity of the velocity field.

We emphasize that even if the interaction potential $W$ is radially symmetric, it is quite challenging to prove or disprove radial symmetry of global or local minimizers of the interaction energy \eqref{energy}. Notice that despite of the fact that the energy is rotationally invariant for radial functions with radially symmetric interaction potentials, the uniqueness of global minimizers of the interaction energy, modulo translations, is not known except for particular cases \cite{F35,lopes2017uniqueness} using the linear interpolation convexity (LIC). 

Nothing is essentially known about uniqueness for local miminizers. Actually, in order to discuss about local minimizers of the interaction energy \eqref{energy}, we obviously need to specify the topology in probability measures that we use to measure the distance. Here, we follow previous works \cite{BCLR2,CCP15,CDM16} that showed that transport distances between probability measures are the right tool to deal with this variational problem. We remind the reader the main properties of transport distances in Section 2, in particular, the infinity Wasserstein distance $d_\infty$ plays an important role in order to write Euler-Lagrange conditions for local minimizers \cite{BCLR2,CDP19}. 

The problem of finding global minimizers of the interaction energy, modulo translations, for particular potentials is a classical problem in potential theory \cite{F35,ST97}. More precisely, for the repulsive logarithmic potential with quadratic confinement $W(\bx)=\tfrac{|\bx|^2}2-\ln |\bx|$ in 2D, it is known~\cite{F35} that the unique global minimizer is the characteristic function of a suitable Euclidean ball. When the repulsive singularity at zero is stronger than Newtonian but still locally integrable, regularity results for $d_\infty$-local minimizers have been obtained \cite{CDM16} and the uniqueness is known for particular power-law potential cases again \cite{CV1,CV2,lopes2017uniqueness}. The existence of compactly supported global minimizers of the interaction energy for generic interaction potentials was obtained in \cite{CCP15,SST15} based on the condition of H-stability of interaction potentials. 

Some qualitative properties of the support of the minimizers are known depending on the smoothness/singularity of the potential at the origin.
Interaction potentials which are at least $C^2$ smooth at the origin generically lead to minimizers concentrated on Dirac points for which particular geometric constraints and explicit forms are known \cite{CFP,LM21}. 
The dimensionality of the support of the $d_\infty$-local minimizers was estimated in terms of the repulsive singularity strength of weakly singular interaction potentials at the origin in \cite{BCLR2} showing that the more repulsive the potential at zero gets the larger the support of the $d_\infty$-local minimizers is. We refer to weakly singular repulsive-attractive potentials as potentials with a repulsive singularity at the  origin behaving between Newtonian singularity and smooth quadratic behavior at the origin.
Other related problems include singular anisotropic potentials~\cite{MRS19,CMMRSV,CMMRSV2}, in which the explicit form of the global minimizers are known in particular cases, and interaction energies with constraints \cite{BCT18,CT2,frank2016aliquid,frank2019proof}. Explicit stationary solutions of \eqref{eq:basic} are known for some power-law potentials~\cite{CH}: attractive power is an even integer and weakly singular at the origin. They were expected to be indeed global minimizers supported by strong numerical evidence~\cite{gutleb2020computing}. 

As a conclusion, a key problem not covered by the current literature is to find sufficient conditions by variational methods for radial symmetry or the break of radial symmetry of $d_\infty$-local minimizers for weakly singular repulsive-attractive potentials. In order to attack these issues, our first objective is to further exploit and refine the convexity properties of the interaction energy to study the radial symmetry and uniqueness of $d_\infty$-local minimizers of \eqref{energy}. 
The crucial assumption on the interaction energy is the LIC property, which basically means that $E$ is convex along the linear interpolation between any two measures $\rho_0$ and $\rho_1$. It is well-known that the global minimizer is unique and radially-symmetric under the LIC assumption for certain particular functionals~\cite{Lieb81,lopes2017uniqueness,MRS19}.

The main goal of the first part of this work, Sections 3 to 5, is to find sufficient conditions leading to radial symmetry of local minimizers of \eqref{energy} and its consequences.
Assuming the LIC property, we first show that any $d_\infty$-local minimizer of the interaction energy \eqref{energy} is radially-symmetric, see Theorem \ref{thm_Wirs}. Then, by imposing an extra assumption on the sign of $\Delta^2 W$, we obtain the uniqueness of $d_\infty$-local minimizers modulo translations, see Theorem \ref{thm_min0}. As a test case of our theory, we apply our results to the power-law potentials $W(\bx) = \frac{|\bx|^a}{a}-\frac{|\bx|^b}{b}$, and identify the ranges of $a$ and $b$ for which we show the radial symmetry and uniqueness of $d_\infty$-local minimizers, see Theorem \ref{thm_ab}. In particular, we prove that some of the steady states given by the explicit formula in \cite{CH} are the global minimizers of the interaction energy \eqref{energy}, and also their unique $d_\infty$-local minimizer, see also related results in one dimension \cite{frank2021minimizers}. This also confirms accurate numerical simulations of equilibrium measures \cite{gutleb2020computing}.

We emphasize that Sections 3-5 together lead to the first results in the literature proving radial symmetry and uniqueness of $d_\infty$-local minimizers of the interaction energy \eqref{energy} for a general family of interaction potentials. The importance of showing radial-symmetry and uniqueness is not only from the variational viewpoint but also from the evolutionary viewpoint of gradient flows associated to the interaction energy. The connection to the long time asymptotics of the corresponding aggregation equation~\cite{BertozziLaurent,BertozziBrandman,BertozziCarrilloLaurent,BertozziLaurentRosado,BalagueCarrilloLaurentRaoul,BCLR2,CDFLS} is not explored in this work although there are still important open problems under different assumptions on the interaction potential. Nevertheless, we remark that the radial symmetry of steady states and the gradient flow structure of the aggregation equation are crucial properties for showing precise long time asymptotics both for the aggregation equation with particular interaction potentials in \cite{BertozziLaurentLeger,CV1,CV2,shu2021newtonian} and also for the aggregation-diffusion equations in \cite{CHVY19,CHMV18,kaib,DYY20,shu2020equilibration,shu2020tightness}.

The next main question that we want to address in this work is to give sufficient conditions to allow for non-radial local/global minimizers and even fractal behavior on the structure of their support. The radial symmetry is broken for interaction potentials at least $C^2$ smooth at the origin. The support of their stationary states consists of finite number of isolated Dirac points with some conditions \cite{CFP}, and some particular configurations such as simplices appear as asymptotic limits for power-law potentials \cite{LM21}.
Stationary states with complex structure have been reported in the numerical literature and by studying the stability/instability of Delta ring solutions \cite{BalagueCarrilloLaurentRaoul,SunUminskyBertozzi,BKHUB15}. The dimensionality of the support of stationary states was estimated in \cite{BCLR2} as mentioned earlier. In this quest, we can wonder if fractal behavior of the support of the minimizers appears among the natural family of power-law potentials. It was numerically observed in \cite{BCLR2} that stationary states for power-law potentials seem not to show fractal behavior, i.e., the dimension of the support of the steady states seems to be an integer. We show in Section 6, Theorem \ref{thm_1d}, that this is in fact the case in one dimension, i.e. fractal behavior is not possible at least in one dimension for some power-law potentials.

The main result of the second part of this work is to give a generic family of potentials for which we have fractal-like behavior. In order to achieve this, Section 7 introduces a novel notion of concavity of the interaction potential $W$ allowing us to show certain fractal behavior on superlevel sets of $d_\infty$-local minimizers of the interaction energy \eqref{energy}. This notion of concavity is based on negative regions of the Fourier transform of the potential, in contrast to the LIC property and a slightly stronger notion of convexity, the Fourier-LIC (FLIC) property defined in Section 2. More precisely, the main result, Theorem \ref{lem_concave}, asserts that if the interaction potential is infinitesimal-concave, then any superlevel set of $d_\infty$-local minimizers of the interaction energy \eqref{energy} does not contain interior points. Section 8 provides explicit constructive examples of infinitesimal-concave potentials in any dimension based on a careful modification of power-law kernels in Fourier variables, see Theorem \ref{thm_cons1}. We finally show in Corollary \ref{cor_frac} that for these potentials the behavior of the support of the corresponding $d_\infty$-local minimizers of the interaction energy \eqref{energy} is almost fractal, in the sense that the interior of any superlevel set is empty, and moreover the support does not contain isolated points. A related idea of breaking the symmetry of minimizers of the interaction energy \eqref{energy} via a finite number of unstable Fourier modes for the uniform distribution on the sphere was described in \cite{von2012soccer}.

Finally, Section 9 provides another constructive example in which a steady state is the uniform distribution on a Cantor set, see Theorems \ref{thm_Wkss} and \ref{thm_M}.
The main idea of this construction is to produce a potential in a recursive hierarchical manner that introduces some kind of concavity at a sequence of small scales.
These three sections 7-9 show all together that the behavior of the support of $d_\infty$-local minimizers of infinitesimal concave potentials can be hugely sophisticated. This is further corroborated by the numerical simulations in Section 10 illustrating by means of particle methods the intricate fractal-like structures of the steady states. One of the remaining open problems is to prove or disprove the fractal behavior of $d_\infty$-local minimizers for weakly singular power-law like interaction potentials in $d\ge 2$. 


\section{Preliminaries \& Convexity}

We write $\Pd$ for the set of Borel probability measures. Given $1\leq p<\infty$, we write $\mathcal{P}_p(\R^d)$ for the subset of $\Pd$ of measures with finite $p$th moment. The \emph{$p$th Wasserstein distance} $d_p(\mu,\nu)$ between two probability measures $\mu$ and $\nu$ belonging to $\mathcal{P}_p(\R^d)$ is
\begin{equation*}
	d_p(\mu,\nu) = \min_{\pi \in \Pi(\mu,\nu)} \left( \int_{\R^d\times \R^d} |\bx-\by|^p \rd \pi(\bx,\by) \right)^{1/p} ,
\end{equation*}
where $\Pi(\mu,\nu)$ is the set of transport plans between $\mu$ and $\nu$; i.e., $\Pi(\mu,\nu)$ is the subset of $\mathcal{P}(\Rd\times\Rd)$ of measures with $\mu$ as first marginal and $\nu$ as second marginal. We also define the \emph{$\infty$-Wasserstein distance} $d_\infty(\mu,\nu)$, whenever $\mu$ and $\nu$ are compactly supported, by
\begin{equation*}
	d_\infty(\mu,\nu) = \inf_{\pi\in\Pi(\mu,\nu)} \sup_{(\bx,\by) \in \supp\pi} |\by-\bx|,
\end{equation*}
where the $\supp$ denotes the support.

We will adopt the following notation for Fourier transform and its inverse
\begin{equation}
\cF[f](\xi) = \hat{f}(\xi) = \int_\Rd f(\bx)e^{-i\bx\cdot\xi}\rd{\bx},\quad \cF^{-1}[g](\bx)= \check{g}(\bx) = \frac{1}{(2\pi)^d}\int_\Rd g(\xi)e^{i\bx\cdot\xi}\rd{\xi}\,,
\end{equation} 
for all $\xi,\bx\in\Rd$. Then
$
\hat{\delta} = 1,\, \check{\delta} = (2\pi)^{-d}
$, 
and
\begin{equation}
\cF\left[\exp\left(-\frac{|\cdot|^2}{2}\right)\right](\xi) =(2\pi)^{d/2}\exp(-\frac{|\xi|^2}{2})\,.
\end{equation}

We will use in this work different notions of convexity for the interaction energy functional \eqref{energy}. For the sake of notational simplicity, we will denote by $E$ either the energy functional \eqref{energy} acting on probability measures or the bilinear form acting on signed measures associated to the interaction potential $W$. We also denote particle densities by $\rho$ whenever it is a probability measure and $\mu$ if it is a signed measure. For notational simplicity, we will use $\rho(\bx) \,\rd{\bx}$ as the integration against the measure $\rho$ no matter if it can be identified with a Lebesgue integrable function or not. In the sequel, we also use $C$ and $c$ to refer to generic positive constants.

We start by the simplest notion of \emph{linear interpolation convexity} (LIC): we say that the interaction energy $E$ is LIC, if for any probability measures $\rho_0,\rho_1\in\Ptwo$,  $\rho_0\ne \rho_1$, such that $E[\rho_0]<\infty$ and $E[\rho_1]<\infty$ with the same total mass and center of mass, the function $t\mapsto E[(1-t)\rho_0 + t \rho_1]$ is strictly convex, or equivalently, $E[\mu]> 0$ for every nonzero signed measure $\mu\in\Md$ with $\int_\Rd\mu(\bx)\rd{\bx}=\int_\Rd\bx\mu(\bx)\rd{\bx}=0$ and $E[|\mu|]<\infty$. Notice the equivalence by taking $\mu=\rho_0-\rho_1$. This notion of convexity has been classically used in statistical mechanics, see \cite{Lieb81}.

A stronger convexity property is the following: if for any $0<r<R<\infty$, there exists $c>0$ such that
\begin{equation}\label{flic}
E[\mu] \ge c\int_{r\le |\xi| \le R} |\hat{\mu}(\xi)|^2\rd{\xi}\,,
\end{equation}
then we say $E$ has \emph{Fourier linear interpolation convexity} (FLIC). It is straightforward to check that FLIC implies LIC for the interaction energy. As already mentioned in the introduction, these notions of convexity were used in \cite{lopes2017uniqueness}, the author showed that the interaction energy \eqref{energy} associated to the attractive potential $W(x)=\tfrac{|\bx|^a}{a}$ has the LIC property, for $2\le a \le 4$ and that the interaction energy \eqref{energy} associated to the repulsive potential $W(x)=-\tfrac{|\bx|^b}{b}$ has the FLIC property for $-d < b < 0$, since its Fourier transform is shown to be $c|\bx|^{-d-b}$ for some $c>0$. 

It is clear that LIC implies the uniqueness of global minimizers for the interaction energy \eqref{energy} in $\Ptwo$ with fixed total mass and center of mass. Observe also that the sum of an LIC potential and an (F)LIC potential is an (F)LIC potential. The positivity of the Fourier transform of the interaction potential, related to then FLIC convexity, has been used to prove uniqueness of minimizers for interaction energies with asymmetric potentials, see \cite{MRS19,CMMRSV} and it is also related to the concept of H-stability in statistical mechanics to detect phase transitions in aggregation-diffusion equations \cite{CGPS20,CG21} and the existence of compactly supported minimizers for the interaction energy \cite{CCP15,SST15}.

It is important to remark that linear interpolation convexity of the interaction energy as defined above is totally different from displacement convexity of \eqref{energy} in the optimal transportation sense, see \cite{McCann97}. In fact, the energy functional \eqref{energy} is strictly displacement convex as soon as the potential $W$ is strictly convex.

We now remind the reader about the necessary conditions for local minimizers of the interaction energy \eqref{energy} in \cite{BCLR2,CDP19}. The following conditions come from Euler-Lagrange variational arguments using the mass constraint, and they are related to obstacle-like problems as in \cite{CDM16}. Although further hypotheses on the interaction potential $W$ may be considered in various places below, we minimally assume that
\begin{equation}\label{eq:hyp}
\begin{gathered}
	\text{$W\colon\R^d\to (-\infty,\infty]$ is locally integrable (i.e., is in $L^1_\mathrm{loc}(\R^d)$),}\\ 
	\text{$W$ is bounded below, lower semicontinuous and symmetric (i.e., $W(x) = W(-x)$ for all $x\in\R^d$).}
\end{gathered}\tag{\textbf{H}}
\end{equation}
 From now on, we will assume without loss of generality that the interaction potential $W\geq 0$. Notice that these assumptions imply that $V=W*\rho\geq0$ is always a lower semicontinuous function for all $\rho\in\Pd$, see \cite[Lemma 2]{BCLR2}. The results in \cite[Proposition 1]{BCLR2}, together with \cite[Remark 2.3]{CDM16}, give
 
\begin{lemma}\label{lemel1}
If $\rho\in\Pd$ is a compactly supported local energy minimizer in the $d_\infty$-sense for the energy $E$ with interaction potential $W$ satisfying \eqref{eq:hyp}, then there exists $\epsilon_0>0$, such that for any $\bx\in\supp\rho$, 
\begin{equation}
V(\by) \ge V(\bx),\quad \mbox{a.e. }\by\in B(\bx;\epsilon_0).
\end{equation}
\end{lemma}
\begin{lemma}\label{lemel2}
If $\rho\in\Ptwo$ is a local energy minimizer in the $d_2$-sense for the energy $E$ with interaction potential $W$ satisfying \eqref{eq:hyp}, then $V$ is constant on $\supp\rho$ in the sense that $V(\bx)=C_\rho:=2 E[\rho]$ $\rho$-a.e., and
\begin{equation}\label{W2cond}
V(\bx) \ge C_\rho,\quad \mbox{a.e. } \bx \in\Rd\, . 
\end{equation}
\end{lemma}

In order to show the FLIC property of the interaction energy, one can reduce to show \eqref{flic} for compactly supported signed measures $\mu$. 

\begin{lemma}\label{lem:reduction}
Assume then interaction potential $W$ satisfies \eqref{eq:hyp} and that the associated interaction energy $E$ satisfies \eqref{flic} for compactly supported signed measures $\mu$, then the interaction energy $E$ is FLIC.
\end{lemma}

\begin{proof}
To see this, assume that we already proved \eqref{flic} for compactly supported $\mu\in\Md$. Then let $\mu\in\Md$ be a signed measure with $E[|\mu|]<\infty$ and $\mu=\rho_0-\rho_1$ for some $\rho_0,\rho_1\in \Ptwo$ with the same center of mass. Then 
\begin{equation}\label{mumean}
\int_{\Rd} \mu(\bx)\rd{\bx} = \int_{\Rd} \bx\mu(\bx)\rd{\bx} = 0\,,
\end{equation}
and one can write $\mu=\mu_+-\mu_-$ as the positive and negative parts, with $m_\mu:=\int_{\Rd}\mu_+\rd{\bx}=\int_{\Rd}\mu_-\rd{\bx}>0$ and $E[\mu_+]<\infty,\,E[\mu_-]<\infty$.

Fix a compactly supported smooth non-negative radial function $\psi(\bx)$ with $\int_{\Rd}\psi(\bx)\rd{\bx}=1$. 
Let $N\in\mathbb{N}$ be sufficiently large, and let
\begin{equation}
\tilde{\mu}_N = \lambda_N\mu_+\chi_{[-N,N]^d}-\mu_-\chi_{[-N,N]^d},\quad \lambda_N:=\frac{\int_{[-N,N]^d}\mu_-\rd{\bx}}{\int_{[-N,N]^d}\mu_+\rd{\bx}}
\end{equation}
be a compactly supported signed measure. Since $\lim_{N\rightarrow\infty}\int_{[-N,N]^d}\mu_+\rd{\bx} = m_\mu>0$, $\lambda_N$ is well-defined for sufficiently large $N$ with $\lim_{N\rightarrow\infty}\lambda_N=1$, and we have $\int_{\Rd}\tilde{\mu}_N\rd{\bx}=0$. 

Then define 
\begin{equation}
\mu_N(\bx) = \tilde{\mu}_N(\bx) + c_N(\psi(\bx-\bx_N)-\psi(\bx+\bx_N)),\quad |\bx_N|=1,\,c_N\ge 0,\quad 2c_N\bx_N = -\int_{\Rd}\bx\tilde{\mu}_N(\bx)\rd{\bx}.
\end{equation}
Then $\mu_N$ is compactly supported, satisfies $\int_{\Rd} \mu(\bx)\rd{\bx}=0$, and
\begin{equation}\begin{split}
\int_{\Rd} \bx\mu(\bx)\rd{\bx} = & \int_{\Rd} \bx\tilde{\mu}_N(\bx)\rd{\bx} + c_N\int_{\Rd} \bx\psi(\bx-\bx_N)\rd{\bx}- c_N\int_{\Rd} \bx\psi(\bx+\bx_N)\rd{\bx} \\
= & \int_{\Rd} \bx\tilde{\mu}_N(\bx)\rd{\bx} + c_N\bx_N + c_N\bx_N = 0.
\end{split}\end{equation}
Also, $\lim_{N\rightarrow\infty}c_N=0$.
Then we will show that
\begin{equation}
|E[\mu]-E[\mu_N]| \le |E[\mu]-E[\mu\chi_{[-N,N]^d}]|+|E[\mu\chi_{[-N,N]^d}]-E[\tilde{\mu}_N]|+|E[\tilde{\mu}_N]-E[\mu_N]|
\end{equation}
converges to zero as $N\rightarrow\infty$. For the first term on the RHS, this is a consequence of $E[|\mu|]<\infty$ since 
\begin{equation}
\left|\iint_{(x,y)\notin [-N,N]^d\times [-N,N]^d}\!\!\!\!\!\!W(\bx-\by)\mu(\by)\mu(\bx)\rd{\by}\rd{\bx}\right| \le \iint_{(x,y)\notin [-N,N]^d\times [-N,N]^d}\!\!\!\!\!\!W(\bx-\by)|\mu(\by)||\mu(\bx)|\rd{\by}\rd{\bx}
\end{equation}
and the RHS converges to zero as $N\rightarrow\infty$.
For the second term, this is a consequence of $E[|\mu|]<\infty$ and $\lim_{N\rightarrow\infty}\lambda_N=1$, since $\tilde{\mu}_N=\mu\chi_{[-N,N]^d}+(\lambda_N-1)\mu_+\chi_{[-N,N]^d}$ and then
\begin{equation}\begin{split}
& |E[\mu\chi_{[-N,N]^d}]-E[\tilde{\mu}_N]| \\ & = \left|(\lambda_N-1)\int_{\Rd}\int_{\Rd}W(\bx-\by)\mu(\by)\chi_{[-N,N]^d}(\by)\rd{\by}\mu_+(\bx)\chi_{[-N,N]^d}(\bx)\rd{\bx} + (\lambda_N-1)^2E[\mu_+\chi_{[-N,N]^d}]\right|
\end{split}\end{equation}
with the last integral and $E[\mu_+\chi_{[-M,M]^d}]$ being finite.

For the third term, we first take a non-negative compactly-supported radial smooth function $\Psi$ with $\Psi(\bx-\bz)\ge \psi(\bx-\bx_1)$ for any $|\bx_1|=1$ and 
\begin{equation}\label{zcond}
|\bz|\le 2\frac{\int_{\Rd}|\bx \mu(\bx)|\rd{\bx}}{\int_{\Rd} |\mu|\rd{\bx}}.
\end{equation}
Notice that $E[\Psi]<\infty$. 
%
%
Take $M>0$ large enough so that $\int_{\Rd}|\mu|\chi_{[-M,M]^d}\rd{\bx}\ge \frac{1}{2}\int_{\Rd}|\mu|\rd{\bx}$. Then the FLIC property for compactly supported measures implies an estimate along a linear interpolation curve
\begin{align*}
    E\left[\frac{1}{2}\Lambda_M\Psi(\cdot-\bz_M)+\frac{1}{2}|\mu|\chi_{[-M,M]^d}\right] &\le \max\{E[\Lambda_M\Psi],E[|\mu|\chi_{[-M,M]^d}]\} \\
    &\le \max\left\{\left(\int_{\Rd}| \mu|\rd{\bx}\right)^2E[\Psi],E[|\mu|]\right\}=C
\end{align*}
with
\begin{equation}
\Lambda_M = \frac{1}{\|\Psi\|_{L^1}}\int_{\Rd}|\mu|\chi_{[-M,M]^d}\rd{\bx},\quad \bz_M = \frac{1}{\int_{\Rd}|\mu|\chi_{[-M,M]^d}\rd{\bx}}\int_{\Rd}\bx|\mu(\bx)|\chi_{[-M,M]^d}\rd{\bx}
\end{equation}
since both $\Lambda_M\Psi(\cdot-\bz_M)$ and $|\mu|\chi_{[-M,M]^d}$ have same total mass and center of mass. Also, since $\int_{\Rd}|\mu|\chi_{[-M,M]^d}\rd{\bx}\ge \frac{1}{2}\int_{\Rd}|\mu|\rd{\bx}$, $\bz_M$ satisfies \eqref{zcond}. This implies 
\begin{equation}
\int_{\Rd}\int_{\Rd}W(\bx-\by)|\mu(\by)|\chi_{[-M,M]^d}(\by)\rd{\by}\psi(\bx-\bx_1)\rd{\bx} \le C
\end{equation}
for any $M$ sufficiently large and $|\bx_1|=1$. Taking $M\rightarrow\infty$, we get
\begin{equation}
\int_{\Rd}\int_{\Rd}W(\bx-\by)|\mu(\by)|\rd{\by}\psi(\bx-\bx_1)\rd{\bx} \le C.
\end{equation}
Therefore, we obtain that
\begin{equation}\begin{split}
|E[\tilde{\mu}_N]-E[\mu_N]| \le & c_N \int_{\Rd}\int_{\Rd}W(\bx-\by)|\tilde{\mu}_N(\by)|\rd{\by}(\psi(\bx-\bx_N)+\psi(\bx+\bx_N))\rd{\bx} \\&+ c_N^2E[\psi(\bx-\bx_N)+\psi(\bx+\bx_N)] \\
\le & c_N C\int_{\Rd}\int_{\Rd}W(\bx-\by)|\mu(\by)|\rd{\by}\Psi(\bx)\rd{\bx} + 4c_N^2E[\Psi] \le C(c_N+c_N^2) \\
\end{split}\end{equation}
converges to zero as $N\rightarrow\infty$, since $\lim_{N\rightarrow\infty}c_N=0$.

Similarly one can show that $\lim_{N\rightarrow\infty}\int_{r\le |\xi| \le R} |\hat{\mu}_N(\xi)|^2\rd{\xi}=\int_{r\le |\xi| \le R} |\hat{\mu}(\xi)|^2\rd{\xi}$. In fact, we write
\begin{equation}\label{EmuL2}\begin{split}
& \left|\int_{r\le |\xi| \le R}  |\hat{\mu}(\xi)|^2\rd{\xi}-\int_{r\le |\xi| \le R} |\hat{\mu}_N(\xi)|^2\rd{\xi} \right| \le  \left|\int_{r\le |\xi| \le R} |\hat{\mu}(\xi)|^2\rd{\xi}-\int_{r\le |\xi| \le R} |\cF(\mu\chi_{[-N,N]^d})(\xi)|^2\rd{\xi}\right| \\
& + \left|\int_{r\le |\xi| \le R} |\cF(\mu\chi_{[-N,N]^d})(\xi)|^2\rd{\xi}-\int_{r\le |\xi| \le R} |\hat{\tilde{\mu}}_N(\xi)|^2\rd{\xi}\right| + \left|\int_{r\le |\xi| \le R} |\hat{\tilde{\mu}}_N(\xi)|^2\rd{\xi}-\int_{r\le |\xi| \le R} |\hat{\mu}_N(\xi)|^2\rd{\xi}\right|.
\end{split}\end{equation}
First notice that $\hat{\mu}, \cF(\mu\chi_{[-N,N]^d}), \hat{\tilde{\mu}}_N$ are uniformly bounded in $L^\infty(\Rd)$. The first term in \eqref{EmuL2} converges to zero as $N\rightarrow\infty$ since
\begin{equation}
|\cF(\mu-\mu\chi_{[-N,N]^d})(\xi)| = |\cF(\mu\chi_{([-N,N]^d)^c})(\xi)| \le \int_{([-N,N]^d)^c} |\mu|\rd{\bx} \rightarrow 0,\quad \text{ as }N\rightarrow\infty
\end{equation}
uniformly in $\xi\in\Rd$. The second term converges to zero since 
\begin{equation}
|\cF(\mu\chi_{[-N,N]^d}-\tilde{\mu}_N)(\xi)| = |(\lambda_N-1)\hat{\mu}_+(\xi)| \le |\lambda_N-1|\int_{\Rd} |\mu|\rd{\bx} \rightarrow 0,\quad \text{ as }N\rightarrow\infty.
\end{equation}
The third term converges to zero since 
\begin{equation}
|\cF(\tilde{\mu}_N-\mu_N)(\xi)| = |c_N\cF(\psi(\cdot-\bx_N)-\psi(\cdot+\bx_N))(\xi)| \le 2c_N\int_{\Rd} \psi\rd{\bx} \rightarrow 0,\quad \text{ as }N\rightarrow\infty.
\end{equation}

The compactly supported signed measure $\mu_N$ satisfying \eqref{mumean}, and therefore satisfies $\eqref{flic}$ by assumption. Then \eqref{flic} for $\mu$ follows by taking $N\rightarrow\infty$. 
\end{proof}

In order to obtain further consequences of the LIC convexity of the interaction energy, we need stronger assumptions on the potential $W$:
\begin{equation}\label{eq:hyp2}
\begin{gathered}
	\text{$W$ satisfies \eqref{eq:hyp} and $W$ is either continuous on $\mathbb{R}^d\backslash \{0\}$ with $W(0)=+\infty$ or $W\in C(\mathbb{R}^d)$.}\\ 
	\text{$\forall R>0$, there exists $C_R>0$ such that
$\frac{1}{|B(\bx;r)|}\int_{B(\bx;r)}W(\by)\rd{\by} \le C_R W(\bx), \forall |\bx|\le R,\,0<r\le R$.}
\end{gathered}\tag{\textbf{H-s}}
\end{equation}
These hypotheses can be verified for power law interaction potentials. An important consequence of the LIC convexity of the interaction energy is:

\begin{theorem}\label{thm_W2g}
Assume the interaction energy $E$ associated to a potential $W$ satisfying \eqref{eq:hyp2} is LIC and that there exists a global minimizer of $E$ in $\Ptwo$. Then given any $\rho\in\Ptwo$ satisfying the necessary condition \eqref{W2cond} for the $d_2$-local minimizer is the global minimizer.
\end{theorem}

To show this, we need a technical lemma where the stronger hypotheses are essential.
\begin{lemma}\label{lem_psi}
Consider the interaction energy $E$ associated to a potential $W$ satisfying \eqref{eq:hyp2} and suppose $\rho\in\Ptwo$ is compactly supported and satisfies $E[\rho]<\infty$. Given a non-negative radially decreasing smooth compactly supported mollifier $\psi$ supported on $B_1$ with $\int_{\Rd}\psi\rd{\bx}=1$. Denote $\psi_\alpha(\bx) = \frac{1}{\alpha^d}\psi(\frac{\bx}{\alpha})$, then $E[\rho*\psi_\alpha]<\infty$ for sufficiently small $\alpha>0$, and
\begin{equation}
\lim_{\alpha\rightarrow0+}E[\rho*\psi_\alpha]=E[\rho].
\end{equation}
\end{lemma}

\begin{proof} Notice it is straightforward to show that $\rho*\psi_\alpha$ converges weakly to $\rho$ as $\alpha\rightarrow0+$, and conclude that $E[\rho]\le \liminf_{\alpha\rightarrow0+}E[\rho*\psi_\alpha]$. This lemma improves this $\liminf$ result to a limit.
Let $R>0$ be such that $\supp\rho\subseteq B_{R/2}$.
\begin{equation}
E[\rho*\psi_\alpha]=\frac{1}{2}\int_{\Rd}(W*\rho*\psi_\alpha)(\bx)(\rho*\psi_\alpha)(\bx)\rd{\bx}=\frac{1}{2}\int_{B_{R/2}}(W*\psi_\alpha*\psi_\alpha*\rho)(\bx)\rho(\bx)\rd{\bx}.
\end{equation}
Notice that $\Psi_\alpha=\psi_\alpha*\psi_\alpha$  is also a compactly supported (on $B_{2\alpha}$) non-negative radially decreasing smooth function with $\int_{\Rd}\Psi_\alpha \rd{\bx}=1$. Therefore, for any $\bx\in B_{R/2}$,
\begin{equation}
(W*\Psi_\alpha)(\bx) = \int_{B_{2\alpha}}W(\bx-\by)\Psi_\alpha(\by)\rd{\by}  = \int_{B_{2\alpha}}W(\bx-\by)\int_0^{\Psi_\alpha(\by)}\rd{h}\rd{\by} = \int_0^{\Psi_\alpha(0)}\!\! \int_{B_{\Psi_\alpha^{-1}(h)}}W(\bx-\by)\rd{\by}\rd{h},
\end{equation}
where $\Psi_\alpha^{-1}$ denotes the inverse function of $\Psi_\alpha$ as a decreasing of $r\in [0,2\alpha]$. By the assumptions \eqref{eq:hyp2}, if $\alpha \le R/2$, we have
\begin{equation}
\int_0^{\Psi_\alpha(0)}\!\! \int_{B_{\Psi_\alpha^{-1}(h)}}W(\bx-\by)\rd{\by}\rd{h} \le C_R\int_0^{\Psi_\alpha(0)} W(\bx)\int_{B_{\Psi_\alpha^{-1}(h)}}\!\!\!\rd{\by}\rd{h} = C_R W(\bx)\int_{B_{2\alpha}}\Psi_\alpha(\by)\rd{\by} = C_R W(\bx).
\end{equation}
This implies
\begin{equation}
E[\rho*\psi_\alpha]\le \frac{C_R}{2}\int_{B_{R/2}}(W*\rho)(\bx)\rho(\bx)\rd{\bx} = C_R E[\rho]
\end{equation}
i.e., $E[\rho*\psi_\alpha]$ is finite. 

Then write
\begin{equation}\label{rhopsi}
E[\rho*\psi_\alpha]=\frac{1}{2}\int_{\Rd}\int_{\Rd}(W*\Psi_\alpha)(\bx-\by)\rho(\by)\rd{\by}\rho(\bx)\rd{\bx}.
\end{equation}
We now use the continuity assumptions on $W$ in \eqref{eq:hyp2}. If $W$ is continuous on $\R^d$, then 
$(W*\Psi_\alpha)(\bx-\by)$ converges to $W(\bx-\by)$ at every $(\bx,\by)\in\Rd\times\Rd$. Otherwise in case $W(0)=+\infty$, by the continuity of $W$ away from 0, $(W*\Psi_\alpha)(\bx-\by)$ converges to $W(\bx-\by)$ at every $(\bx,\by)\in\Rd\times\Rd$ with $\bx\neq\by$. We also have that the diagonal set $x=y$ is negligible with respect to the product measure $\rho(\by)\rd{\by}\rho(\bx)\rd{\bx}$ since $E[\rho]<\infty$. Therefore we see that $(W*\Psi_\alpha)(\bx-\by)$ converges to $W(\bx-\by)$ almost everywhere with respect to the measure $\rho(\by)\rd{\by}\rho(\bx)\rd{\bx}$.

In both cases, since we proved that $(W*\Psi_\alpha)(\bx-\by)\le C_R W(\bx-\by)$ for any $\bx\ne \by$, the integral in \eqref{rhopsi} is dominated by the integral  $$\frac{1}{2}\int_{\Rd}\int_{\Rd}C_R W(\bx-\by)\rho(\by)\rd{\by}\rho(\bx)\rd{\bx} = C_R E[\rho],
$$ 
and the dominated convergence theorem gives the conclusion.
\end{proof}

\begin{proof}[Proof of Theorem \ref{thm_W2g}]
Notice that the global minimizer is unique once we fix the center of mass according to our discussion above. Let $\rho_\infty\in\Ptwo$ be the global energy minimizer with the same center of mass, and assume on the contrary that $\rho\ne \rho_\infty$.  Define $\rho_1=\rho_\infty*\psi_\alpha$ with $\psi_\alpha$ as defined in Lemma \ref{lem_psi}, and 
\begin{equation}
\rho_t = (1-t)\rho + t \rho_1
\end{equation}
as the linear interpolation curve, satisfying $\frac{\rd^2}{\rd{t}^2}E[\rho_t] \ge 0,\,\forall t\in (0,1)$. Notice that $E[\rho_\infty] < E[\rho]$ since $\rho\ne \rho_\infty$ taking into account the uniqueness of global minimizer. By Lemma \ref{lem_psi}, $E[\rho_1]<E[\rho]$ if $\alpha$ is sufficiently small. Then we see that
\begin{equation}
\frac{\rd}{\rd{t}}\Big|_{t=0}E[\rho_t] < 0\,.
\end{equation}
On the other hand, notice that by the bi-linearity of $E$ and the necessary condition of the $d_2$-local minimizers in \eqref{W2cond},
\begin{equation}
\frac{\rd}{\rd{t}}\Big|_{t=0}E[\rho_t] = \int_\Rd V(\bx)(\rho_1(\bx)-\rho(\bx))\rd{\bx} = \int_\Rd V(\bx)\rho_1(\bx)\rd{\bx} - C_\rho m_0 \ge \int_\Rd C_\rho\rho_1(\bx)\rd{\bx} - C_\rho m_0 = 0
\end{equation}
where the inequality uses the fact that $\rho_1=\rho_\infty*\psi_\alpha$ is a smooth function and then the possible zero-measure exceptional set in \eqref{W2cond} does not contribute. This gives a contradiction.
\end{proof}

We now focus on defining properly the concept of steady state for the interaction energy $E$ and the aggregation equation \eqref{eq} we are dealing with. We will denote by $\partial V(x)$ the subdifferential of $V$ at the point $x\in\Rd$.

\begin{definition}\label{steadystate}
We say $\rho\in\Pd$ is a steady state of the interaction energy $E$ if $0\in \partial V (\bx)$ $\rho$-a.e.
\end{definition}

\begin{remark}\label{rmkss}
Assume that $\rho\in \Pd$ is a steady state with $V(\bx)\in C^1(\Rd)$, then $\bu=\nabla W\ast \rho$ is continuous on $\Rd$ and $\bu=0$ on $\supp\rho$. Moreover, $\rho$ is a stationary distributional solution to \eqref{eq}. 

Notice also that any $d_\infty$-local minimizer is a steady state of the interaction energy $E$.
\end{remark}


\section{Radial symmetry of $d_\infty$-local minimizers and steady states}

In this section, for the locally stable steady states (in the sense of Remark \ref{rmkss}) of the aggregation equation \eqref{eq}, we give sufficient conditions for their radial symmetry.  

\begin{theorem}\label{thm_Wirs}
Assume $d\ge 2$ and $W$ satisfies \eqref{eq:hyp}, $W$ is radially symmetric, and the interaction energy $E$ is LIC. Then every compactly supported $d_\infty$-local minimizer $\rho\in\Pd$ is radially symmetric. 

Furthermore, suppose $E$ is FLIC and $\rho\in\Pd$ is a compactly supported steady state of the interaction energy $E$ such that $\nabla^2 W*\rho$ is continuous and $\nabla^2 W*\rho\ge 0$ (i.e., being semi-positive-definite) on $\supp\rho$, then $\rho$ is radially symmetric.
\end{theorem}

This theorem is proven in several steps, and we first focus on the 2D case and prove by contradiction. We find better competitors in case of asymmetry for local minimizers or quantifying the behavior of the hessian at the boundary of the support for steady states. The multi-D cases can be similarly done by choosing a direction along which $\rho$ is not rotationally symmetric to reach contradiction, but it needs some further technical details. Without loss of generality, we may assume due to translational invariance that a given $d_\infty$-local minimizer / steady state has zero center of mass $\int_\Rd \bx\rho(\bx)\rd{\bx}=0$. Let $\cR_\theta f$ denote the rotation of a function $f$ by the angle $\theta$:
\begin{equation}
(\cR_\theta f)(\bx) := f(R_\theta^{-1}\bx),\quad R_\theta = \begin{pmatrix}
\cos\theta & -\sin\theta \\
\sin\theta & \cos\theta 
\end{pmatrix}\,.
\end{equation}

\begin{proof}[Proof of Theorem \ref{thm_Wirs}, local minimizer case, for $d=2$]
Assume on the contrary that a compactly supported $d_\infty$-local minimizer $\rho$ is not radially symmetric. Define 
\begin{equation}
\rho_\theta := \frac{1}{2}(\rho + \cR_\theta \rho)\,.
\end{equation}
By the radial symmetry of $W$, we have $E[\rho] = E[\cR_\theta \rho]$. Then by LIC, we see that $E[\rho_\theta]<E[\rho]$ since $\rho\ne \cR_\theta\rho$ if $|\theta|$ is small enough. On the other hand, notice that $d_\infty(\rho,\rho_\theta) \le R\theta$ where $R=\max_{\bx\in\supp\rho}|\bx|<\infty$. Therefore, by the definition of $d_\infty$-local minimizer, we have $E[\rho_\theta]\ge E[\rho]$ for $|\theta|$ small enough, leading to a contradiction.
\end{proof}

To prove the statement on steady states, we first give a lower bound of the linear interpolation convexity.
\begin{lemma}\label{lem_rot}
Assume $d=2$ and the interaction energy $E$ is FLIC. Assume $\rho\in\Pd$ with zero center of mass is not radially-symmetric. Then for small $\theta>0$,
\begin{equation}
E[\rho-\cR_\theta \rho] \ge c\theta^2\,.
\end{equation}
\end{lemma}

To prove this, we first give a lemma:
\begin{lemma}\label{lem_rot2}
Let $f(x)$ be a function defined on the torus $\mathbb{T}$. Then for $|\theta| \le c_f$ being small,
\begin{equation}
\|f(\cdot) - f(\cdot-\theta)\|_{L^2_{\mathbb{T}}}^2 \ge c\theta^2 \|f-\bar{f}\|_{L^2_{\mathbb{T}}}^2,\quad \bar{f} = \frac{1}{|\mathbb{T}|}\int_{\mathbb{T}}f(x)\rd{x}
\end{equation}
where $c$ is an absolute constant.
\end{lemma}

\begin{proof}
Notice that by Fourier series expansion, we can write
\begin{equation}
\|f(\cdot) - f(\cdot-\theta)\|_{L^2}^2 = \sum_{k\in\mathbb{Z},\,k\ne 0} |\hat{f}(k)(1-e^{ik\theta})|^2
\end{equation}
and
\begin{equation}
\|f-\bar{f}\|_{L^2}^2 = \sum_{k\in\mathbb{Z},\,k\ne 0} |\hat{f}(k)|^2 \,.
\end{equation}
Therefore, there exists $K=K_f$ such that
\begin{equation}
\sum_{k\in\mathbb{Z},\,k\ne 0, |k|\le K} |\hat{f}(k)|^2 \ge \frac{1}{2}\|f-\bar{f}\|_{L^2}^2\,.
\end{equation}
We observe that
$
|1-e^{ik\theta}|^2 \ge \sin^2(k\theta) \ge c\theta^2$, for all $|k|\le K$, if $|\theta|\le 0.1/K$. Then the conclusion follows.

\end{proof}

\begin{proof}[Proof of Lemma \ref{lem_rot}]
Due to the FLIC assumption \eqref{flic} of the interaction energy, it suffices to show that
\begin{equation}
\int_{R_1\le |\xi|\le R_2} |\hat{\rho}-\cR_{-\theta}\hat{\rho}|^2\rd{\xi} \ge c\theta^2
\end{equation}
for some $0<R_1<R_2$. Since $\rho$ is not radially symmetric, we have the same property for $\hat{\rho}$, which further implies
\begin{equation}
\int_0^\infty \|f(r,\cdot)-\bar{f}(r)\|_{L^2_{\mathbb{T}}}^2 r\rd{r} > 0,
\end{equation}
where we denote $f(r,\phi) = \hat{\rho}(r(\cos\phi,\sin\phi)^T)$, $\phi \in [0,2\pi]$, and $\bar{f}(r)=\frac{1}{2\pi}\int_{\mathbb{T}}f(r,\phi)\rd{\phi}$ is the angular average of $f$. Therefore, there exists $\epsilon>0$ such that
\begin{equation}
S:= \{r: \|f(r,\phi)-\bar{f}(r)\|_{L^2_{\mathbb{T}}}^2 > \epsilon\}
\end{equation}
has positive measure. We may assume $S\subseteq [R_1,R_2]$ for some $0<R_1<R_2$ by replacing with a subset of itself if necessary. 
For each $r\in S$, let $c_r$ denote the constant $c_f$ in Lemma \ref{lem_rot2} with $f = f(r,\cdot)$. We can express the set $S$ as a union of nested sets
\begin{equation}
S = \bigcup_{\delta>0} S_\delta,\quad S_\delta:=\{r\in S: c_r > \delta\}
\end{equation}
for any $\delta >0$, concluding that $|S_\delta|>0$ for some $\delta>0$. Therefore, using Lemma \ref{lem_rot2}, we infer that
\begin{equation}
\|f(r,\cdot)-f(r,\cdot-\theta)\|_{L^2_{\mathbb{T}}}^2 \ge c\theta^2 \|f(r,\cdot)-\bar{f}(r)\|_{L^2_{\mathbb{T}}}^2  \ge c\epsilon \theta^2,\quad \forall |\theta|\le \delta, r \in S_\delta\,,
\end{equation}
where $c$ is an absolute constant. Therefore, we finally obtain
\begin{equation}
\int_{R_1\le |\xi|\le R_2} |\hat{\rho}-\cR_{-\theta}\hat{\rho}|^2\rd{\xi} \ge 2\pi\int_{S_\delta} \|f(r,\cdot)-f(r,\cdot+\theta)\|_{L^2_{\mathbb{T}}}^2 r \rd{r} \ge c\epsilon |S_\delta| R_1 \theta^2
\end{equation}
 Then the conclusion follows.
\end{proof}

\begin{proof}[Proof of Theorem \ref{thm_Wirs}, steady state case, for $d=2$]

Assume on the contrary that $\rho$ is a steady state satisfying the assumptions but not radially symmetric. Define 
\begin{equation}
\rho_\theta := \frac{1}{2}(\rho + \cR_\theta \rho)\,.
\end{equation}
Then Lemma \ref{lem_rot} shows that for small $|\theta|$,
\begin{equation}\label{strictlin}
E[\rho_\theta] \le E[\rho] - c\theta^2
\end{equation}
by noticing that $\frac{\rd^2}{\rd{t}^2} E[(1-t)\rho+t \cR_\theta \rho] = 2E[\rho-\cR_\theta \rho]$.

Since $\rho$ is a steady state, $\nabla V=0$ on $\supp \rho$ by definition.   By the continuity of $\nabla^2 V$ and its semi-positive-definiteness on $\supp\rho$, for any $\epsilon>0$, there exists $\delta>0$ such that
\begin{equation}
\frac{1}{2}\nabla^2 V(\bx) \ge -\epsilon I_2,\quad \forall \bx \text{ with } \dist(\bx,\supp\rho)\le \delta\,.
\end{equation}
Here, we have used the classical notation for order of square matrices with $I_d$ being the identity matrix in dimension $d$. Therefore, if $\bx_1\in \supp\rho$ and $|\bx_2-\bx_1|\le \delta$, then
\begin{equation}
V(\bx_2) \ge V(\bx_1)-\epsilon |\bx_2-\bx_1|^2
\end{equation}
by Taylor expansion.

Notice that 
$$
E[\rho_\theta] = \frac{1}{2}E[\rho] + \frac{1}{4}\int_{\Rd} (\cR_\theta \rho)(\bx) V(\bx)\rd{\bx} \qquad \mbox{and} \qquad E[\rho] = \frac{1}{2}\int_{\Rd} \rho(\bx) V(\bx)\rd{\bx}
$$  
due to rotational symmetry of $W$. Then, we have
\begin{equation}\begin{split}
\int_{\Rd} (\cR_\theta \rho)(\bx) V(\bx)\rd{\bx} - 2E[\rho] = & \int_{\Rd} \Big((\cR_\theta \rho)(\bx)-\rho(\bx)\Big) V(\bx)\rd{\bx}= \int_{\Rd} \Big( \rho(R_\theta^{-1}\bx)-\rho(\bx)\Big) V(\bx)\rd{\bx} \\
= & \int_{\Rd} \rho(\bx) \Big(V(R_\theta \bx)-V(\bx)\Big)\rd{\bx} \ge -\epsilon (R\theta)^2 \int_{\Rd} \rho(\bx)\rd{\bx} = -\epsilon (R\theta)^2 \,,
\end{split}\end{equation}
if $R\theta<\delta$, where $R=\max_{\bx\in\supp\rho}|\bx|<\infty$. Taking $\epsilon=\frac{c}{R^2}$ where $c$ is given by \eqref{strictlin}, we conclude that
\begin{equation}
E[\rho_\theta] \ge E[\rho] - \frac{c}{4}\theta^2
\end{equation}
for $\theta$ small enough, contradicting \eqref{strictlin}.
\end{proof}

We now go back to the general $d\ge 3$ case. Denote $\bx = (x_1,x_2,\by)^T$. Since $\rho$ is not radially-symmetric, we may assume without loss of generality that the center of mass of $\rho$ is 0, and
\begin{equation}
\rho\ne \tilde{\cR}_\theta\rho,\quad \tilde{\cR}_\theta\rho(x_1,x_2,\by) := \rho(R_\theta^{-1}(x_1,x_2)^T,\by)
\end{equation}
for small $|\theta|$. The local minimizer case can be done similarly as above using $\rho_\theta = \frac{1}{2}(\rho+ \tilde{\cR}_\theta\rho)$ instead. To treat the steady state case, we next generalize Lemma \ref{lem_rot}.

\begin{lemma}\label{lem_rot3}
Assume $d\ge 3$ and the interaction energy $E$ is FLIC. Assume $\rho\in\Pd$ with zero center of mass satisfies $\rho\ne \tilde{\cR}_\theta\rho$. Then for small $\theta>0$, $E[\rho-\tilde{\cR}_\theta \rho] \ge c\theta^2$ for some $c>0$.
\end{lemma}
\begin{proof}

Due to the FLIC property of $E$, it suffices to show that
\begin{equation}
\int_{R_1\le |\xi|\le R_2} |\hat{\rho}-\cR_{-\theta}\hat{\rho}|^2\rd{\xi} \ge c\theta^2
\end{equation}
for some $0<R_1<R_2$.
The condition $\rho\ne \tilde{\cR}_\theta\rho$ implies the same property of $\hat{\rho}$, which implies
\begin{equation}
\int_0^\infty\int_{\mathbb{R}^{d-2}} \|f(r,\cdot,\eta)-\bar{f}(r,\eta)\|_{L^2_{\mathbb{T}}}^2 \rd{\eta}r\rd{r} > 0
\end{equation}
where we denote $f(r,\phi,\eta) = \hat{\rho}(r(\cos\phi,\sin\phi)^T,\eta)$, $\phi\in [0,2\pi]$ and $\bar{f}(r,\eta) = \frac{1}{2\pi}\int_{\mathbb{T}}f(r,\phi,\eta)\rd{\phi} $. Therefore, there exists $\epsilon>0$ such that
\begin{equation}
S:= \{(r,\eta): \|f(r,\phi,\eta)-\bar{f}(r,\eta)\|_{L^2_{\mathbb{T}}}^2 > \epsilon\}
\end{equation}
has positive measure. We may assume $S\subseteq \{(r,\eta):\sqrt{r^2+\eta^2}\in[R_1,R_2]\}$ for some $0<R_1<R_2$ by replacing with a subset of itself if necessary. 
For each $(r,\eta)\in S$, let $c_{(r,\eta)}$ denote the constant $c_f$ in Lemma \ref{lem_rot2} with $f = f(r,\cdot,\eta)$. Similarly as in the 2D case, we write $S$ as
\begin{equation}
S = \bigcup_{\delta>0} S_\delta,\quad S_\delta:=\{(r,\eta)\in S: c_{(r,\eta)} > \delta\}
\end{equation}
to conclude that $|S_\delta|>0$ for some $\delta>0$. Therefore, we obtain
\begin{equation}
\|f(r,\cdot,\eta)-f(r,\cdot-\theta,\eta)\|_{L^2_{\mathbb{T}}}^2 \ge c\theta^2 \|f(r,\cdot,\eta)-\bar{f}(r,\eta)\|_{L^2_{\mathbb{T}}}^2  \ge c\epsilon \theta^2,\quad \forall |\theta|\le \delta, (r,\eta) \in S_\delta\,,
\end{equation}
where $c$ is an absolute constant. As a consequence, we get
\begin{equation}
\int_{R_1\le |\xi|\le R_2} |\hat{\rho}-\tilde{\cR}_{-\theta}\hat{\rho}|^2\rd{\xi} \ge \int_{S_\delta} \int_{\mathbb{R}^{d-2}}\|f(r,\cdot,\eta)-f(r,\cdot+\theta,\eta)\|_{L^2_{\mathbb{T}}}^2 \rd{\eta}\,2\pi r \rd{r} \ge c\epsilon |S_\delta| R_1 \theta^2 \,.
\end{equation}
Then the conclusion follows.
\end{proof}

\begin{proof}[Proof of Theorem \ref{thm_Wirs}, for $d\ge 3$]
Similarly as above, we define
\begin{equation}
\rho_\theta := \frac{1}{2}(\rho + \tilde{\cR}_\theta \rho)
\end{equation}
Then Lemma \ref{lem_rot3} shows that for small $|\theta|$,
$E[\rho_\theta] \le E[\rho] - c\theta^2$,
by noticing that $\frac{\rd^2}{\rd{t}^2} E[(1-t)\rho+t \cR_\theta \rho] = 2E[\rho-\tilde{\cR}_\theta \rho]$. Similar to the $d=2$ case, we can use the steady state properties and the assumption on $\nabla^2 V$ to show that for any $c_1>0$,
\begin{equation}
E[\rho_\theta] \ge E[\rho] - \frac{c}{4}\theta^2
\end{equation}
if $|\theta|$ is small enough, leading to the contradiction.
\end{proof}

\begin{remark}[Failure of uniqueness and radial symmetry for $d_\infty$-local minimizers of 1D FLIC potentials]
For FLIC potentials, the uniqueness of global minimizer clearly implies its radial symmetry in any dimension. However, for $d_\infty$ local minimizers, although the radial symmetry is still true for $d\ge 2$, it is generally false for $d=1$. We provide an example of a $d_\infty$-local minimizer of a 1D FLIC potential that may fail to be unique and  radially-symmetric (i.e., in 1D, an even function). Define
\begin{equation}
W_\epsilon(x) = -\frac{x^2}{2}+\frac{|x|^3}{3} - \epsilon \Big(\phi(|x-1|)+\phi(|x+1|)\Big)
\end{equation}
where $\epsilon\ge 0$ and $\phi$ is a non-negative smooth even function supported on $[-1/2,1/2]$ with $\phi''(0)<0$. The result in \cite{lopes2017uniqueness} shows that the interaction energy associated to the potential $W_0$ is FLIC with 
\begin{equation}
\hat{W}_0(\xi) = c|\xi|^{-4},\quad \forall \xi\ne 0\,.
\end{equation}
Since $|\hat{\phi}(\xi)| \le C(1+|\xi|)^{-5}$, there holds $\hat{W}_\epsilon(\xi)>0,\,\forall \xi\ne 0$ if $\epsilon>0$ is small enough, and then $W_\epsilon$ is FLIC. 

Notice now that $W_\epsilon(x)$ has local minima at $x=\pm1$, with
\begin{equation}
W_\epsilon''(1) = W_\epsilon''(-1) = 1-\epsilon\phi''(0)=:\lambda > 1
\end{equation}
Also notice that $W_\epsilon''(0) = -1$.
It follows that for any $0<\alpha<\tfrac12$, $\rho_\alpha(x) = (\tfrac12-\alpha)\delta(x-\frac{1}{2}) + (\tfrac12+\alpha)\delta(x+\frac{1}{2})$ is a steady state for $W_\epsilon$, since 
\begin{equation}
(W_\epsilon*\rho_\alpha)(x) = (\tfrac12-\alpha)W_\epsilon(x-\tfrac{1}{2}) +(\tfrac12+\alpha) W_\epsilon(x+\tfrac{1}{2})
\end{equation}
satisfies $(W_\epsilon*\rho_\alpha)'(\tfrac{1}{2})=(W_\epsilon*\rho_\alpha)'(-\tfrac{1}{2})=0$. Also, since
\begin{equation}
(W_\epsilon*\rho_\alpha)''(\tfrac{1}{2}) = -(\tfrac12-\alpha)+(\tfrac12+\alpha)\lambda,\quad  (W_\epsilon*\rho_\alpha)''(-\tfrac{1}{2}) = (\tfrac12-\alpha)\lambda-(\tfrac12+\alpha)\,,
\end{equation}
we observe that both quantities are positive by taking $|\alpha|<\frac{\lambda-1}{2(\lambda+1)}$. This implies $\rho_\alpha$ is a $d_\infty$-local minimizer for every such $\alpha$ (which will be justified in the next paragraph), and this shows that the $d_\infty$-local minimizers of $W_\epsilon$ are non-unique and non-radially-symmetric in general.

To see that $\rho_\alpha$ is a $d_\infty$-local minimizer, we consider any alternate $\rho\ne \rho_\alpha$ with the same total mass and center of mass and $\beta:=d_\infty(\rho,\rho_\alpha)$ being small. Then $\supp\rho\subseteq [-\tfrac12-\beta,-\tfrac12+\beta]\cup [\tfrac12-\beta,\tfrac12+\beta]$ with 
$$
\int_{[-\tfrac12-\beta,-\tfrac12+\beta]} \rho\rd{x} = \tfrac12+\alpha\qquad \mbox{and} \qquad \int_{[\tfrac12-\beta,\tfrac12+\beta]} \rho\rd{x} = \tfrac12-\alpha. 
$$
Define $\rho_t = (1-t)\rho_\alpha+t\rho$, and we have $\frac{\rd^2}{\rd{t}^2}E[\rho_t] > 0$ for any $0\le t \le 1$ by the LIC property of $E$ (with the interaction potential $W_\alpha$). Denoting $V_\alpha:=W_\epsilon*\rho_\alpha$, we have
\begin{equation}\begin{split}
\frac{\rd}{\rd{t}} & \Big|_{t=0}E[\rho_t] =  \int_\R V_\alpha(x)(\rho(x)-\rho_\alpha(x))\rd{x}\\
= & \int_{[-\tfrac12-\beta,-\tfrac12+\beta]} V_\alpha(x)\big(\rho(x)-(\tfrac12+\alpha)\delta(x+\tfrac{1}{2})\big)\rd{x} + \int_{[\tfrac12-\beta,\tfrac12+\beta]} V_\alpha(x)\big(\rho(x)-(\tfrac12-\alpha)\delta(x-\tfrac{1}{2})\big)\rd{x}
\end{split}\end{equation}
Since $V_\alpha\in C^2$ has vanishing first derivative and positive second derivative at $\pm\tfrac12$, we have $V_\alpha(x)\ge V_\alpha(\pm\tfrac12)$ for $|x-(\pm\tfrac12)|\le\beta$ if $\beta>0$ is small enough. Then we see that the first integral in the last expression is non-negative since 
$$
\int_{[-\tfrac12-\beta,-\tfrac12+\beta]} \rho\rd{x} = \tfrac12+\alpha,
$$
and similar for the second integral. Therefore we get $\frac{\rd}{\rd{t}}|_{t=0}E[\rho_t]\ge 0$, which implies $\frac{\rd}{\rd{t}}E[\rho_t]\ge 0$ for any $0\le t \le 1$ since $\frac{\rd^2}{\rd{t}^2}E[\rho_t] > 0$. Therefore 
\begin{equation}
E[\rho]-E[\rho_\alpha]=\int_0^1\frac{\rd}{\rd{t}}E[\rho_t]\rd{t} \ge 0
\end{equation}
which shows that $\rho_\alpha$ is a $d_\infty$-local minimizer.
\end{remark}


\section{From radially-symmetric steady states to uniqueness of $d_\infty$-local minimizer}

We first show our main result concerning the uniqueness of minimizers as a consequence of their radial symmetry.
\begin{theorem}\label{thm_min0}
Assume $W$ satisfies \eqref{eq:hyp2}, $W\in C^4(\Rd\setminus\{0\})$ is radially symmetric, the interaction energy \eqref{energy} associated to the potential $W$ is LIC and 
\begin{equation}\label{thm_min_1}
\Delta^2 W(\bx) < 0,\quad \forall \bx \ne 0
\end{equation}
and
\begin{equation}\label{thm_min_2}
 \Delta W(\bx)>0,\quad \forall |\bx| \textnormal{ sufficiently large.}
\end{equation}
Let $\rho$ be a compactly supported $d_\infty$-local minimizer such that $\nabla W * \rho$ is continuous with zero center of mass. Then $\rho$ is the unique global minimizer of $E$ over $\Ptwo$ with zero center of mass.
\end{theorem}


We want to take advantage on the fourth derivative assumption \eqref{thm_min_1} in order to apply maximum principle arguments for certain operators. We need some preliminary results in this direction.

\begin{lemma}\label{lem_d4}
Let $f\in C^1([x_1,x_2])\cap C^4((x_1,x_2))$. Assume
\begin{equation}
f'(x_1)=f'(x_2)=0,\quad (\cL^2f)(x)<0,\,\forall x\in (x_1,x_2),\quad \cL f := f'' + a(x)f'
\end{equation}
for some positive function $a(x)\in C((x_1,x_2])$. Then either $x_1$ or $x_2$ is not a local minimum point of $f$.
\end{lemma}
\begin{proof}
Denote $g = \cL f$. By assumption $(\cL g)(x)<0,\,\forall x\in (x_1,x_2)$. This implies that there exist at most 2 points in $(x_1,x_2)$ where $g=0$. In fact, suppose $g(y_1)=g(y_2)=g(y_3)=0$ with $y_1<y_2<y_3$, then $(\cL g)(x)<0$ forces $g$ to be positive on $(y_1,y_2)$ and $(y_2,y_3)$ due to the classical maximum principle, and it follows that $g'(y_2)=0$, $g''(y_2)\ge 0$ which is a contradiction to $(\cL g)(y_2)<0$. 

Denote $A(x) = - \int_x^{x_2}a(y)\rd{y}$ which is a non-positive increasing continuous function defined on $(x_1,x_2]$, satisfying $A'(x)=a(x)$. Then notice that
\begin{equation}
\int_{x_1}^{x_2}e^{A(x)}(f''+a(x)f')\rd{x} = \int_{x_1}^{x_2}(e^{A(x)}f')'\rd{x} = e^{A(x_2)}f'(x_2) - \lim_{x\rightarrow x_1+}e^{A(x)}f'(x_1) = 0
\end{equation}
where the integral is interpreted as an improper integral, and the existence of the last limit follows from the assumption $f'(x_1)=0$ and the fact that $A(x)$ is negative and increasing for $x<x_2$. This requires that $g=f''+a(x)f'$ is positive at some point in $(x_1,x_2)$, and $g$ is negative at some other point in $(x_1,x_2)$, and therefore  there exists at least 1 point in $(x_1,x_2)$ where $g=0$.

Now we separate into the following cases:
\begin{itemize}
\item There exists 1 point $y$ in $(x_1,x_2)$ such that $g(y)=0$, and $g|_{(x_1,y)}>0$, $g|_{(y,x_2)}<0$. In this case, 
\begin{equation}
e^{A(x_2-\epsilon)}f'(x_2-\epsilon) = e^{A(x_2)}f'(x_2)-\int_{x_2-\epsilon}^{x_2}e^{A(x)}g(x)\rd{x} > 0
\end{equation}
if $0<\epsilon<x_2-y$. This implies that $x_2$ is not a local minimum point of $f$.
\item There exists 1 point $y$ in $(x_1,x_2)$ such that $g(y)=0$, and $g|_{(x_1,y)}<0$, $g|_{(y,x_2)}>0$. In this case, 
\begin{equation}
e^{A(x_1+\epsilon)}f'(x_1+\epsilon) = \lim_{x\rightarrow x_1+}e^{A(x)}f'(x_1)+\int_{x_1}^{x_1+\epsilon}e^{A(x)}g(x)\rd{x} = \int_{x_1}^{x_1+\epsilon}e^{A(x)}g(x)\rd{x} < 0
\end{equation}
if $0<\epsilon<y-x_1$. This implies that $x_1$ is not a local minimum point of $f$.
\item There exist 2 points $z_1<z_2$ in $(x_1,x_2)$ such that $g(z_1)=g(z_2)=0$. Then $(\cL g)(x)<0$ implies that $g|_{(x_1,z_1)}<0$, $g|_{(z_1,z_2)}>0$, $g|_{(z_2,x_2)}<0$ since otherwise one gets a contradiction at $z_1$ or $z_2$ similarly as done above for $y_2$. Then it follows as in the previous two cases that neither $x_1$ nor $x_2$ is a local minimum point of $f$.
\end{itemize}

\end{proof}

\begin{lemma}\label{lem_d42}
Let $f\in C^1([x_1,\infty))\cap C^4((x_1,\infty))$. Assume
\begin{equation}
f'(x_1)=0,\quad (\cL^2f)(x)<0,\,\forall x\in (x_1,x_2),\quad \cL f := f'' + a(x)f',\quad (\cL f)(x_2)>0
\end{equation}
for some $x_2>x_1$ and some positive function $a(x)\in C\big((x_1,x_2]\big)$. Then either $x_1$ is not a local minimum point of $f$, or $x_1$ is the global minimum point of $f$ on $[x_1,x_2]$.
\end{lemma}
\begin{proof}
Similar to the previous proof, there exists at most 2 points in $(x_1,x_2)$ where $g=\cL f = 0$. We claim that in fact there exists at most one point with the additional assumption $g(x_2)>0$. In fact, suppose $g(y_1)=g(y_2)=0$ with $y_1<y_2$, then $(\cL g)(x)<0$ forces $g$ to be positive on $(y_1,y_2)$ due to the classical maximum principle. Moreover, $g$ is also positive on $(y_2,x_2)$ since there are no other zeros of $g$ and $g(x_2)>0$, and it follows as above that $g'(y_2)=0$, $g''(y_2)\ge 0$ which is a contradiction to $(\cL g)(y_2)<0$. Now, we may separate into the following cases:
\begin{itemize}
\item There exists 1 point $y$ in $(x_1,x_2)$ such that $g(y)=0$, and $g|_{(x_1,y)}<0$, $g|_{(y,x_2)}>0$. Notice there is a change of sign on $g$ at the zero and $g(x_2)>0$, otherwise we arrive at a contradiction again. In this case, $x_1$ is not a local minimum point of $f$ proceeding as in the first two cases of Lemma \ref{lem_d4}. 
\item Assume now that $g|_{(x_1,x_2)}>0$. In this case
\begin{equation}
e^{A(x)}f'(x) = \lim_{x\rightarrow x_1+}e^{A(x)}f'(x_1) + \int_{x_1}^x e^{A(z)}g(z)\rd{z} = \int_{x_1}^x e^{A(z)}g(z)\rd{z}> 0,\quad \forall x\in (x_1,x_2)
\end{equation}
which implies that $f$ is increasing on $[x_1,x_2]$. Therefore $x_1$ is a global minimum point of $f$ on $[x_1,x_2]$.
\end{itemize}
\end{proof}

\begin{proof}[Proof of Theorem \ref{thm_min0}]
Assume first that $d\geq 2$. By Theorem \ref{thm_Wirs}, $\rho$ is radially-symmetric. 

Suppose there exists a subset of $(\supp\rho)^c$,
$S_1 = \{R_1 < |\bx| < R_2\}$
for some $0\le R_1<R_2$, with $\{|\bx|=R_1\}\subseteq \supp\rho$ and $\{|\bx|=R_2\}\subseteq \supp\rho$. Denote $(\cL V)(r) = V''(r) + \frac{d-1}{r}V'(r)$ as an operator on the radial direction, then $\Delta V(\bx) = (\cL V)(r)$, $\Delta^2 V(\bx) = (\cL^2 V)(r) < 0$ for any $R_1<r<R_2$. Since $\rho$ is a $d_\infty$-local minimizer, we have $V'(R_1)=V'(R_2)=0$ by Lemma \ref{lemel1} and the assumption that $V\in C^1$. Therefore Lemma \ref{lem_d4} shows that either $R_1$ or $R_2$ is not a local minimum point of $V(r)$, which contradicts Lemma \ref{lemel1}.

Suppose there exists a subset of $(\supp\rho)^c$,
$
S_2 = \{|\bx| < R_2\} 
$
for some $R_2>0$, with $\{|\bx|=R_2\}\subseteq \supp\rho$. Then $\Delta^2 V(\bx) = (\cL^2 V)(r) < 0$ for any $0<r<R_2$. Since $\rho$ is a $d_\infty$-local minimizer, we have $V'(R_2)=0$ by Lemma \ref{lemel1} and the assumption that $V\in C^1$, and $V'(0)=0$ by the radial symmetry of $V$. Therefore Lemma \ref{lem_d4} shows that either $0$ or $R_2$ is not a local minimum point of $V(r)$. By Lemma \ref{lemel1}, $R_2\in \supp V(r)$ is a local minimum point of $V(r)$, and thus $0$  is not a local minimum point of $V(r)$, which implies $\Delta V(0)\le 0$. Since $\Delta V$ is radially symmetric and $\Delta^2 V(\bx)<0$ for $|\bx|<R_2$, we have $\Delta V(\bx)<0$ for any $0<|\bx|<R_2$. This contradicts the fact that $V'(R_2)=0$ by integrating on $S_2$.

Therefore $\supp\rho$ is a ball, and denote its radius as $R_1\ge 0$. Since $\Delta V(\bx) = (\cL V)(r)>0$ for all $r$ large enough (due to the assumption \eqref{thm_min_2}), we apply Lemma \ref{lem_d42} on $[R_1,R_2]$ for large $R_2$, and obtain that $R_1$ is a global minimum of $V(r)$ on any $[R_1,R_2]$. Then the conclusion follows from Theorem \ref{thm_W2g}.

For the case $d=1$, one could argue similarly and show that there cannot be an interval $(x_1,x_2)\subseteq\supp\rho$ with $\{x_1,x_2\}\subseteq\supp\rho$, and conclude that $\supp\rho$ is an interval $[X_1,X_2]$. Then one could obtain similarly that $X_2$ is the global minimum point of $V$ on $[X_2,\infty)$, and $X_1$ is the global minimum point of $V$ on $(-\infty,X_1]$. Then Theorem \ref{thm_W2g} shows that $\rho$ is the unique global minimizer of $E$. The radial symmetry of $\rho$ is obtained by noticing that $\rho(-x)$ is also a global minimizer with zero center of mass, and therefore equal to $\rho(x)$.
\end{proof}

\begin{remark}\label{thm_min0_rem}
In the previous proof, if $d\ge 2$ and suppose we know that $\rho$ does not have an isolated Dirac mass at $0$, then in the case of $S_1$ we also can exclude the possibility of $R_1=0$, and therefore we may weaken the continuity assumption to `$\nabla W*\rho$ is continuous on $\Rd\backslash\{0\}$'.
\end{remark}

In case the interaction energy does not necessarily meet the FLIC property, we can obtain a similar result for stationary states under an additional regularity assumption: $\nabla W*\rho$ and $\Delta W*\rho$ are continuous.

\begin{theorem}\label{thm_min}
Assume $W$ satisfies \eqref{eq:hyp}, $W\in C^4(\Rd\setminus\{0\})$ is radially symmetric with
\begin{equation}
\Delta^2 W(\bx) < 0,\quad \forall \bx \ne 0
\end{equation}
and
\begin{equation}
 \Delta W(\bx)>0,\quad \forall |\bx| \textnormal{ sufficiently large.}
\end{equation}
Let $\rho$ be a compactly supported steady state in the sense of Definition \ref{steadystate}, with $\nabla W*\rho$ and $\Delta W*\rho$ being continuous, and $\Delta W*\rho\ge0$ on $\supp\rho$. Then the complement of $\supp\rho$ is connected.
If in addition, $\rho$ is radially-symmetric, or $d=1$, then $\rho$ satisfies the $d_2$-local minimizer condition \eqref{W2cond}.

If in addition, $W$ satisfies FLIC, then any compactly supported steady state $\rho$, with the regularity assumption that $\nabla^2 W*\rho$ is continuous and semi-positive definite on $\supp{\rho}$, is the global minimizer of the interaction energy.
\end{theorem}

We first remark that the last part of the previous theorem is a direct consequence of Theorem \ref{thm_Wirs} and Theorem \ref{thm_W2g} together with the first part of it. To prove this result, we need the following lemmas, which are standard maximum principle arguments in elliptic theory.

\begin{lemma}\label{lem_el1}
Let $\Omega$ be a bounded open set, and $f\in C(\bar{\Omega})\cap C^2(\Omega)$. If $\Delta f\le 0$ on $\Omega$, then $\min_{\bx\in\bar{\Omega}}f(\bx) = \min_{\bx\in\partial\Omega}f(\bx)$.
Assume further $f\in C^1(\bar{\Omega})$ and $\Omega$ is connected, if $\Delta f\le 0$ on $\Omega$ and $\nabla f = 0$ on $\partial\Omega$, then $f$ is constant on $\bar{\Omega}$.
\end{lemma}

\begin{proof}[Proof of Theorem \ref{thm_min}]

Take $R$ large enough such that $\supp\rho \subseteq \{|\bx|<R\}$, and  by \eqref{thm_min_2}, we have
\begin{equation}\label{delV1}
\Delta V(\bx) = \int_{\Rd} \Delta W(\bx-\by)\rho(\by)\rd{\by} > 0,\quad \forall |\bx|=R
\end{equation}
for $R$ large enough. Denote $S = \{|\bx|<R\}\backslash \supp\rho$.

{\bf STEP 1}: Prove that $\Delta V \ge 0$ on $S$.
First notice that for any $\bx\in S$,
\begin{equation}\label{V4}
\Delta^2 V(\bx) = \int_{\Rd} \Delta^2 W(\bx-\by)\rho(\by)\rd{\by} < 0
\end{equation}
by \eqref{thm_min_1}. Then applying Lemma \ref{lem_el1} to $\Delta V$ which is continuous on $\bar S$ by assumption, we get
\begin{equation}
\inf_{\bx\in S} \Delta V(\bx) = \min_{\bx\in\partial S} \Delta V(\bx) \,.
\end{equation}
Notice that
$
\partial S= \partial(\supp\rho) \cup \{|\bx|=R\}$ and
\begin{equation}
\min_{\bx\in\partial(\supp\rho)} \Delta V(\bx) \ge 0\,,
\end{equation}
since $\Delta V\ge 0$ on $\supp\rho$ by assumption. Combined with \eqref{delV1}, we get
\begin{equation}
\min_{\bx\in\partial S} \Delta V(\bx) \ge 0\,,
\end{equation}
which implies the claim.

{\bf STEP 2}: Prove the connectivity of $S$.
Suppose on the contrary that $S_1$ is a connected component of $S$ whose closure does not intersect $\{|\bx|=R\}$. Then $\partial S_1\subseteq \partial(\supp\rho)$, which implies
\begin{equation}
\nabla V(\bx) = 0,\quad \forall \bx\in \partial S_1\,,
\end{equation}
by the continuity of $\nabla V$ since $\rho$ is a stationary state. Using the second part of Lemma \ref{lem_el1} since $-\Delta V\leq 0$ on $S_1$, we deduce that $V$ is constant on $S_1$. This fact contradicts \eqref{V4}.

{\bf STEP 3}: Prove the $d_2$-local minimizer condition \eqref{W2cond}. Assume that $d\geq 2$ and $\rho$ is radially symmetric. Due to the connectivity of $S$, $S$ has the form $\bar S = \{R_1 \le |\bx| \le R\}$
for some $0\leq R_1<R$. Let us denote for simplicity the radial potential generated by $\rho$ as $V(\bx)=V(r),\,r=|\bx|$, then it satisfies
\begin{equation}
V'(R_1) = 0,\quad \mbox{and} \quad \Delta V(r) \ge 0,\, r>R_1\,.
\end{equation}
This implies that $V(r)$ is increasing on $[R_1,\infty)$, which gives the desired conclusion. Indeed, to see this, notice that
\begin{equation}
\Delta V(r) = V''(r) + (d-1)\frac{V'(r)}{r} \ge 0,\, r> R_1\,.
\end{equation}
Multiplying by $r^{d-1}$ and integrating leads to $V'(r)\ge 0$. The 1D case without the radial symmetry assumption can be handled similarly.
\end{proof}


\section{Radial Symmetry \& Global minimizers for power-law potentials}

We apply the previous results to the power-law potential
\begin{equation}\label{Wab}
W(\bx) = \frac{|\bx|^a}{a} - \frac{|\bx|^b}{b}
\end{equation}
where $a>b>-d$ and with the convention $\tfrac{|x|^0}{0}:=\ln x$. 

Let $\rho\in\PP$ be compactly supported with zero center of mass. We define $\rho$ to be \emph{mild} if $d=1$ with $\nabla W*\rho$ continuous, or $d\ge 2$ with $\nabla W*\rho$ continuous on $\mathbb{R}^d\backslash \{0\}$.

\begin{theorem}\label{thm_ab}
Let $W$ be defined by \eqref{Wab}. If $(a,b)$ satisfies
\begin{equation}\label{thm_ab_1}
2\le a \le 4,\quad -d< b < 2
\end{equation}
Then any compactly supported $d_\infty$-local minimizer $\rho$ with $d\ge 2$  is radially-symmetric. If $(a,b)$ further satisfies
\begin{equation}\label{thm_ab_2}
 \left.\begin{split}
&a=2,\quad 2-d \le b < 4-d,\quad \text{when }d\ge 2 \\
&2\le a \le 3,\quad 1 \le b < 2,\quad \text{when }d=1 \\
\end{split}\right.
\end{equation}
Then the unique global minimizer $\rho_\infty$ is mild, and it is the unique mild $d_\infty$-local minimizer. 

If $a=2$ and $2-d<b<\min\{4-d,2\}$ or $a=4$ and $2-d<b<\bar b = (2+2d-d^2)/(d+1)$, then the global minimizer is given by the explicit formulas in \eqref{WCH} and \eqref{WCH2}.

\end{theorem}

\begin{remark}
In one dimension, the ranges $-1<b<2$ for $a=2$ and $2<a<3$ for $b=2$ were discussed in \cite{frank2021minimizers}.
We now discuss the sharpness of the assumption \eqref{thm_ab_2} in the case $d\ge 2$. 

It is shown in \cite{BalagueCarrilloLaurentRaoul} that for $W$ given by \eqref{Wab} with $a\ge 2$ and $\frac{(3-d)a-10+7d-d^2}{a+d-3}=:b_{\max}<b<2$, then the Dirac Delta on a particular $(d-1)$-dimensional sphere is a $d_\infty$-local minimizer. Notice that $b_{\max}<4-d$ if $a>2$. Therefore the assumption $a=2$ in \eqref{thm_ab_2} cannot be improved within our framework, in which a necessary step to get the uniqueness of $d_\infty$-local minimizer is to show the connectivity of $(\supp\rho)^c$ for any $d_\infty$-local minimizer $\rho$, see Theorem \ref{thm_min}. 

In the case $d\ge 3$, the same reasoning also shows that the upper bound $4-d$ for $b$ cannot be improved within our framework, because for $a=2$ and $4-d=b_{\max}<b<2$ the Dirac Delta on a sphere is a $d_\infty$-local minimizer.
\end{remark}

To prove this theorem, we first analyze the FLIC property of power-law potentials to justify the radial symmetry of $d_\infty$-local minimizers. Then, using the radial symmetry, we show the mild property of the global minimizer $\rho_\infty$ with the range of parameters \eqref{thm_ab_2}. In the special case $a=2$, we obtain a better regularity result, $\nabla^2 W*\rho_\infty$ is continuous, by using the explicit formulas given by \cite{CH}. Finally we show the uniqueness of mild $d_\infty$-local minimizer by applying Theorem \ref{thm_min0}.

\subsection{The (F)LIC property and radial symmetry}

We first figure out the values of $(a,b)$ such that $W$ has FLIC. 
The author in \cite{lopes2017uniqueness} proved that $\frac{|\bx|^a}{a}$ has the LIC property, for $2\le a \le 4$. It is clear, also used in \cite{lopes2017uniqueness}, that $-\frac{|\bx|^b}{b}$ has the FLIC property for $-d < b < 0$, since its Fourier transform is $c|\bx|^{-d-b}$ for some $c>0$. We next extend the range of $b$.

\begin{theorem}\label{lem_L1}
The interaction energy $E$ associated to the potential $\frac{|\bx|^a}{a}-\frac{|\bx|^b}{b}$, with $-d < b < 2\leq a\leq 4$, except $0\leq b<1$ in one dimension, has the FLIC property. 

As a consequence, any compactly supported $d_\infty$-local minimizer of the interaction energy with $-d < b < 2\leq a\leq 4$ and $d\ge 2$ is radially symmetric. In particular, the unique global minimizer for the interaction energy with $-d < b < 2\leq a\leq 4$ and $d\ge 2$ is radially symmetric.
\end{theorem}

\begin{proof}
The range $-d<b<0$ is done in \cite{lopes2017uniqueness}.
Since $W$ satisfies \eqref{eq:hyp}, then Lemma \ref{lem:reduction} allows us to reduce to the case of nonzero compactly supported signed measures $\mu\in\Md$, with 
\begin{equation}\label{lem_L1_1}
\int_{\Rd} \mu(\bx)\rd{\bx} = \int_{\Rd} \bx\mu(\bx)\rd{\bx} = 0\,.
\end{equation}
Notice that $\frac{|\bx|^a}{a}$ has the LIC property, for $2\le a \le 4$, as proven in \cite{lopes2017uniqueness}. Since $\mu$ is compactly supported, we are reduced to show the FLIC property for the repulsive part of the interaction potential, $U(\bx) = -\frac{|\bx|^b}{b}$. Then we separate into cases:

{\bf Case 1}: $2-d<b<2$.
Let 
$
f(\bx) = \Delta^{-1}\mu(\bx)
$.
Then  for $|\bx|$ large, by \eqref{lem_L1_1}, we have
\begin{equation}\label{ftail}
|f(\bx)| \le C|\bx|^{-d},\quad |\nabla f(\bx)| \le C|\bx|^{-d-1}\,.
\end{equation}
Since $b<2$, this suffices to justify the integration-by-parts below 
\begin{equation}
\int_{\Rd} \int_{\Rd} U(\bx-\by)\Delta f(\by)\rd{\by}\mu(\bx)\rd{\bx} 
=  \int_{\Rd} \int_{\Rd} \Delta U(\bx-\by) f(\by)\rd{\by}\mu(\bx)\rd{\bx} \,.
\end{equation}
Notice that
$
\Delta U(\bx) = - (b+d-2)|\bx|^{b-2}
$
has Fourier transform
$
\widehat{(\Delta U)}(\xi) = - (b+d-2)c|\xi|^{-d-b+2}
$,
where $c$ is a positive constant. Then
\begin{equation}
\int_{\Rd} \int_{\Rd} \Delta U(\bx-\by) f(\by)\rd{\by}\mu(\bx)\rd{\bx} 
= \int_{\Rd} \widehat{(\Delta U)}(\xi) \hat{f}(\xi)\bar{\hat{\mu}}(\xi)\rd{\xi} = (b+d-2)c\int_{\Rd} |\xi|^{-d-b} |\hat{\mu}(\xi)|^2\rd{\xi}\,. \end{equation}
Notice that $b+d-2>0$ for all the stated cases. Therefore the conclusion follows.

{\bf Case 2}: the Newtonian cases $b=2-d$, with $d=1,2$.
In this case, since $-U$ is the fundamental solution to the Laplacian (up to a positive constant multiple), we have
\begin{equation}\begin{split}
& \int_{\Rd} \int_{\Rd} U(\bx-\by)\mu(\by)\rd{\by}\mu(\bx)\rd{\bx} 
=  -c\int_{\Rd}   f(\bx)\Delta f(\bx)\rd{\bx} =  c\int_{\Rd}   |\nabla f(\bx)|^2\rd{\bx}\,,  \\
\end{split}\end{equation}
where the integration-by-parts is justified by \eqref{ftail}. Then, since $\nabla f$ is $L^1$ by \eqref{ftail}, we have
\begin{equation}
\widehat{\nabla f}(\xi) = - \frac{i\xi}{|\xi|^2}\hat{\mu}(\xi)\,,
\end{equation}
which implies
\begin{equation}
\int_{\Rd}   |\nabla f(\bx)|^2\rd{\bx} = \int_{\Rd} \frac{1}{|\xi|^2}|\hat{\mu}(\xi)|^2\rd{\xi}\,.
\end{equation}
Finally, we apply Theorem \ref{thm_Wirs} to deduce the radial symmetry of any compactly supported $d_\infty$-local minimizer of $E$
with $-d < b < 2\leq a\leq 4$ and $d\ge 2$.

Moreover, using \cite{CCP15,SST15} the global minimizers of $E$ in this range are compactly supported. Then, the last claim about global minimizers results directly from the FLIC property. 
\end{proof}

\begin{remark}
The cases $0\leq  b <1$ and $2\leq a\leq 4$ in one dimension are not covered by the previous arguments. 
In fact, the FLIC property is still true for these cases. However, these cases do not have the necessary properties for Theorem \ref{thm_min0} to be applicable, and thus, we postpone the proof to the appendix. 
\end{remark}

\subsection{Regularity properties}

We first show \eqref{eq:hyp2} for the power-law potentials \eqref{Wab}.

\begin{lemma}\label{lem_WHS}
For $W$ given by \eqref{Wab} with $a\ge 2$ and $-d<b<2$, there exists $C_1>0$ such that $W+C_1$ satisfies \eqref{eq:hyp2}.
\end{lemma}
\begin{proof}
Take $C_1=\max\{-\inf W,0\}+1 > 0$, and then $W_1:=W+C_1$ is bounded from below by $1$. It suffices to show that for any $R>0$, there exists $C_R>0$ such that
$\frac{1}{|B(\bx;r)|}\int_{B(\bx;r)}W_1(\by)\rd{\by} \le C_R W_1(\bx), \forall |\bx|\le R,\,0<r\le R$.
 If $b>0$, then $W_1$ is continuous, and the function
\begin{equation}
\phi(\bx,r) = \left\{\begin{split}& \frac{\frac{1}{|B(\bx;r)|}\int_{B(\bx;r)}W_1(\by)\rd{\by}}{W_1(\bx)},\quad r\ne 0\\
& 1,\quad r=0
\end{split}\right.
\end{equation}
is defined for $\widebar{B(0;R)}\times [0,R]$ and continuous. Therefore $\phi$ achieves its maximum, and the conclusion follows.

If $b< 0$, with the constants $C$ depending on $R$,
\begin{equation}
\int_{B(\bx;r)}W_1(\by)\rd{\by} \le C\int_{B(\bx;r)}\rd{\by} + C\int_{B(\bx;r)}|\by|^b\rd{\by}
\end{equation}
for any $|\bx|\le R$ and $r\le R$. If $r\le \frac{|\bx|}{2}$, then $\frac{|\bx|}{2}\le |\by| \le \frac{3|\bx|}{2}$ in the last integral, and
\begin{equation}
\int_{B(\bx;r)}|\by|^b\rd{\by} \le C|\bx|^b\int_{B(\bx;r)}\rd{\by} \le CW_1(\bx)\int_{B(\bx;r)}\rd{\by}
\end{equation}
and the conclusion follows. If $r> \frac{|\bx|}{2}$, then
\begin{equation}
\int_{B(\bx;r)}|\by|^b\rd{\by} \le \int_{B(0;r)}|\by|^b\rd{\by} = C\int_0^r s^bs^{d-1}\rd{s} = Cr^{b+d}
\end{equation}
using the radially-decreasing property of $|\by|^b$ and the assumption $-d<b<0$. Therefore
\begin{equation}
\frac{1}{|B(\bx;r)|}\int_{B(\bx;r)}|\by|^b\rd{\by}\le Cr^b \le C|\bx|^b \le CW_1(\bx)
\end{equation}
and the conclusion follows.

The case $b=0$ (i.e., $-\frac{|\bx|^b}{b}:=-\ln |\bx|$) can be treated similarly as the $b<0$ case.
\end{proof}

Next we give the mild property of the global minimizer for power-law potentials.
\begin{lemma}\label{lem_rhoreg}
Assume $a\ge 2$, $2-d<b<2$ and $W$ be given by \eqref{Wab}. Assume $\rho\in\PP$ is compactly supported. If we have either of the following:
\begin{itemize}
\item $d=1$;
\item $d\ge 2$, and $\rho$ is radially-symmetric;
\end{itemize}
then $\rho$ is mild. In particular, the global minimizer $\rho_\infty$ with $(a,b)$ satisfying the assumption of Theorem \ref{lem_L1} is mild.
\end{lemma}

\begin{proof}

For the case $d=1$, we have $1<b<2$. Notice that $\nabla W(x) = \sgn(x)(|x|^{a-1}-|x|^{b-1})$ is continuous, and therefore $\nabla W*\rho$ is continuous.

For the case $2-d<b<2$ with $d\ge 2$ and $\rho$ being compactly-supported and radially-symmetric, there exists a measure $\tilde{\rho}$ supported on $[0,R]$ for some $R>0$ such that
\begin{equation}\label{phiradial}
\int_\Rd\phi(\by)\rho(\by)\rd{\by} = \int_{[0,R]} \int_{|\by|=1}\phi(r\by)\rd{S(\by)} \tilde{\rho}(r)\rd{r}
\end{equation}
for any continuous function $\phi$, where $\rd{S(\by)}$ denotes the surface measure on the unit sphere. 
We clearly have
\begin{equation}\label{masstrho}
\int_{[0,R]} \tilde{\rho}(r)\rd{r}=\frac{1}{|S^{d-1}|}
\end{equation}
by taking $\phi=1$, since $\rho$ has total mass 1. 
\eqref{phiradial} is also applicable to $W(\bx-\by)$ for fixed $\bx$ by an approximation argument on the potential, and gives
\begin{equation}
(W*\rho)(\bx) =\int_{[0,R]} \int_{|\by|=1}W(s\vec{e}_1-r\by)\rd{S(\by)} \tilde{\rho}(r)\rd{r},\quad s=|\bx|
\end{equation}

To analyze the continuity of $W*\rho$ for $W$ given by \eqref{Wab} with $a\ge 2$, we only need to consider the potential $W=-\frac{|\bx|^b}{b}$. Also, due to the radial symmetry of $W*\rho$, we only need to consider directional derivative along the radial direction, which is
\begin{equation}\label{dWrhor}
(\nabla W*\rho)(\bx)\cdot \frac{\bx}{|\bx|} = \int_{[0,R]} \int_{|\by|=1}\omega'(|s\vec{e}_1-r\by|)\frac{s\vec{e}_1-r\by}{|s\vec{e}_1-r\by|}\cdot \vec{e}_1\rd{S(\by)} \tilde{\rho}(r)\rd{r}
\end{equation}
for $\bx\ne 0$, where on the RHS we write $W(\bx)=\omega(|\bx|)$. Here \eqref{dWrhor} can be justified as long as the RHS is dominated by an $L^1$ function uniformly in a neighborhood of $s$, which we will prove in the rest of the proof. 
For $W=-\frac{|\bx|^b}{b}$, we have
\begin{equation}\begin{split}
\int_{|\by|=1}(-|s\vec{e}_1-r\by|^{b-1})&\frac{s\vec{e}_1-r\by}{|s\vec{e}_1-r\by|}\cdot \vec{e}_1\rd{S(\by)} \\ &= -|S^{d-2}|\int_{0}^\pi \big((s-r\cos\theta)^2+(r\sin\theta)^2\big)^{(b-2)/2}(s-r\cos\theta) |\sin \theta|^{d-2}\rd{\theta}.
\end{split}\end{equation}
Here the last integrand is dominated by 
\begin{equation}
\big((s-r\cos\theta)^2+(r\sin\theta)^2\big)^{(b-1)/2}=((s-r)^2 + 2rs(1-\cos\theta))^{(b-1)/2} \le C\min\{|s-r|^{b-1},rs|\theta|^{b-1}\}.
\end{equation}
Therefore the RHS of \eqref{dWrhor} with  is dominated by
\begin{equation}
C\int_{[0,R]} \int_{0}^\pi \min\{|s-r|^{b-1},rs|\theta|^{b-1}\}\theta^{d-2}\rd{\theta}  \tilde{\rho}(r)\rd{r},
\end{equation}
which is uniformly bounded (by a constant multiple of \eqref{masstrho}) for $s\in [\epsilon,R]$ for any $\epsilon>0$ as long as $b>2-d$. This justify \eqref{dWrhor}, and proves the continuity of $(\nabla W*\rho)(\bx)$ by dominated convergence, for $s=|\bx|\ne 0$.

Finally, if the assumption of Theorem \ref{lem_L1} is satisfied for $(a,b)$, then $E$ has the FLIC property, and then the global minimizer is unique, compactly supported and radially-symmetric. Therefore the previous argument can be applied to conclude that the global minimizer is mild. 
\end{proof}

Finally we recall the explicit construction \cite{CH} for the explicit steady state for power-law potentials defined in \eqref{Wab} with $a=2$ and $2-d<b<\min\{4-d,2\}$, given by:
\begin{equation}\label{WCH}
\rho_\infty(\bx) = A(R^2-|\bx|^2)^{1-\frac{b+d}{2}}\chi_{|\bx|\le R}
\end{equation}
where $A=\frac{-d\Gamma(\frac{d}{2})\sin\frac{(b+d)\pi}{2}}{(b+d-2)\pi^{\frac{d}{2}+1}}>0$, and $R$ is uniquely determined by the total mass condition $\int_\Rd \rho_\infty\rd{\bx}=1$. We also have the explicit formula \cite{CH} for the case $a=4$ and $2-d<b<\bar b<3-d$, given by:
\begin{equation}\label{WCH2}
\rho_\infty(\bx) = (A_1 R^2+ A_2 (R^2-|\bx|^2)) (R^2-|\bx|^2)^{1-\frac{b+d}{2}}\chi_{|\bx|\le R}
\end{equation}
with $A_1$, $A_2$ and $R$ uniquely determined by the total mass condition $\int_\Rd \rho_\infty\rd{\bx}=1$ and a second moment condition. Notice that the upper bound $\bar b$ is given by the relation $A_1+A_2=0$ for which the function given by \eqref{WCH2} touches 0 at the origin. We refer to \cite{CH} for further details.

\begin{lemma}\label{lem_WCH}
If $W$ is given by \eqref{Wab} with $a=2$ and $2-d<b<\min\{4-d,2\}$ or $a=4$ and $2-d<b<\bar b$, then $\rho_\infty$ defined in \eqref{WCH} or \eqref{WCH2} is a steady state with $\nabla^2 W*\rho_\infty$ being continuous.
\end{lemma}

\begin{proof}
It is proved in \cite{CH} that $\rho_\infty$ defined in \eqref{WCH} or \eqref{WCH2} is a steady state.

It is clear that $\nabla^2(\frac{|\bx|^a}{a}) * \rho_\infty$ is continuous. To deal with $\nabla^2(\frac{|\bx|^b}{b}) * \rho_\infty$, first notice that $\rho_\infty \in L^p$ for any $p<\frac{2d}{b+d-2}$. Then by Hardy-Littlewood-Sobolev inequality,
\begin{equation}
|\bx|^{b-2-\epsilon}  * \rho_\infty \in L^q,\quad \forall 1\le q < \infty
\end{equation}
for $\epsilon>0$ small enough, by checking the index relation
\begin{equation}
\frac{b+d-2}{2d} < \frac{d-(2-b)}{d}
\end{equation}
for $b>2-d$.

Then notice that
\begin{equation}
\Delta (\frac{|\bx|^b}{b}) * \rho_\infty = (b+d-2) |\bx|^{b-2} * \rho_\infty = c|\bx|^{-d+\epsilon}*\Big( |\bx|^{b-2-\epsilon} * \rho_\infty\Big)
\end{equation}
where the last parenthesis is in $L^q$ for any $1\le q < \infty$. Therefore we see that $\Delta (\frac{|\bx|^b}{b}) * \rho_\infty$ is continuous by taking $q$ large enough.
The continuity of $\nabla^2 (\frac{|\bx|^b}{b}) * \rho_\infty$ follows by taking the Riesz transform on $\rho_\infty$ which is bounded on $L^p$. 
\end{proof}



%
%
%
%

Finally we prove Theorem \ref{thm_ab}.
\begin{proof}[Proof of Theorem \ref{thm_ab}]

We first notice that for the special case $(a,b)=(2,2-d)$, \cite[Theorem 3.4(i)]{CDM16} shows that any $d_\infty$-local minimizer is in $L^\infty$ with $W*\rho_\infty$ is $C^{1,1}$. Then the unique $d_\infty$-local minimizer is the characteristic function of a ball as shown in \cite[Theorem 2.1]{shu2021newtonian}, coinciding with the formula \eqref{WCH}. We also refer to \cite{F35} for the classical proof that the characteristic of the ball is the global minimizer. In the rest of the proof, we will assume $(a,b)\ne(2,2-d)$.

Let $\rho$ be a compactly supported $d_\infty$-local minimizer. When $d\ge 2$ and \eqref{thm_ab_1} is satisfied, $W$ is FLIC by Theorem \ref{lem_L1}. Then Theorem \ref{thm_Wirs} implies that $\rho$ is radially-symmetric.

If $(a,b)$ further satisfies $2-d\le b<2$, then the unique global minimizer $\rho_\infty$ is mild. In fact, for the Newtonian case $b=2-d$, \cite[Theorem 3.4(i)]{CDM16} shows that $W*\rho_\infty$ is $C^{1,1}$, and in particular, $\nabla W*\rho_\infty$ is continuous and thus $\rho_\infty$ is mild. For the case $2-d<b<\min\{4-d,2\}$, since $\rho_\infty$ is radially-symmetric, Lemma \ref{lem_rhoreg} shows that $\rho_\infty$ is mild.

If $(a,b)$ further satisfies \eqref{thm_ab_2}, by calculating
\begin{equation}
\Delta W(\bx) = (a+d-2)|\bx|^{a-2} - (b+d-2)|\bx|^{b-2}
\end{equation}
\begin{equation}
\Delta^2 W(\bx) = (a+d-2)(a-2)(a+d-4)|\bx|^{a-4} - (b+d-2)(b-2)(b+d-4)|\bx|^{b-4}
\end{equation}
we see that \eqref{thm_min_1} and \eqref{thm_min_2} are satisfied under the assumption \eqref{thm_ab_2} with $(a,b)\ne(2,2-d)$. Also, by Lemma \ref{lem_WHS}, $W$ satisfies \eqref{eq:hyp2} up to adding a constant. Moreover, the positive bound from below on the dimension of the support of $d_\infty$-local minimizers, due to \cite[Theorem 1]{BCLR2}, implies that $\rho_\infty$ does not contain an isolated Dirac at 0.
Therefore, if the $d_\infty$-local minimizer $\rho$ is mild, then Theorem \ref{thm_min0} together with Remark \ref{thm_min0_rem} allows to conclude that $\rho$ is the global minimizer.

Finally, if $a=2$ with $2-d<b<4-d$, then Lemma \ref{lem_WCH} implies that $\rho_\infty$ defined in \eqref{WCH} is a steady state with $\nabla^2W*\rho_\infty$ being continuous (and therefore $\nabla^2W*\rho_\infty=0$ on $\supp\rho_\infty=\widebar{B(0;R)}$). Then Theorem \ref{thm_min} and Theorem \ref{thm_W2g} implies that $\rho_\infty$ is the global minimizer.

If $a=4$ with $2-d<b<\bar b$, then Lemma \ref{lem_WCH} implies that $\rho_\infty$ defined in \eqref{WCH2} is a steady state with $\nabla^2W*\rho_\infty$ being continuous. In particular, $(\Delta W*\rho_\infty)(\bx)=0$ for $|\bx|=R$. Notice that $\Delta W$, viewed as a function of $r=|\bx|$, is increasing on $r\in (0,\infty)$. Therefore the radial function $\Delta W*\rho_\infty$ is increasing on $[R,\infty)$, and thus $(\Delta W*\rho_\infty)(\bx)\ge 0$ for any $|\bx|\ge R$. This implies that $\rho_\infty$ satisfies the condition \eqref{W2cond}, and then Theorem \ref{thm_W2g} implies that $\rho_\infty$ is the global minimizer.

\end{proof}

\begin{remark}
In the case $d\ge 3$, $a=2$, $b=4-d$, one can show that the steady state $\rho$ in the form of Dirac Delta on a sphere (with radius $R$) given by \cite{BalagueCarrilloLaurentRaoul} is the global minimizer. In fact, notice that $\Delta^2 W(\bx)=0$ for any $\bx\ne 0$, and therefore $\Delta^2 W*\rho=0$ in $B(0;R)$, which implies that $\Delta W*\rho$ is constant in $B(0;R)$ by radial symmetry. The steady state $\rho$ is mild by Lemma \ref{lem_rhoreg}, i.e., $\nabla W*\rho$ is continuous and vanishes for $|\bx|=R$. Then we obtain $\Delta W*\rho=0$ in $B(0;R)$, which implies that $W*\rho$ is constant in $\widebar{B(0;R)}$. Similar to the proof of Lemma \ref{lem_rhoreg}, one can show that $\Delta W*\rho$ is continuous. Then a similar argument as in the last paragraph of the proof of Theorem \ref{thm_ab} shows that $\Delta W*\rho\ge 0$ for any $|\bx|\ge R$. Therefore we see that $\rho$ is the global minimizer by Theorem \ref{thm_W2g}.
\end{remark}


\section{$d_\infty$-local minimizers of near power-law potentials in 1D are not fractal}

Consider a 1D potential of the form
\begin{equation}\label{Wb1}
W(x) = -\frac{|x|^b}{b} + W_1(x)
\end{equation}
where $1<b<2$ and $W_1$ is smooth. Applying the dimensionality result in \cite{BCLR2}, we know that $1\geq\textnormal{dim}(\supp\rho)\ge 2-b$ for all $d_\infty$-local minimizers. We now prove that the dimension of the support is actually the maximum. It is in this sense that we name these $d_\infty$-local minimizers not being fractal, although this does not exclude the case where $\supp\rho$ is the union of some closed intervals and a set with fractal dimension.

\begin{theorem}\label{thm_1d}
Let $\rho\in\PP(\R)$ be a $d_\infty$-local minimizer of \eqref{Wb1} which is supported inside $(-R,R)$, and satisfies $W*\rho\in C^1([-R,R])$. Then there exists $c_s>0$ depending on $b, R$ and $\|W_1\|_{C^4([-2R,2R])}$, such that $|\supp\rho| \ge c_s$. In particular, $\textnormal{dim}(\supp\rho) = 1$.
\end{theorem}

\begin{proof}

By \cite[Theorem 1]{BCLR2}, $\supp\rho$ does not contain any isolated point.

{\bf STEP 1}: Rough estimate of the local mass.
We first show that for every connected component of $\supp\rho$ which is a closed interval $I$, it satisfies
\begin{equation}\label{I2bound}
\int_I \rho(x)\rd{x} \le \frac{1}{b-1}\|W_1\|_{C^2}|I|^{2-b},
\end{equation}
where the $C^2$ norm is on $[-2R,2R]$ (and similar for other $C^2,C^4$ norms in this proof). 
In fact, take any fixed $x\in I$, we have
\begin{equation}
0 = (W''*\rho)(x) = -(b-1)(|\cdot|^{b-2}*\rho)(x) +  (W_1''*\rho)(x)
\end{equation}
since $W*\rho$ is a constant on $I$. Notice that
\begin{equation}
(|\cdot|^{b-2}*\rho)(x) =   \int_{[-R,R]} |x-y|^{b-2}\rho(y)\rd{y} \ge  \int_I |x-y|^{b-2}\rho(y)\rd{y} \ge |I|^{b-2}\int_I \rho(y)\rd{y}
\end{equation}
and 
\begin{equation}
\left|(W_1''*\rho)(x)\right| = \left| \int_{[-R,R]}  W_1''(x-y)\rho(y)\rd{y}\right| \le \|W_1\|_{C^2}
\end{equation}
and \eqref{I2bound} follows.

Then we show that for every open interval $J$ such that $J\cap\supp\rho$ is nonempty and not connected, it satisfies
\begin{equation}\label{J2bound}
\int_J \rho(x)\rd{x} \le \frac{1}{(b-1)(b-2)(b-3)}|J|^{4-b}\|W_1\|_{C^4}
\end{equation}
In fact, since $J\cap\supp\rho$ is nonempty and not connected, we may take $x_1,x_2\in J\cap\supp\rho$ with $x_1<x_2$ and $[x_1,x_2]\not\subseteq\supp\rho$. Then we may take a maximal open interval $(x_3,x_4)$ in  $[x_1,x_2]\backslash\supp\rho$.


Then for any $x\in (x_3,x_4)$, we compute
\begin{equation}\label{W4eq}
W''''(x)=-(b-1)(b-2)(b-3)|x|^{b-4}+W_1''''(x),\quad W''''*\rho = -(b-1)(b-2)(b-3)|\cdot|^{b-4}*\rho+W_1''''*\rho
\end{equation}
and estimate
\begin{equation}\begin{split}
(|\cdot|^{b-4}*\rho)(x) = & \int_{[-R,R]} |x-y|^{b-4}\rho(y)\rd{y} \ge \int_{J}|x-y|^{b-4}\rho(y)\rd{y} \ge |J|^{b-4} \int_J \rho(y)\rd{y}\\
\end{split}\end{equation}
and
\begin{equation}\label{W4C4}
|(W_1''''*\rho)(x)| \le \|W_1\|_{C^4}.
\end{equation}
If \eqref{J2bound} was not true, then $(W''''*\rho)(x)<0$ for any $x\in (x_3,x_4)$. Notice that $W\in C^1(\mathbb{R})$, and so is $W*\rho$. Therefore, applying Lemma \ref{lem_d4}
to $W*\rho$ on $[x_3,x_4]$, we deduce that either $x_3$ or $x_4$ is not a local minimum of $W*\rho$, and we get a contradiction against the $d_\infty$-local minimizer property of $\rho$ in view of Lemma \ref{lemel1}.

{\bf STEP 2}: Decomposition of the support.
Assume on the contrary that $|\supp\rho|< c_s$, with the constant $c_s>0$ to be determined. Then we apply the technical Lemma \ref{lem_1ddecomp}, that is postponed below, to $\supp\rho$ (which has no isolated points) with $\epsilon=c_s$. This gives us a decomposition of $\supp\rho$ into connected component intervals $\{I_1,I_2,\dots\}$ and a cover of intervals $\{J_1,J_2,\dots\}$ with the properties listed therein. In particular, we have $ \sum_k|I_k|+\sum_l |J_l| < |\supp\rho|+\epsilon < 2c_s$ from item 4 of Lemma \ref{lem_1ddecomp}.

By item 3 of Lemma \ref{lem_1ddecomp}, we  may apply the estimate \eqref{J2bound} for $J_1,J_2,\dots$ to get
\begin{equation}\begin{split}
\sum_l \int_{J_l}\rho(x)\rd{x} \le & \frac{\|W_1\|_{C^4}}{(b-1)(b-2)(b-3)}\sum_l|J_l|^{4-b}\le \frac{\|W_1\|_{C^4}}{(b-1)(b-2)(b-3)}(2c_s)^{3-b}\sum_l|J_l| \\
\le & \frac{\|W_1\|_{C^4}}{(b-1)(b-2)(b-3)}(2c_s)^{4-b}.
\end{split}\end{equation}
Therefore, with the condition
\begin{equation}\label{mscond0}
\frac{\|W_1\|_{C^4}}{(b-1)(b-2)(b-3)}(2c_s)^{4-b} \le \frac{1}{4},
\end{equation}
we get 
\begin{equation}
\sum_l \int_{J_l}\rho(x)\rd{x} \le \frac{1}{4}\,.
\end{equation}

Since $\{I_1,I_2,\dots\}\cup\{J_1,J_2,\dots\}$ covers $\supp\rho$, we have $\sum_k \int_{I_k}\rho\rd{x}+\sum_l \int_{J_l}\rho\rd{x} \ge \int_{[-R,R]}\rho\rd{x}= 1$, and then we get
\begin{equation}\label{sumik1}
\sum_k \int_{I_k}\rho(x)\rd{x} \ge \frac{3}{4}.
\end{equation}
Then we define
\begin{equation}\label{mkeq}
\delta_k = |I_k|, \quad m_k =  \int_{I_k} \rho(x)\rd{x} ,\quad S= \Big\{k: m_k \ge \frac{1}{4c_s}\delta_k\Big\}.
\end{equation}
Then $\sum_k \delta_k \le |\supp\rho| < c_s$ since $\{I_k\}$, as the connected component intervals of $\supp\rho$, are disjoint. Notice that
\begin{equation}
\sum_{k\notin S}m_k \le  \frac{1}{4c_s}\sum_{k\notin S}\delta_k \le  \frac{1}{4c_s}c_s = \frac{1}{4}.
\end{equation}
Combined with \eqref{sumik1}, we get
\begin{equation}
\sum_{k\in S} m_k \ge \frac{1}{2}.
\end{equation}
We will also use the fact that
\begin{equation}\label{eqms}
m_k \le \frac{1}{b-1}\|W_1\|_{C^2}|I_k|^{2-b}\le \frac{1}{b-1}\|W_1\|_{C^2}c_s^{2-b}=:m_s
\end{equation}
for every $k$, which can be seen by applying \eqref{I2bound} to $I_k$.

{\bf STEP 3}: 4-th order derivative estimate. Fix $k\in S$. Denote $I_k=[x_1,x_2]$ and then  $\delta_k=|I_k|=x_2-x_1$. Define
\begin{equation}
\tilde{I}_k = [x_1-C_1m_k ,x_2+C_1m_k ],
\end{equation}
where $C_1$ is a large constant to be determined. We first show 
\begin{equation}\label{W4rho2}
(W''''*\rho)(x)<0,\quad \forall x\in \tilde{I}_k\backslash \supp\rho
\end{equation}
under a suitable condition \eqref{mscond1}. In fact, by \eqref{W4eq}, $(W''''*\rho)(x)$ is decomposed into two parts, with the second part controlled by \eqref{W4C4}. To control the first part,
\begin{equation}\begin{split}
(|\cdot|^{b-4}*\rho)(x) = & \int_{[-R,R]} |x-y|^{b-4}\rho(y)\rd{y} \ge \int_{I_k}|x-y|^{b-4}\rho(y)\rd{y} \ge (\delta_k+C_1m_k)^{b-4}m_k \\
\ge & (4c_sm_k+C_1m_k)^{b-4}m_k = (4c_s+C_1)^{b-4}m_k^{b-3} \ge (4c_s+C_1)^{b-4}m_s^{b-3},
\end{split}\end{equation}
where the second inequality uses the fact that $|x-y|\le \delta_k+C_1m_k$ for $x\in\tilde{I}_k$ and $y\in I_k$, the third inequality uses the definition of $S$ in \eqref{mkeq}, and the last inequality uses $m_k\le m_s$ from \eqref{eqms}. Then \eqref{W4rho2} follows if we assume the condition
\begin{equation}\label{mscond1}
(b-1)(b-2)(b-3)(4c_s+C_1)^{b-4}m_s^{b-3} > \|W_1\|_{C^4}.
\end{equation}

{\bf STEP 4}: Vacuum regions and exclusion. Using the same argument as in STEP 1,  we claim that 
\begin{equation}\label{vacuum}
(\tilde{I}_k\backslash I_k)\cap \supp\rho = \emptyset,\quad \forall k\in S\,.
\end{equation}
Suppose not, then there exists $y\in (\tilde{I}_k\backslash I_k)\cap \supp\rho$, which we may assume to satisfies $x_2<y\le x_2+C_1m_k$ without loss of generality. Then $[x_2,y]\not\subseteq \supp\rho$ since $I_k=[x_1,x_2]$ is a connected component of $\supp\rho$. Then, in view of \eqref{W4rho2}, we may find a maximal open interval $(x_3,x_4)$ in $[x_2,y]\backslash \supp\rho$ to apply Lemma \ref{lem_d4} and get a contradiction, similar to the last paragraph of STEP 1.

For $I_k = [x_1,x_2]$, define
\begin{equation}
\bar{I}_k = [x_1-\frac{C_1}{3}m_k ,x_2+\frac{C_1}{3}m_k ]
\end{equation}
and we claim that $\{\bar{I}_k:k\in S\}$ are disjoint under a suitable condition \eqref{mscond2}. In fact, suppose $\bar{I}_k\cap\bar{I}_{k'}\ne\emptyset$ and $m_k\ge m_{k'}$, then by $\delta_{k'}\le 4c_sm_{k'}\le 4c_sm_{k}$,
\begin{equation}
\dist(I_k,I_{k'}) \le \frac{C_1}{3}(m_k+m_{k'}) \le \frac{2C_1}{3}m_k \le C_1m_k - 4c_sm_k\le C_1m_k - \delta_{k'},
\end{equation}
if one assume the condition
\begin{equation}\label{mscond2}
\frac{C_1}{3} \ge 4c_s.
\end{equation}
This implies that $I_{k'}\subseteq \tilde{I}_k$. Since $I_k$ and $I_{k'}$ are disjoint and $I_{k'}\cap \supp\rho\ne \emptyset$, we get a contradiction with \eqref{vacuum}.

The disjoint property of $\{\bar{I}_k\}_{k\in S}$ and the fact that $I_k\subseteq [-R,R]$ show that
\begin{equation}\label{mscontra1}
\sum_{k\in S} |\bar{I}_k| \le 2\Big(R+\frac{C_1}{3}m_s\Big)
\end{equation}
since $\bar{I}_k\subseteq [-(R+\frac{C_1}{3}m_s),(R+\frac{C_1}{3}m_s)]$. On the other hand,
\begin{equation}\label{mscontra2}
\sum_{k\in S} |\bar{I}_k| \ge \frac{2C_1}{3}\sum_{k\in S}m_k \ge \frac{C_1}{3}
\end{equation}
by \eqref{mkeq}. Recall from \eqref{eqms} that $m_s$ can be made arbitrarily small by taking $c_s$ small. To choose the parameters $C_1$ and $c_s$, we first choose $C_1=9R$, and then choose $c_s$ small enough so that $m_s$ is small enough to satisfy \eqref{mscond0}, \eqref{mscond1}, \eqref{mscond2} and $m_s<1/9$, to obtain a contradiction between \eqref{mscontra1} and \eqref{mscontra2}.

\end{proof}

\begin{lemma}\label{lem_1ddecomp}
For any closed set $A\subseteq (-R,R)$ with no isolated points and $\epsilon>0$, there exists a countable collection of intervals $\{I_1,I_2,\dots\}\cup\{J_1,J_2,\dots\}$ which covers $A$, satisfying
\begin{itemize}
\item $\{I_1,I_2,\dots,J_1,J_2,\dots\}$ are subsets of $(-R,R)$.
\item $\{I_k\}$ are the connected components of $A$ being closed intervals. $\{I_k\}$ is a finite or countable collection.
\item For every $l$, $J_l$ is an open interval and $J_l\cap A$ is nonempty and not connected.
\item $\sum_k |I_k|+\sum_l |J_l| < |A|+\epsilon$.
\end{itemize}
(Here we do not regard a single point as a closed interval.)
\end{lemma}
\begin{proof}
We first take $\{I_k\}$ as the collection of all the connected components of $A$ being closed intervals, and this collection is pairwise disjoint, and either finite or countable. Denote $I=\bigcup_k I_k$, then
\begin{equation}
|A| = \sum_k |I_k| + |A\backslash I|\,.
\end{equation}
By the definition of the Lebesgue measure, there exists a countable collection of open intervals $\{J_1,J_2,\dots\}$ which are subsets of $(-R,R)$ and  cover $A\backslash I$, such that $\sum_{l=1}^\infty |J_l| < |A\backslash I|+\epsilon$. By deleting those $J_l$ with empty intersection with $A\backslash I$, we may assume that $J_l\cap (A\backslash I)\ne \emptyset$ for every $l$. 

Then for every $l$, we claim that the nonempty set $J_l\cap A$ is not connected. To see this, denote $J_l=(x_1,x_2)$. Suppose $J_l\cap A$ is connected. Since $A$ does not contain isolated points, $J_l\cap A$ has to be an interval, which is a subset of a connect component of $A$ being an interval. Therefore $J_l\cap A\subseteq I_k$ for some $k$, and we get a contradiction against $J_l\cap (A\backslash I)\ne \emptyset$.

\end{proof}

The main conclusion of this section is that if we were looking for fractal behavior on the support of global minimizers of the interaction energy, power-law potentials or their perturbations are not the right family of potentials at least in one dimension.


\section{Linear interpolation concavity and its consequences on local minimizers}

We say $W$ is \emph{linear-interpolation-concave} with size $\delta$, if there exists a nonzero function $\mu\in L^\infty(\Rd)$ with $\int_\Rd\mu\rd{\bx}=0$ and $\diam(\supp\mu) \le \delta$, such that $E[\mu] < 0$.
We say $W$ is \emph{infinitesimal-concave} if $W$ is linear-interpolation-concave with size $\delta$ for any $\delta > 0$.

\begin{theorem}\label{lem_concave}
Assume that the interaction potential $W$ satisfies \eqref{eq:hyp}. If $W$ is linear-interpolation-concave with size $\delta$, then for any $d_\infty$-local minimizer $\rho$ and $\epsilon_0>0$, $\{\bx: \rho(\bx)\ge \epsilon_0\}$ does not contain a ball of radius $\delta$. In particular, if $W$ is infinitesimal-concave, then for any $d_\infty$-local minimizer $\rho$ and $\epsilon_0>0$, $\{\bx: \rho(\bx)\ge \epsilon_0\}$ has no interior point.
\end{theorem}
Here the meaning of the condition $\rho(\bx)\ge \epsilon_0$ is clear if $\rho$ is a continuous function. Otherwise, we may use the Radon-Nikodym decomposition to write $\rho=\rho_c+\rho_s$ where $\rho_c$ is an $L^1$ function and $\supp\rho_s$ has Lebesgue measure zero. Then $\{\bx: \rho(\bx)\ge \epsilon_0\}$ is interpreted as $\{\bx: \rho_c(\bx)\ge \epsilon_0\}$, and the conclusion of Theorem \ref{lem_concave} is that $\{\bx: \rho_c(\bx)\ge \epsilon_0\}$ does not contain a ball of radius $\delta$ for any representative for $\rho_c\in L^1$ (which may differ by a set with Lebesgue measure zero). 

In order to show Theorem \ref{lem_concave}, we need an improvement on the necessary condition for $d_\infty$-local minimizers in Lemma \ref{lemel1}.

\begin{lemma}\label{lem_elball}
Assume that the interaction potential $W$ satisfies \eqref{eq:hyp}, and $\rho$ is a $d_\infty$-local minimizer for the corresponding interaction energy. If $B(\bx;\delta)\subseteq \supp\rho$, then $V=W*\rho$ is a constant on $B(\bx;\delta)$ almost everywhere.
\end{lemma}

\begin{proof}
We will prove an equivalent statement with $B(\bx;\delta)$ replaced by a cube $Q=[x_1,x_1+\delta]\times\cdots\times [x_d,x_d+\delta]$.
Let $\epsilon_0$ be as in Lemma \ref{lemel1}. By replacing $\epsilon_0$ with a smaller number, we may assume $\frac{\epsilon_0}{\sqrt{d}}\le \delta$. Take $\{\bz_n\}_{n=1}^N$ as the set of points in $B(\bx,\delta)$ with each coordinate in $\frac{\epsilon_0}{3\sqrt{d}}\mathbb{Z}$. Then for any $\by\in Q$, there exists some $\bz_n$ such that $|\by-\bz_n|\le \frac{\epsilon_0}{3}< \frac{\epsilon_0}{2}$, i.e., the collection of balls 
$\{B(\bz_n;\frac{\epsilon_0}{2})\}$ covers $Q$.

Suppose it is not true that $V=W*\rho$ is a constant on $Q$ almost everywhere. Then there exists $C_1$ such that both
\begin{equation}
A_1 = \{\bx\in Q:V(\bx)<C_1\}\quad \mbox{and} \quad A_2 = \{\bx\in Q:V(\bx)\ge C_1\}
\end{equation}
have positive Lebesgue measure. We claim that for any $n=1,\dots,N$,
\begin{equation}\label{claim_bz}
\text{if } B(\bz_n;\frac{\epsilon_0}{2})\cap A_2\ne\emptyset,\quad \text{then }|B(\bz_n;\frac{\epsilon_0}{2})\cap A_1|=0.
\end{equation}
In fact, if $B(\bz_n;\frac{\epsilon_0}{2})\cap A_2\ne\emptyset$, then we may take $\by\in B(\bz_n;\frac{\epsilon_0}{2})\cap A_2$. Applying Lemma \ref{lemel1} to the point $\by$ shows that $V(\by_1)\ge V(\by)\ge C_1$ for $\by_1\in B(\by;\epsilon_0)$ almost everywhere. Since $|\by-\bz_n|<\frac{\epsilon_0}{2}$, we have $B(\bz_n;\frac{\epsilon_0}{2})\subseteq B(\by;\epsilon_0)$, and the claim follows.

Since $A_2$ has positive Lebesgue measure, we may take $\by\in A_2$, and take $\bz_m$ with $\by\in B(\bz_m;\frac{\epsilon_0}{2})$. Applying \eqref{claim_bz}, we see that $|B(\bz_m;\frac{\epsilon_0}{2})\cap A_1|=0$. This implies that $B(\bz_{m'};\frac{\epsilon_0}{2})\cap  A_2\ne\emptyset$ for any $\bz_{m'}$ with $|\bz_m-\bz_{m'}|=\frac{\epsilon_0}{3\sqrt{d}}$, i.e., the closest neighbors of $\bz_m$, since $\bz_{m'}$ is an interior point of $B(\bz_m;\frac{\epsilon_0}{2})$. Applying \eqref{claim_bz} iteratively, we obtain $|B(\bz_n;\frac{\epsilon_0}{2})\cap A_1|=0$ for any $ n$. Since $\{B(\bz_n;\frac{\epsilon_0}{2})\}$ covers $Q$, we get $|A_1|=0$, contradicting the assumption that $A_1$ has positive measure.

\end{proof}

\begin{proof}[Proof of Theorem \ref{lem_concave}]
The second statement is clearly a consequence of the first one. We prove the first statement by contradiction. Suppose the contrary, then $\{\bx: \rho(\bx)\ge \epsilon_0\}$ contains a ball $B(\bx_0;\delta)$. Let $\mu$ be the function as in the definition of linear-interpolation-concavity. By translation, we may assume $\supp\mu \in B(\bx_0;\delta)$. Then define the family of probability measures 
\begin{equation}
\rho_\epsilon = \rho + \epsilon\mu,\quad 0<\epsilon<\frac{\epsilon_0}{\|\mu\|_{L^\infty}}\,.
\end{equation}
See Figure \ref{fig_rhoeps} for an illustration. Since the generated potential $W*\rho$ is constant in $B(\bx_0;\delta)$ almost everywhere by Lemma \ref{lem_elball} and $\mu$ is a mean-zero function supported in $B(\bx;\delta)$, we have $\int_\Rd (W*\rho)(\bx)\mu(\bx)\rd{\bx}=0$, and then
\begin{equation}
E[\rho_\epsilon] = E[\rho] + \epsilon \int_\Rd (W*\rho)(\bx)\mu(\bx)\rd{\bx} + \epsilon^2E[\mu] = E[\rho] + \epsilon^2E[\mu] < E[\rho]
\end{equation}
for any $\epsilon>0$. This contradicts the assumption that $\rho$ is a $d_\infty$-local minimizer since $d_\infty(\rho,\rho_\epsilon) \le C\epsilon$.
\end{proof}

\begin{figure}
\begin{center}
	\includegraphics[width=0.8\textwidth]{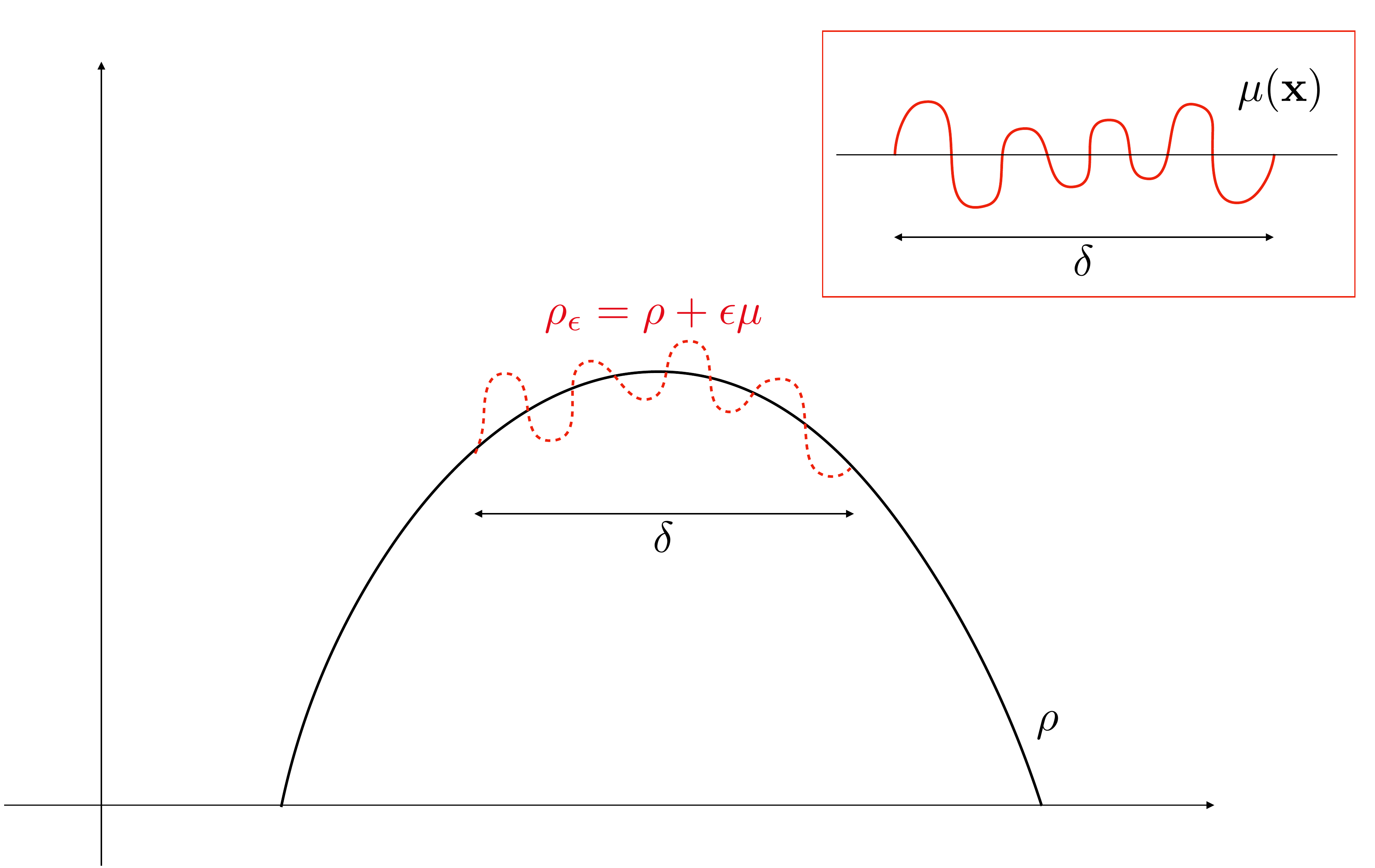}
	\caption{$\mu$ in the definition of linear-interpolation-concavity, and its application is the proof of Theorem \ref{lem_concave}.}
	\label{fig_rhoeps}
\end{center}
\end{figure}

Now we can give sufficient conditions for $W$ to be infinitesimal-concave.

\begin{lemma}\label{lem_inc}
Assume that the interaction potential $W$ is radially symmetric satisfying \eqref{eq:hyp}. Given parameters $\beta,\alpha,c_1,C_1>0$. For any $\delta>0$, there exists $R>0$, such that the conditions
\begin{itemize}
\item[1.] $W(r) \le C_1(1+r)^\beta$. 
\item[2.] $\hat{W}$ (as a distribution) is a function on any compact set inside $\mathbb{R}^d\backslash \{0\}$.
\item[3.] There exists an interval $J\subseteq \mathbb{R}_+$ with $|J|=R$ such that $\hat{W}(\xi) < -c_1R^{-\alpha},\,\forall |\xi|\in J$.
\end{itemize}
imply that $W$ is linear-interpolation-concave with size $\delta$. In particular, if in addition condition 3 is true for any $R\ge 1$, then $W$ is infinitesimal-concave.
\end{lemma}

\begin{proof}
We may assume $\delta\le1$ without loss of generality, and then start by requiring that $R\ge 1/\delta\ge1$.
We fix a smooth positive radial function $\phi(\xi)$ supported on $B(0;\frac{1}{2})$. Then for any $m>0$, there exists $C_m$ such that
\begin{equation}\label{checkphi}
|\check{\phi}(\bx)| \le C_m(1+|\bx|)^{-m}.
\end{equation}

We need to find a nonzero function $\mu\in L^\infty(\Rd)$ with $\int_\Rd\mu\rd{\bx}=0$ and $\diam(\supp\mu) \le \delta$, such that $E[\mu] < 0$. We define $\mu$ by
\begin{equation}
\mu = \chi_{B(0;\delta/2)}\cdot\Big(\mu_1 - \frac{1}{|B(0;\delta/2)|}\int_{B(0;\delta/2)}\mu_1\rd{\bx}\Big),\quad \mu_1 =  \cF^{-1}\Big(\phi(\frac{\cdot - \xi_J\vec{e}_1}{R})+\phi(\frac{\cdot + \xi_J\vec{e}_1}{R})\Big)
\end{equation}
where $\vec{e}_1=(1,0,\dots,0)^T$, and $\xi_J$ is the center of the interval $J=J(R)$ given as in condition 3. Notice that $\mu_1$ is real-valued since $\hat{\mu}_1$ is even, by definition. Also, $\mu_1$ is mean-zero on $\mathbb{R}^d$ since $\hat{\mu}_1(0)=0$, and $\mu$ is mean-zero by definition. See Figure \ref{fig_mu} for an illustration.

\begin{figure}
\begin{center}
	\includegraphics[width=0.95\textwidth]{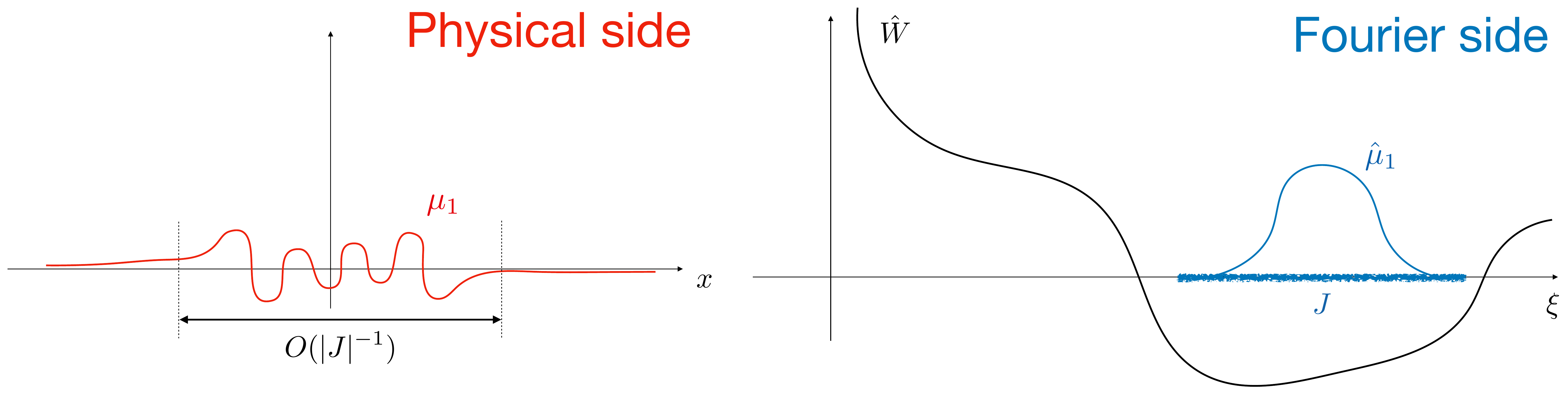}
	\caption{Construction of $\mu$ in the proof of Lemma \ref{lem_inc}. On the Fourier side $\hat{\mu}_1$ is supported on the radial interval $J$ where $\hat{W}$ is negative. On the physical side $\mu_1$ is basically supported on a ball of radius $O(|J|^{-1})$, with a small tail.}
	\label{fig_mu}
\end{center}
\end{figure}

{\bf STEP 1}: We first give a negative upper bound for $E[\mu_1]$. 
In fact, notice that 
\begin{equation}
\supp\Big(\phi(\frac{\cdot - \xi_J\vec{e}_1}{R})\Big) = B(\xi_J\vec{e}_1;\frac{R}{2}) \subseteq\{\xi: |\xi|\in J\}.
\end{equation}
The condition that $\hat{W}$ is a function inside $\mathbb{R}^d\backslash \{0\}$ allows us to apply the formula $E[\mu_1]=\tfrac{1}{2}\int_\Rd \hat{W}|\hat{\mu}_1|^2\rd{\xi}$ since $\mu_1$ is an $L^1$ function with sufficient decay at infinity and $0\notin\supp\hat{\mu}_1$. Therefore we get
\begin{equation}\label{Emu1}\begin{split}
E[\mu_1] = & \frac{1}{2}\int_{|\xi|\in J} \hat{W}(\xi)\left|\phi\Big(\frac{\xi - \xi_J\vec{e}_1}{R}\Big)+\phi\Big(\frac{\xi + \xi_J\vec{e}_1}{R}\Big)\right|^2\rd{\xi} = \int_{|\xi|\in J} \hat{W}(\xi)\left|\phi\Big(\frac{\xi - \xi_J\vec{e}_1}{R}\Big)\right|^2\rd{\xi}\\
 \le & -c_1R^{-\alpha}\int_{|\xi|\in J} \left|\phi\Big(\frac{\xi - \xi_J}{R}\Big)\right|^2\rd{\xi} = -c_1\|\phi\|_{L^2}^2R^{-\alpha+d}.
\end{split}\end{equation}

{\bf STEP 2}: Analyze the difference between $\mu_1$ and $\mu$ on the physical side. Fix a choice of $m$ with $m\ge\beta+\alpha+d+1$. Then \eqref{checkphi} implies
\begin{equation}
\Big|\cF^{-1}\Big(\phi(\frac{\cdot - \xi_J\vec{e}_1}{R})\Big)(\bx)\Big| = \Big|\cF^{-1}\Big(\phi(\frac{\cdot }{R})\Big)(\bx)\Big| \le CR^d(1+R|\bx|)^{-m}.
\end{equation}
Therefore, we deduce
\begin{equation}\label{eqmu1}
|\mu_1(\bx)| \le CR^d(1+R|\bx|)^{-m},\quad \|\mu_1\|_{L^1}\le CR^d\int_\Rd(1+R|\bx|)^{-m}\rd{\bx}=C.
\end{equation}
Combined with condition 1, this implies that
\begin{equation}\label{Wmu1}\begin{split}
|(W*\mu_1)(\bx)| \le & \int_{|\bx-\by|\le |\bx|}|\mu_1(\bx-\by)W(\by)|\rd{\by} + \int_{|\by|> |\bx|}|\mu_1(\by)W(\bx-\by)|\rd{\by} \\
\le & C(1+|\bx|)^\beta\|\mu_1\|_{L^1} + CR^d\int_{|\by|>|\bx|}(1+R|\by|)^{-m}(1+|\bx-\by|)^\beta\rd{\by}\\
\le & C(1+|\bx|)^\beta\|\mu_1\|_{L^1} + CR^d\int_\Rd(1+|\bx-\by|)^{-m}(1+|\bx-\by|)^\beta\rd{\by}\\
\le & CR^d(1+|\bx|)^\beta,
\end{split}\end{equation}
where the second last inequality uses the condition $R\ge 1$ and the fact $|\by|\ge |\bx-\by|/2$ for any $|\by|>|\bx|$, and the last inequality uses the definition of $m$ to get the finiteness of the integral.

Since $\mu_1$ is mean-zero on $\mathbb{R}^d$, we have
\begin{equation}\begin{split}
\left|\int_{B(0;\delta/2)}\mu_1\rd{\bx}\right| = & \left|\int_{B(0;\delta/2)^c}\mu_1\rd{\bx}\right| \le CR^d\int_{|\bx|>\delta/2} (1+R|\bx|)^{-m}\rd{\bx} \\
= & C\int_{|\bx|>\delta R/2} (1+|\bx|)^{-m}\rd{\bx} \le C(\delta R)^{-m+d} .
\end{split}\end{equation}
Therefore, we obtain
\begin{equation}\label{mumu1}\begin{split}
|\mu(\bx)-\mu_1(\bx)| \le & |\chi_{B(0;\delta/2)^c}(\bx)\mu_1(\bx)| + \left|\chi_{B(0;\delta/2)}(\bx)\frac{1}{|B(0;\delta/2)|}\int_{B(0;\delta/2)}\mu_1(\by)\rd{\by}\right| \\
\le & |\mu_1(\bx)|\chi_{B(0;\delta/2)^c}(\bx) + C(\delta R)^{-m+d}\delta^{-d}\chi_{B(0;\delta/2)}(\bx).
\end{split}\end{equation}
Notice that $(W*\chi_{B(0;\delta/2)})(\bx) \le C(1+|\bx|)^\beta$ for any $0<\delta<1$. Therefore, using the assumption $R\ge 1/\delta$ at the beginning of the proof,
\begin{equation}
W*\Big(C(\delta R)^{-m+d}\delta^{-d}\chi_{B(0;\delta/2)}(\bx)\Big)(\bx) \le C(\delta R)^{-m+d}(1+|\bx|)^\beta \le C\delta^{-d}(1+|\bx|)^\beta.
\end{equation}
Therefore, combined with \eqref{Wmu1} (which is also true if $\mu_1$ is replaced by $|\mu_1(\bx)|\chi_{B(0;\delta/2)^c}(\bx)$), we get
\begin{equation}
|(W*\mu)(\bx)| \le C\delta^{-d}R^d(1+|\bx|)^\beta,
\end{equation}
using $R\ge 1$ and $\delta\le 1$.
Finally, combining with \eqref{eqmu1}, \eqref{mumu1} and using $R\ge 1,\,\delta\le 1$,
\begin{equation}\begin{split}
|E[\mu]-E[\mu_1]| \le & \frac{1}{2}\int_\Rd |(W*\mu)(x)(\mu(\bx)-\mu_1(\bx))|\rd{\bx} + \frac{1}{2}\int_\Rd |(W*\mu_1)(\bx)(\mu(\bx)-\mu_1(\bx))|\rd{\bx} \\
\le & C\delta^{-d}R^d\int_\Rd(1+|\bx|)^\beta|(\mu(\bx)-\mu_1(\bx))|\rd{\bx}  \\
\le & C\delta^{-d}R^d\int_{B(0;\delta/2)^c}(1+|\bx|)^\beta|\mu_1(\bx)|\rd{\bx} + C\delta^{-2d}R^d(\delta R)^{-m+d}\int_{B(0;\delta/2)}(1+|\bx|)^\beta\rd{x} \\
\le & C\delta^{-d}R^{2d}\int_{B(0;\delta/2)^c}(1+R|\bx|)^\beta(1+R|\bx|)^{-m}\rd{\bx} + C\delta^{-2d}R^d(\delta R)^{-m+d}\delta^d\\
\le & C\delta^{-d}R^d(\delta R)^{-m+\beta+d}+ C\delta^{-d}R^d(\delta R)^{-m+d},
\end{split}\end{equation}
which implies
\begin{equation}\label{Emumu1}
\frac{|E[\mu]-E[\mu_1]|}{R^{-\alpha+d}} \le C(\delta^{-m+\beta}R^{-m+\beta+\alpha+d} + \delta^{-m}R^{-m+\alpha+d}).
\end{equation}
By the choice of $m$, all the above powers on $R$ are negative. Compared with \eqref{Emu1}, if $R$ large enough, one can guarantee the above RHS is no more than $c_1\|\phi\|_{L^2}^2/2$, and the conclusion is obtained.
\end{proof}

\begin{remark}
Notice that $m$ in the above proof can be chosen arbitrarily large. Therefore, when comparing the RHS of \eqref{Emumu1} and \eqref{Emu1}, for any $\epsilon>0$, one can choose $m$ such that $R\sim \delta^{-1-\epsilon}$.
\end{remark}


\section{Construction of infinitesimal-concave attractive-repulsive potentials}

In this section we aim to construct a class of infinitesimal-concave attractive-repulsive potentials. We start from the Riesz repulsion with quadratic attraction
\begin{equation}
W_0(\bx) = c_{d,\alpha}|\bx|^{\alpha-d} + \frac{|\bx|^2}{2},\quad c_{d,\alpha} = \frac{\Gamma((d-\alpha)/2)}{\pi^{d/2}2^\alpha \Gamma(\alpha/2)}
\end{equation}
where $0<\alpha<d+2$ is a parameter. Here $c_{d,\alpha}$ is chosen so that $\cF[c_{d,\alpha}|\bx|^{\alpha-d}] = |\xi|^{-\alpha}$ away from $\xi=0$, and notice that $c_{d,\alpha}>0$ when $0<\alpha<d$, and $c_{d,\alpha}<0$ when $d<\alpha<d+2$. When $\alpha=d$, we use the tradition $c_{d,\alpha}|\bx|^{\alpha-d}:=-\frac{2}{\pi^{d/2}2^\alpha \Gamma(\alpha/2)}\ln|\bx|$. Finally, in one dimension, we consider the ranges $0<\alpha\leq 1$ and $2\leq\alpha<3$, since for $1<\alpha<2$ we do not know if $\cF[c_{1,\alpha}|\bx|^{\alpha-1}] = |\xi|^{-\alpha}$ holds.

Then define the new potentials by adding a hierarchical sequence of perturbations: 
\begin{equation}\label{Wcons}
W(\bx) = W_0(\bx) - c_W\sum_{k=1}^\infty \lambda^{(\alpha-d)k} \exp\Big(-\frac{|\bx|^2}{2\lambda^{2k}}\Big)
\end{equation}
where $0<\lambda<1$ and $c_W>0$. 

\begin{theorem}\label{thm_cons1}
Let $W$ be given by \eqref{Wcons}. There exist $c(d,\alpha)$ and $C(d,\alpha)$ such that the following hold:
\begin{itemize}
\item If $c_W > c(d,\alpha)$, then $W$ satisfies the conditions in Lemma \ref{lem_inc} for any $R\ge 1$, with $\beta=2$ and the same $\alpha$. It follows that $W$ is infinitesimal-concave.
\item If $c_W < C(d,\alpha)$, then for sufficiently small $\lambda$ (depending on $d$, $\alpha$ and $C(d,\alpha)-c_W$), $W$ is repulsive for short distances and attractive for long distances: there exists $R_W>0$ such that $W'(r) < 0,\,\forall 0<r<R_W$ and $W'(r) > 0,\,\forall r>R_W$.
\end{itemize}
\end{theorem}

\begin{remark}
The explicit expressions of $c(d,\alpha)$ and $C(d,\alpha)$ are given in \eqref{cda} and \eqref{Cda} respectively. One can show that within the range of $\alpha$ stated before, $c(d,\alpha)<C(d,\alpha)$ if and only if
\begin{equation}\label{condfrac}
\frac{d+2}{2}<\alpha<d+2,\textnormal{ for }d\ge 2;\quad 2\le \alpha < 3,\textnormal{ for }d=1
\end{equation}
after some tedious but easy explicit computations. Therefore, for such $\alpha$, one can construct $W$ which is infinitesimal-concave and also repulsive for short distances and attractive for long distances.

\end{remark}

\begin{proof}
\

{\bf First claim}:
Item 1 of the conditions of Lemma \ref{lem_inc} is clear. To check items 2 and 3, we start by noticing that $\cF[|\bx|^2]= -(2\pi)^d\Delta\delta(\xi)$ is supported on $\{0\}$, and (as a distribution) $\cF[c_{d,\alpha}|\bx|^{\alpha-d}] = \frac{1}{|\xi|^\alpha},\,\forall \xi\ne 0$. Therefore, for any $\xi\ne 0$, 
\begin{equation}
\hat{W}(\xi) = \frac{1}{|\xi|^\alpha} - c_W\sum_{k=1}^\infty \cF \Big[\lambda^{(\alpha-d)k}\exp(-\frac{(\cdot)^2}{2\lambda^{2k}})\Big](\xi) = \frac{1}{|\xi|^\alpha} - (2\pi)^{d/2}c_W\sum_{k=1}^\infty\lambda^{\alpha k} \exp\Big(-\frac{\lambda^{2k}|\xi|^2}{2}\Big)
\end{equation}
where we use $\cF[\exp(-\frac{(\cdot)^2}{2})]=(2\pi)^{d/2}\exp(-\frac{|\xi|^2}{2})$. See Figure \ref{fig_kth} for an illustration. This shows item 2 of the conditions of Lemma \ref{lem_inc}. 

\begin{figure}
\begin{center}
	\includegraphics[width=0.95\textwidth]{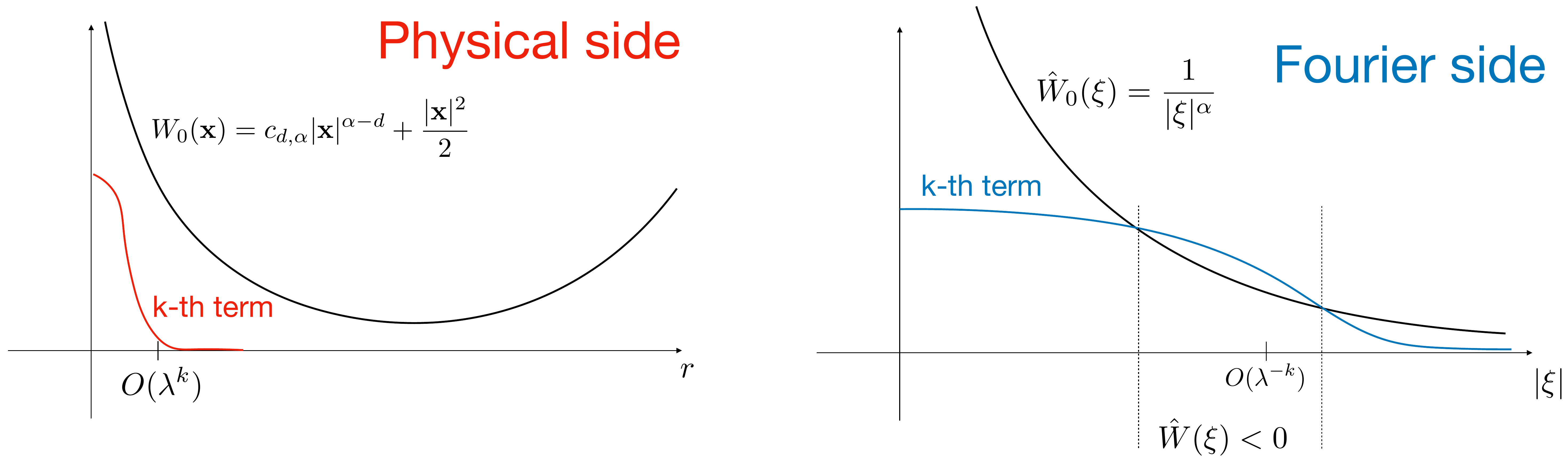}
	\caption{The $k$-th term in the construction of $W$. On the physical side, it is essentially supported in a ball of radius $O(\lambda^k)$, while on the Fourier side, it is essentially supported in a ball of radius $O(\lambda^{-k})$, and its size at $|\xi|\sim O(\lambda^{-k})$ is of the same order of $\hat{W}_0$, and $c_W$ is chosen such that $\hat{W}_0(\xi)$ is smaller than the $k$-th term at this scaling. }
	\label{fig_kth}
\end{center}
\end{figure}

Assume $c_W > c(d,\alpha)$ where $c(d,\alpha)$ will be given later in \eqref{cda}. For any fixed $j\in\mathbb{Z}_+$, we claim that 
\begin{equation}\label{claim_Wxi}
\hat{W}(\xi) < -c|\xi|^{-\alpha},\quad \mbox{for all $\xi$ with } (\sqrt{\alpha}-\epsilon)\lambda^{-j} \le |\xi| \le (\sqrt{\alpha}+\epsilon)\lambda^{-j}
\end{equation}
for some positive constants $\epsilon$ and $c$ independent of $j$. Then, since $\lambda<1$, item 3 of the conditions in Lemma \ref{lem_inc} follows since the length of the interval $R=2\epsilon \lambda^{-j}$ is arbitrarily large by taking $j$ large. 

To prove the claim \eqref{claim_Wxi}, notice that 
\begin{equation}
\hat{W}(\xi) \le \frac{1}{|\xi|^\alpha} - (2\pi)^{d/2}c_W\lambda^{\alpha j} \exp\Big(-\frac{\lambda^{2j}|\xi|^2}{2}\Big) = \lambda^{\alpha j}\psi_1(\lambda^{2j}|\xi|^2),\quad \psi_1(y):= \frac{1}{y^{\alpha/2}} - (2\pi)^{d/2}c_We^{-y/2}\,.
\end{equation}
Notice that the ratio of the two terms in $\psi_1$ is
\begin{equation}
\psi_2(y) := \frac{e^{-y/2}}{1/y^{\alpha/2}} = y^{\alpha/2}e^{-y/2}
\end{equation}
which achieves its maximum at $\psi_2(\alpha) = \alpha^{\alpha/2}e^{-\alpha/2}$ since $\psi_2'(y) = (\frac{\alpha}{2}-\frac{y}{2})y^{(\alpha-2)/2}e^{-y/2}$. Therefore, if 
\begin{equation}\label{cda}
c_W>c(d,\alpha):=\alpha^{-\alpha/2}e^{\alpha/2}(2\pi)^{-d/2}
\end{equation} 
then $\psi_1(y) \le -c < 0$ in a neighborhood of $\alpha$. It follows that $\psi_1(y) \le -c < 0$ if $\sqrt{y}\in (\sqrt{\alpha}-\epsilon,\sqrt{\alpha}+\epsilon)$ for some $\epsilon>0$, which implies \eqref{claim_Wxi}. 


{\bf Second claim}:

{\bf STEP 1}: We first analyze small $r=|\bx|$: 
\begin{equation}\label{dWsmall}\begin{split}
W'(r) = & -c_{d,\alpha}(d-\alpha)r^{\alpha-d-1} + r + c_W\sum_{k=1}^\infty (\lambda^{(\alpha-d-2)k} r)\exp\Big(-\frac{|\lambda^{-k} r|^2}{2}\Big) \\
= &  \left(-c_{d,\alpha}(d-\alpha) + r^{-\alpha+d+2} + c_W\sum_{k=1}^\infty \psi(\lambda^{-2k} r^2) \right)r^{\alpha-d-1} \\
\end{split}\end{equation}
where
\begin{equation}
\psi(y):=y^{(-\alpha+d+2)/2} e^{-y/2} .
\end{equation}
Since $\psi'(y) = (\frac{-\alpha+d+2}{2} - \frac{y}{2})y^{(-\alpha+d)/2}e^{-y/2}$, $\psi$ achieves its maximum at $$\psi(-\alpha+d+2) = (-\alpha+d+2)^{(-\alpha+d+2)/2}e^{-(-\alpha+d+2)/2} = ((-\alpha+d+2)/e)^{(-\alpha+d+2)/2} ,$$ and being increasing/decreasing on $[0,-\alpha+d+2]$ and $[-\alpha+d+2,\infty)$ respectively. 

Taking $\lambda$ small enough, we can guarantee that $\lambda\in[0,-\alpha+d+2]$ and $\lambda^{-1}\in[-\alpha+d+2,\infty)$.  Notice that for any $r>0$, $\{\lambda^{-2k}r^2\}_{k=1}^\infty \cap [\lambda,\lambda^{-1})$ has at most one element. Then
\begin{equation}\begin{split}
\sum_{k=1}^\infty \psi(\lambda^{-2k} r^2) \le & ((-\alpha+d+2)/e)^{(-\alpha+d+2)/2} + \sum_{k=-\infty}^0 \psi(\lambda^{1-2k})+ \sum_{k=0}^\infty \psi(\lambda^{-1-2k}) \\
\le & ((-\alpha+d+2)/e)^{(-\alpha+d+2)/2}+ C\sum_{k=-\infty}^0 \lambda^{(1-2k)(-\alpha+d+2)/2}+ C\sum_{k=0}^\infty \lambda^{1+2k} \\
\le & ((-\alpha+d+2)/e)^{(-\alpha+d+2)/2} + C(\lambda^{(-\alpha+d+2)/2} + \lambda),
\end{split}\end{equation}
where we use $\psi(y) \le y^{(-\alpha+d+2)/2}$ and $\psi(y) \le C/y$ to estimate the two summations respectively ($C$ only depending on $d$ and $\alpha$). Therefore, putting together \eqref{dWsmall} with the condition that 
\begin{equation}\label{Cda}
c_W<C(d,\alpha):=(e/(-\alpha+d+2))^{(-\alpha+d+2)/2}c_{d,\alpha}(d-\alpha),
\end{equation}
then there exists $r_1>0$ and $\lambda_1>0$ depending on $C(d,\alpha)-c_W$, such that $W'(r)<0$ for all $0<\lambda\le\lambda_1$ and $0<r\le r_1$, i.e., $W$ is repulsive on $(0,r_1]$.

{\bf STEP 2:} Then we analyze $W'(r)$ for $r\geq r_1$ fixed by the previous step. Notice that 
\begin{equation}\begin{split}
W''(r) = & c_{d,\alpha}(d-\alpha)(d+1-\alpha)r^{\alpha-d-2} + 1 + c_W\sum_{k=1}^\infty \lambda^{(\alpha-d-2)k} (1-\lambda^{-2k}r^2)\exp\Big(-\frac{|\lambda^{-k} r|^2}{2}\Big)\,. \\
\end{split}\end{equation}
We separate into two cases:

\begin{itemize}
\item If $0<\alpha\le d+1$, then $W_0''(r)=c_{d,\alpha}(d-\alpha)(d+1-\alpha)r^{\alpha-d-2} + 1 > 0$ for all $r>0$, and is bounded from below by 1 for $r\ge r_1$.  Notice that $ye^{-y/2}\le C_n y^{-n}$ for any $n>0$ and $y>0$. Using this with $y=\lambda^{-2k}r^2$ and $n=(-\alpha+d+3)/2$, we get
\begin{equation}\begin{split}
W''(r) \ge & 1 - c_WC_n\sum_{k=1}^\infty \lambda^{(\alpha-d-2)k} \lambda^{2nk}r^{-2n}\ge 1 - c_WC_nr_1^{-2n}\sum_{k=1}^\infty \lambda^{k}  
=  1 - c_WC_nr_1^{-2n}\frac{\lambda}{1-\lambda}
\end{split}\end{equation}
which is positive if $\lambda$ is small enough. Then $W(r)$ is convex as a function of $r$ on $[r_1,\infty)$, which implies that there exists only one point $R_W$ where $W'$ changes sign ($R_W$ always exists because $W'(r)<0$ for $r\in (0,r_1]$ and $W'(r)>0$ for large enough $r$).

\item If $d+1< \alpha<d+2$, then explicit calculation shows that there exists $0<r_2<r_3$ such that $W_0''(r)>0$ on $[r_2,\infty)$, and $W_0'(r)<0$ on $(0,r_3]$. If $r_1<r_3$, let $-w_1<0$ be an upper bound of $W_0'$ on $[r_1,r_3]$. Then for $r_1<r\le r_3$, \eqref{dWsmall} combined with the fact that $\psi(y)\le C/y$ gives
\begin{equation}\begin{split}
W'(r) \le & -w_1 + c_W\sum_{k=1}^\infty \psi(\lambda^{-2k} r^2) r^{\alpha-d-1}  \le -w_1 + c_WC\sum_{k=1}^\infty \lambda^{2k} r^{-2} r^{\alpha-d-1} \\
\le & -w_1 + c_WCr_1^{\alpha-d-3}\frac{\lambda^2}{1-\lambda^2}
\end{split}\end{equation}
which is negative if $\lambda$ is small enough. Therefore $W'(r)<0$ on $(r_1,r_3]$. Similar to the previous case,  there exists only one point $R_W\in [r_2,\infty)$ where $W'$ changes sign. This implies $W'(r)<0$ on $(0,R_W)$ and $W'(r)>0$ on $(R_W,\infty)$. If $r_1\ge r_3$, then one can get the conclusion similar to the previous case for $0<\alpha\le d+1$ since again $W_0''(r)$ is bounded below by a positive constant due to $r_2<r_3\leq r_1$.
\end{itemize}

\end{proof}

Finally we show that any $d_\infty$-local minimizer of the potential $W$ in \eqref{Wcons} does not collapse to Dirac masses.

\begin{proposition}\label{profisolated}
Let $W$ be given by \eqref{Wcons} and $C(d,\alpha)$ as in Theorem \ref{thm_cons1}. If $c_W<C(d,\alpha)$ and $\lambda$ is sufficiently small, then the support of any compactly supported $d_\infty$-local minimizer $\rho$ does not contain any isolated points.
\end{proposition}

\begin{proof}
Assume on the contrary that there is an isolated point $\bx_0\in \supp\rho$, then there exists $\epsilon>0$ such that $\supp\rho\cap B(\bx_0;\epsilon)=\{\bx_0\}$, and $\rho = a\delta(\bx-\bx_0) + \rho\chi_{B(\bx_0;\epsilon)^c}$ for some $a>0$. Then for any $\bx$ with $|\bx-\bx_0|<\epsilon/2$,
\begin{equation}\begin{split}
(W*\rho)(\bx) = & aW(\bx-\bx_0) + \int_\Rd W(\by)\rho(\bx-\by)\chi_{B(\bx_0;\epsilon)^c}(\bx-\by)\rd{\by} \\
 = & aW(\bx-\bx_0) + \int_{|\by|\ge \epsilon/2} W(\by)\rho(\bx-\by)\chi_{B(\bx_0;\epsilon)^c}(\bx-\by)\rd{\by} =: aW(\bx-\bx_0) + \cI(\bx).
\end{split}\end{equation}
Notice that $W$ is a smooth function on $\{\by:|\by|\ge \epsilon/2\}$, and its derivatives are bounded 
on any compact subset. Therefore $\cI$ is a smooth function on $\bx\in B(\bx_0;\epsilon/2)$. 

On the other hand, by the second item of Theorem \ref{thm_cons1}, $W'(r)<0$ for $0<r<\epsilon_0$ if $\epsilon_0$ is small. Indeed, by its proof (STEP 1 of the second claim), one can show a quantitative version, i.e., there exists $c>0$ such that
\begin{equation}
W'(r)<-cr^{\alpha-d-1},\quad \forall 0<r<\epsilon_0.
\end{equation}
Then, either $W(0)=\infty$ (when $0<\alpha\le d$), or $W(0)<\infty$ (when $d<\alpha<d+2$) with
\begin{equation}\label{W0Wr}
W(0)-W(r) \ge c\int_0^r s^{\alpha-d-1}\rd{s} = cr^{\alpha-d}.
\end{equation}

Then we separate into the following cases:
\begin{itemize}
\item If $W(0)=\infty$, then $(W*\rho)(\bx_0)=\infty$, while $(W*\rho)(\bx)<\infty$ nearby. Then $\bx_0$ is not a local minimum of $W*\rho$, contradicting the assumption that $\rho$ is a $d_\infty$-local minimizer, in view of Lemma \ref{lemel1}.
\item If $W(0)<\infty$ and $\nabla \cI(\bx_0)\ne 0$, then for $0<\epsilon_1<\frac{\epsilon}{2|\nabla\cI(\bx_0)|}$, Taylor expansion of $\cI$ with \eqref{W0Wr} gives
\begin{equation}\begin{split}
& (W*\rho)(\bx_0-\epsilon_1\nabla\cI(\bx_0))-(W*\rho)(\bx_0)\\
= & a\Big(W(\epsilon_1|\nabla\cI(\bx_0)|)-W(0)\Big) + \nabla\cI(\bx_0)\cdot (-\epsilon_1\nabla\cI(\bx_0)) + O(|\epsilon_1\nabla\cI(\bx_0)|^2) \\
\le & -\epsilon_1|\nabla\cI(\bx_0)|^2(1+O(\epsilon_1)) < 0,
\end{split}\end{equation}
if $\epsilon_1$ is small enough, and we get a similar contradiction.
\item If $W(0)<\infty$ and $\nabla \cI(\bx_0)= 0$, then for $\bx\in B(\bx_0;\epsilon/2)$, Taylor expansion of $\cI$ with \eqref{W0Wr}  gives
\begin{equation}\begin{split}
(W*\rho)(\bx)-(W*\rho)(\bx_0) = & a\Big(W(|\bx-\bx_0|)-W(0)\Big) + O(|\bx-\bx_0|^2) \\
\le & -ac|\bx-\bx_0|^{\alpha-d} + O(|\bx-\bx_0|^2) < 0
\end{split}\end{equation}
if $|\bx-\bx_0|\ne 0$ is small enough, and we get a similar contradiction since $\alpha <d+2$.
\end{itemize}
\end{proof}

Finally, we can state our main result of this section. We say that $\rho\in\PP(\Rd)$ is \emph{almost fractal} if the support of $\rho$ does not contain isolated points and the interior of any superlevel set is empty.

\begin{corollary}\label{cor_frac}
Given a potential of the form \eqref{Wcons} with parameters satisfying \eqref{condfrac}, $c(d,\alpha)<c_W<C(d,\alpha)$, and $\lambda$ is sufficiently small, then any compactly supported $d_\infty$-local minimizer $\rho$ is almost fractal.
\end{corollary}

\begin{proof}
If \eqref{condfrac} holds, then there exists $c_W$ with $c(d,\alpha)<c_W<C(d,\alpha)$. Now, we apply Theorem \ref{thm_cons1}, Theorem \ref{lem_concave}, and Proposition \ref{profisolated} to conclude that the support of any compactly supported $d_\infty$-local minimizer $\rho$ does not contain isolated points and the superlevel sets do not contain any interior point.
\end{proof}


\section{Cantor set as steady state}

In the 1D case, we construct a potential $W$ such that the uniform distribution on some Cantor set is a steady state, satisfying the necessary condition \eqref{W2cond} for $d_2$-local minimizers. The main idea of this construction is to mimic the  structure of the potential as appeared in \eqref{Wcons} in a recursive hierarchical manner. We will define $W$ as a piecewise-quadratic potential so that it would be easier to verify its steady states compared to \eqref{Wcons}. The main strategy is to produce a potential that introduces some kind of concavity at a sequence of small scales.

\subsection{Steady state}

We will use a hierarchical construction of an interaction potential $W$, with a fixed positive number $M>2$ being the size ratio between adjacent layers, in correspondence with $\lambda^{-1}$ in the previous section. For notation convenience, we denote
\begin{equation}
a_{(j)}:=a M^{-j}
\end{equation}
for $a>0$ and $j\in\mathbb{Z}$.

We first define a uniform measure supported on a Cantor set inside $[0,1]$. Define the intervals $I_{k,l},\,k=0,1,\dots,\,l = 0,1,\dots,2^k-1$ iteratively by
\begin{equation}
I_{0,0} = [0,1],\quad I_{k+1,2l} = I_{k,l}^{\text{left}},\,I_{k+1,2l+1} = I_{k,l}^{\text{right}}
\end{equation}
where for an interval $I=[x_1,x_2]$,
\begin{equation}
I^{\text{left}} = [x_1,x_1+\frac{x_2-x_1}{M}],\quad I^{\text{right}} = [x_2-\frac{x_2-x_1}{M},x_2]\,.
\end{equation}
Then define the functions
\begin{equation}\label{rhok}
\rho_k = \Big(\frac{M}{2}\Big)^k \sum_{l=0}^{2^k-1} \chi_{I_{k,l}}
\end{equation}
which has total mass 1 and supported on $S_k=\bigcup_{l=0}^{2^k-1}I_{k,l}$. The weak limit of $\rho_k$, denoted as $\rho_\infty$, is the uniform distribution on a Cantor set $S=\bigcap_{k=0}^\infty S_k$.

We also introduce the following notation: for $A,B\subseteq \mathbb{R}$,
\begin{equation}
|A-B| = \{y\in\mathbb{R}:y=|x_1-x_2| \mbox{ for some }x_1\in A,\,x_2\in B\}\,.
\end{equation}

We first prove the following lemma, which shows that the possible pairwise distances of points in $S_k$ have a hierarchical structure.
\begin{lemma}\label{lem_x12}
Assume $M>3$. For any $k\ge 0$ and $l_1,l_2\in \{0,1,\dots,2^k-1\}$, $|I_{k,l_1}-I_{k,l_2}|$ is a subset of one of the following disjoint sets: $[0,1_{(k)}],\, [(M-2)_{(k)}, M_{(k)}],\,\dots,\,[(M-2)_{(1)}, M_{(1)}]$, and it is a subset of $[0,1_{(k)}]$ if and only if $l_1=l_2$.
\end{lemma}

\begin{proof}
First notice that $M>3$ implies that the intervals $[0,1_{(k)}],\, [(M-2)_{(k)}, M_{(k)}],\,\dots,\,[(M-2)_{(1)}, M_{(1)}]$ are disjoint. 

The situation $l_1=l_2$ is clear. For $l_1\ne l_2$, we use induction on $k$. The case $k=0$ is clear because the condition $l_1\ne l_2$ never happens. 

Suppose the conclusion is true for $k-1$. For $x_1\in I_{k,l_1}$ and $x_2\in I_{k,l_2}$, there holds
\begin{equation}
I_{k,l_1} \subseteq I_{k-1,\floor{l_1/2}},\quad I_{k,l_2} \subseteq I_{k-1,\floor{l_2/2}}\,.
\end{equation}
If $\floor{l_1/2}\ne \floor{l_2/2}$, then $|x_1-x_2| \in [(M-2)_{(j)}, M_{(j)}]$ for some integer $1\le j \le k-1$ depending on $\floor{l_1/2},\floor{l_2/2},k-1$ by the induction hypothesis, and this implies the conclusion. If $\floor{l_1/2}= \floor{l_2/2}=l$ then $I_{k,l_1}=I_{k-1,l}^{\text{left}}$ and $I_{k,l_2}=I_{k-1,l}^{\text{right}}$ (or the other way around). Then
\begin{equation}
|x_1-x_2| \le |I_{k-1,l}| = M^{-(k-1)} = M_{(k)}
\end{equation}
and 
\begin{equation}
|x_1-x_2| \ge \dist(I_{k-1,l}^{\text{left}},I_{k-1,l}^{\text{right}}) = (M-2)_{(k)}.
\end{equation}
This finishes the proof.
\end{proof}

Then we will define a sequence of potentials $W_k,\,k\ge 0$. In this subsection, we only partially  define them by requiring that $W_k$ is even and
\begin{equation}\label{Wk}
W_k'(x) =  \frac{M}{2M-4}(-\sgn(x) + 2x) + \left\{\begin{split} & 0,\quad |x|\le 1_{(k)} \\
& a_j\sgn(x) + b_jx,\quad (M-2)_{(j)} \le |x| \le  M_{(j)},\,j=1,\dots,k \\
& \text{smoothly connected},\quad \text{otherwise}\end{split}\right.
\end{equation}
where $a_j,b_j$ are constants to be determined. The above lemma shows that the unspecified `smoothly connected' part does not affect $(W_k'*\rho_k)(x),\,(W_k''*\rho_k)(x),\,x\in S_k=\supp\rho_k$, and therefore does not affect whether $\rho_k$ is a steady state of $W_k$. The constant in front of the Newtonian potential part in \eqref{Wk} is chosen for convenience to simplify computations in the next result.

We introduce the notation
\begin{equation}
\bar{\chi}_S(x) = \chi_S(|x|),
\end{equation}
for any $S\subseteq [0,\infty)$.

\begin{theorem}\label{thm_Wkss}
Assume $M>3$. Choose $a_j,b_j$ by 
\begin{equation}
a_j = -\Big(M-\frac{1}{2}\Big),\quad b_j=M^j,\quad j=1,\dots,k,
\end{equation}
then $\rho_k$ is a steady state of the interaction energy with the interaction potential $W_k$ satisfying \eqref{Wk}.
\end{theorem}

\begin{proof}
We proof by induction on $k$. The case $k=0$ is clear, since $W_0(x)=\frac{M}{2M-4}(-|x| + x^2)$ and $\rho_0=\chi_{[0,1]}$ is the well-known steady state for the Newtonian repulsion with quadratic attraction in one dimension. 

Suppose the conclusion holds for $k-1$. For $x\in I_{k,2l_0}$ for some $l_0$, by Lemma \ref{lem_x12}, any $y\in \supp\rho_k=S_k$ satisfies either $|x-y|\le 1_{(k)}$, or $|x-y|\in [(M-2)_{(j)},M_{(j)}]$ for some $1\le j \le k$. Therefore the `smoothly connected' part in \eqref{Wk} makes no contribution to $(W_k'*\rho_k)(x)$ or $(W_k''*\rho_k)(x)$, and there holds
\begin{equation}
(W_k'*\rho_k)(x) = \frac{M}{2M-4}\Big((-\sgn(\cdot) + 2(\cdot))*\rho_k\Big)(x) + \sum_{j=1}^k \Big(\Big(\bar{\chi}_{[(M-2)_{(j)},M_{(j)}]}\cdot (a_j\sgn(\cdot) + b_j(\cdot))\Big) * \rho_k\Big)(x)
\end{equation}
and
\begin{equation}\begin{split}
(W_k''*\rho_k)(x) = & \frac{M}{2M-4}\Big((-2\delta_0 + 2)*\rho_k\Big)(x) + \sum_{j=1}^k b_j\Big(\bar{\chi}_{[(M-2)_{(j)},M_{(j)}]}* \rho_k\Big)(x) \\
= & \frac{M}{2M-4}(-2\rho_k(x)+2) + \sum_{j=1}^k b_j\Big(\bar{\chi}_{[(M-2)_{(j)},M_{(j)}]}* \rho_k\Big)(x)
\end{split}\end{equation}
where $\delta_0$ denotes the Dirac delta function. To show the conclusion for $k$, i.e., $\rho_k$ is a steady state for $W_k$, it is equivalent to show that $W_k'*\rho_k$ vanishes in each $I_{k,2l_0}$ since the same result for $I_{k,2l_0+1}$ can be handled by symmetry. This is equivalent to show
\begin{equation}\label{Wkcond}
(W_k''*\rho_k)(x) = 0,\,\forall x\in I_{k,2l_0},\quad (W_k'*\rho_k)(x_1) = 0
\end{equation}
where $x_1$ denotes the left endpoint of $I_{k,2l_0}$. The induction hypothesis shows that \eqref{Wkcond} is true when $k$ is replaced by $k-1$ and $2l_0$ replaced by any $l$.

\

{\bf STEP 1}: show $(W_k''*\rho_k)(x) = 0,\,\forall x\in I_{k,2l_0}$. 

Since $(W_{k-1}''*\rho_{k-1})(x) = 0,\,\forall x\in I_{k-1,l_0}$ and $I_{k,2l_0}\subseteq I_{k-1,l_0}$, it suffices to show 
\begin{equation}\label{ddWkrho}\begin{split}
0 = & (W_k''*\rho_k)(x) - (W_{k-1}''*\rho_{k-1})(x) \\
= & \frac{M}{2M-4}(-2\rho_k(x) + 2\rho_{k-1}(x)) + \sum_{j=1}^k b_j\Big(\bar{\chi}_{[(M-2)_{(j)},M_{(j)}]}* \rho_k\Big)(x) -  \sum_{j=1}^{k-1} b_j\Big(\bar{\chi}_{[(M-2)_{(j)},M_{(j)}]}* \rho_{k-1}\Big)(x)\\
= & \frac{2M}{2M-4}\Big(-\Big(\frac{M}{2}\Big)^k + \Big(\frac{M}{2}\Big)^{k-1}\Big) +  b_k\Big(\bar{\chi}_{[(M-2)_{(k)},M_{(k)}]}* \rho_k\Big)(x) \\
& +  \sum_{j=1}^{k-1} b_j\Big(\bar{\chi}_{[(M-2)_{(j)},M_{(j)}]}* (\rho_k-\rho_{k-1})\Big)(x)\\
= & -\Big(\frac{M}{2}\Big)^k +  b_k\Big(\bar{\chi}_{[(M-2)_{(k)},M_{(k)}]}* \rho_k\Big)(x) +  \sum_{j=1}^{k-1} b_j\Big(\bar{\chi}_{[(M-2)_{(j)},M_{(j)}]}* (\rho_k-\rho_{k-1})\Big)(x).
\end{split}\end{equation}

Recall that $\supp\rho_k=S_k=\bigcup_{l=0}^{2^k-1}I_{k,l}$. Notice that if $l\ne 2l_0,2l_0+1$, we have
$$
\dist(I_{k,2l_0},I_{k,l}) \ge \dist(I_{k-1,l_0},I_{k-1,\floor{l/2}}) \ge (M-2)_{(k-1)} > 1_{(k-1)}= M_{(k)},
$$ 
and $|I_{k,2l_0}| = 1_{(k)}<(M-2)_{(k)}$. Therefore, in the convolution $\Big(\bar{\chi}_{[(M-2)_{(k)},M_{(k)}]}* \rho_k\Big)(x)$, the only contribution comes from $\rho_k|_{ I_{k,2l_0+1}}=(\frac{M}{2})^k\chi_{I_{k,2l_0+1}}$, which gives
\begin{equation}\begin{split}
\Big(\bar{\chi}_{[(M-2)_{(k)},M_{(k)}]}* \rho_k\Big)(x) = \Big(\frac{M}{2}\Big)^k\Big(\bar{\chi}_{[(M-2)_{(k)},M_{(k)}]}* \chi_{I_{k,2l_0+1}}\Big)(x) 
=  \Big(\frac{M}{2}\Big)^k |\chi_{I_{k,2l_0+1}}| = 2^{-k},
\end{split}\end{equation}
where the second equality follows from the fact that $|x-y|\in [(M-2)_{(k)},M_{(k)}]$ for any $x\in I_{k,2l_0}$ and $y\in I_{k,2l_0+1}$.

Next we will show that the last summation in \eqref{ddWkrho} vanishes. In fact, notice that the definition of $\rho_k$ in \eqref{rhok} implies that
\begin{equation}
\supp(\rho_k-\rho_{k-1}) = \supp\rho_{k-1} = \bigcup_{l=0}^{2^{k-1}-1} I_{k-1,l}
\end{equation}
and we have the zeroth and first moment conservation conditions in each $I_{k-1,l}$ as the following:
\begin{equation}\label{rhok1}
\int_{I_{k-1,l}} (\rho_k-\rho_{k-1})\rd{x} = \int_{I_{k-1,l}} (\rho_k-\rho_{k-1})x\rd{x} = 0,\quad l=0,1,\dots,2^{k-1}-1.
\end{equation}
Lemma \ref{lem_x12} gives that $|I_{k-1,l_0}-I_{k-1,l}|$ is a subset of one of the following disjoint sets: $[0,1_{(k-1)}],\, [(M-2)_{(k-1)}, M_{(k-1)}],\,\dots,\,[(M-2)_{(1)}, M_{(1)}]$. Using the zeroth moment conservation, we get 
\begin{equation}
\Big(\bar{\chi}_{[(M-2)_{(j)},M_{(j)}]}* (\rho_k-\rho_{k-1})|_{I_{k-1,l}}\Big)(x) = 0, \quad x\in I_{k,2l_0},
\end{equation}
for any $j=1,\dots,k-1$, and thus
\begin{equation}
 \sum_{j=1}^{k-1} b_j\Big(\bar{\chi}_{[(M-2)_{(j)},M_{(j)}]}* (\rho_k-\rho_{k-1})\Big)(x) = 0, \quad x\in I_{k,2l_0}.
\end{equation}
Therefore, we conclude that the RHS of \eqref{ddWkrho} is zero with the condition $b_k=M^k$.

\

{\bf STEP 2}: show $(W_k'*\rho_k)(x_1) = 0$. 

Similar to the previous step, it suffices to show 
\begin{equation}\label{Wk1est}\begin{split}
0 = & (W_k'*\rho_k)(x_1) - (W_{k-1}'*\rho_{k-1})(x_1) \\
= & \frac{M}{2M-4}\Big((-\sgn(x) + 2x)*(\rho_k-\rho_{k-1})\Big)(x_1)  + \sum_{j=1}^k \Big(\Big(\bar{\chi}_{[(M-2)_{(j)},M_{(j)}]}\cdot (a_j\sgn(x) + b_jx)\Big) * \rho_k\Big)(x_1) \\
& -  \sum_{j=1}^{k-1} \Big(\Big(\bar{\chi}_{[(M-2)_{(j)},M_{(j)}]}\cdot (a_j\sgn(x) + b_jx)\Big) * \rho_{k-1}\Big)(x_1) \\
= & \frac{M}{2M-4}\Big((-\sgn(x) + 2x)*(\rho_k-\rho_{k-1})\Big)(x_1)  + \Big(\Big(\bar{\chi}_{[(M-2)_{(k)},M_{(k)}]}\cdot (a_k\sgn(x) + b_kx)\Big) * \rho_k\Big)(x_1) \\
& +  \sum_{j=1}^{k-1} \Big(\Big(\bar{\chi}_{[(M-2)_{(j)},M_{(j)}]}\cdot (a_j\sgn(x) + b_jx)\Big) * (\rho_k-\rho_{k-1})\Big)(x_1) \\
\end{split}\end{equation}

First, by \eqref{rhok1}, $x*(\rho_k-\rho_{k-1}) = 0$. Therefore, by further using the moment conservation property \eqref{rhok1},
\begin{equation}\begin{split}
& \Big((-\sgn(x) + 2x)*(\rho_k-\rho_{k-1})\Big)(x_1) =  \Big((-\sgn(x))*(\rho_k-\rho_{k-1})\Big)(x_1) \\
= & -\Big(2\chi_{[0,\infty)}*(\rho_k-\rho_{k-1})\Big)(x_1) = 2\int_{-\infty}^{x_1}(\rho_k-\rho_{k-1})\rd{x} = 0,
\end{split}\end{equation}
since $x_1$, as the left endpoint of $I_{k,2l_0}$, is also the left endpoint of $I_{k-1,l_0}$.

Next, in the second quantity on the RHS of \eqref{Wk1est}, similar to STEP 1, the only contribution comes from $\rho_k|_{I_{k,2l_0+1}}$, and this term can be computed by
\begin{equation}\begin{split}
& \Big(\Big(\bar{\chi}_{[(M-2)_{(k)},M_{(k)}]}\cdot (a_k\sgn(x) + b_kx)\Big) * \rho_k\Big)(x_1) = \int_{I_{k,2l_0+1}}(a_k\sgn(x_1-x) + b_k(x_1-x))\rho_k(x)\rd{x} \\
= & -\Big(\frac{M}{2}\Big)^k \Big(a_k|I_{k,2l_0+1}| + b_k|I_{k,2l_0+1}| \cdot \Big(M-\frac{1}{2}\Big)|I_{k,2l_0+1}|\Big) \\
= & -2^{-k}\Big(a_k + b_k \Big(M-\frac{1}{2}\Big) M^{-k}\Big)\,.
\end{split}\end{equation}

Finally, the last quantity on the RHS of \eqref{Wk1est} is zero by \eqref{rhok1}, similar as before.

Therefore we conclude that the RHS of \eqref{Wk1est} is zero with the condition
\begin{equation}
a_k = -b_k \Big(M-\frac{1}{2}\Big) M^{-k} = -\Big(M-\frac{1}{2}\Big).
\end{equation}

\end{proof}

\subsection{The condition \eqref{W2cond} for $d_2$-local minimizer}

In this section, we specify our potential by defining the choice of `smooth connections' in \eqref{Wk} as
\begin{equation}
W_k'(x) = \frac{M}{2M-4} (-1 + 2x) + \omega_k(x) ,\quad x>0
\end{equation}
with
\begin{equation}\begin{split}
\omega_k(x)= \left\{\begin{array}{ll}  0, & 0<x\le (M-\alpha)_{(k)} \\[3pt]
(-\frac{3}{2}+\frac{3}{4}\cdot\frac{2}{\alpha-2}(M-2))  - \frac{3}{4}\cdot\frac{2}{\alpha-2}M^k x , & (M-\alpha)_{(k)} \le x \le (M-2)_{(k)} \\[3pt]
 -(M-\frac{1}{2})  + M^j x, & (M-2)_{(j)} \le x \le M_{(j)},\,j=1,\dots,k \\[3pt]
\frac{1}{2} , & 1_{(j)} \le x \le (M-\alpha)_{(j)},\,j=1,\dots,k-1 \\[3pt]
(-\frac{3}{2}+\frac{2}{\alpha-2}(M-2))  - \frac{2}{\alpha-2}M^j x , & (M-\alpha)_{(j)} \le x \le (M-2)_{(j)},\, j=1,\dots,k-1 \\[3pt]
\frac{1}{2}, & x\ge M_{(1)}=1 \\[3pt]
\end{array}\right.
\end{split}\end{equation}
where $2< \alpha \le M-1$ is a parameter to be determined. The smooth connections $\omega_k$ is extended as an odd function on $\mathbb{R}$, in correspondence to $W_k'$. Moreover, $\omega_k$ is a continuous, piecewise linear function, and 
\begin{equation}\begin{split}
\omega_k'(x)= \left\{\begin{array}{ll}  0, & 0<x\le (M-\alpha)_{(k)} \\[3pt]
- \frac{3}{4}\cdot\frac{2}{\alpha-2}M^k, & (M-\alpha)_{(k)} \le x \le (M-2)_{(k)} \\[3pt]
M^j, & (M-2)_{(j)} \le x \le M_{(j)},\,j=1,\dots,k \\[3pt]
0, & 1_{(j)} \le x \le (M-\alpha)_{(j)},\,j=1,\dots,k-1 \\[3pt]
- \frac{2}{\alpha-2}M^j, & (M-\alpha)_{(j)} \le x \le (M-2)_{(j)},\, j=1,\dots,k-1 \\[3pt]
0, & x\ge M_{(1)}=1 \\[3pt]
\end{array}\right.
\end{split}\end{equation}

Notice that for $k\ge 2$,
\begin{equation}\label{omega}\begin{split}
W_k''(x) & -W_{k-1}''(x) =  \omega_k'(x)-\omega_{k-1}'(x) = \left\{\begin{array}{ll}  
-\frac{3}{4}\cdot\frac{2}{\alpha-2}M^k  , & (M-\alpha)_{(k)} \le x \le (M-2)_{(k)} \\[3pt]
M^k , & (M-2)_{(k)} \le x \le M_{(k)} \\[3pt]
-\frac{1}{4}\cdot\frac{2}{\alpha-2}M^{k-1}  , & (M-\alpha)_{(k-1)} \le x \le (M-2)_{(k-1)} \\[3pt]
0, & \text{otherwise} \\[3pt]
\end{array}\right. \\
= & M^k \bar{\chi}_{[(M-2)_{(k)},M_{(k)}]} -\frac{1}{4}\cdot\frac{2}{\alpha-2}M^{k-1}\bar{\chi}_{ [(M-\alpha)_{(k-1)} , (M-2)_{(k-1)}]}  -\frac{3}{4}\cdot\frac{2}{\alpha-2}M^k \bar{\chi}_{[(M-\alpha)_{(k)} , (M-2)_{(k)}]}
\end{split}\end{equation}
is compactly supported on $\{(M-\alpha)_{(k)} \le |x| \le (M-2)_{(k-1)})\}$ with mean-zero on it.
See Figure \ref{fig_omegak} for an illustration.

\begin{figure}
\begin{center}
	\includegraphics[width=0.8\textwidth]{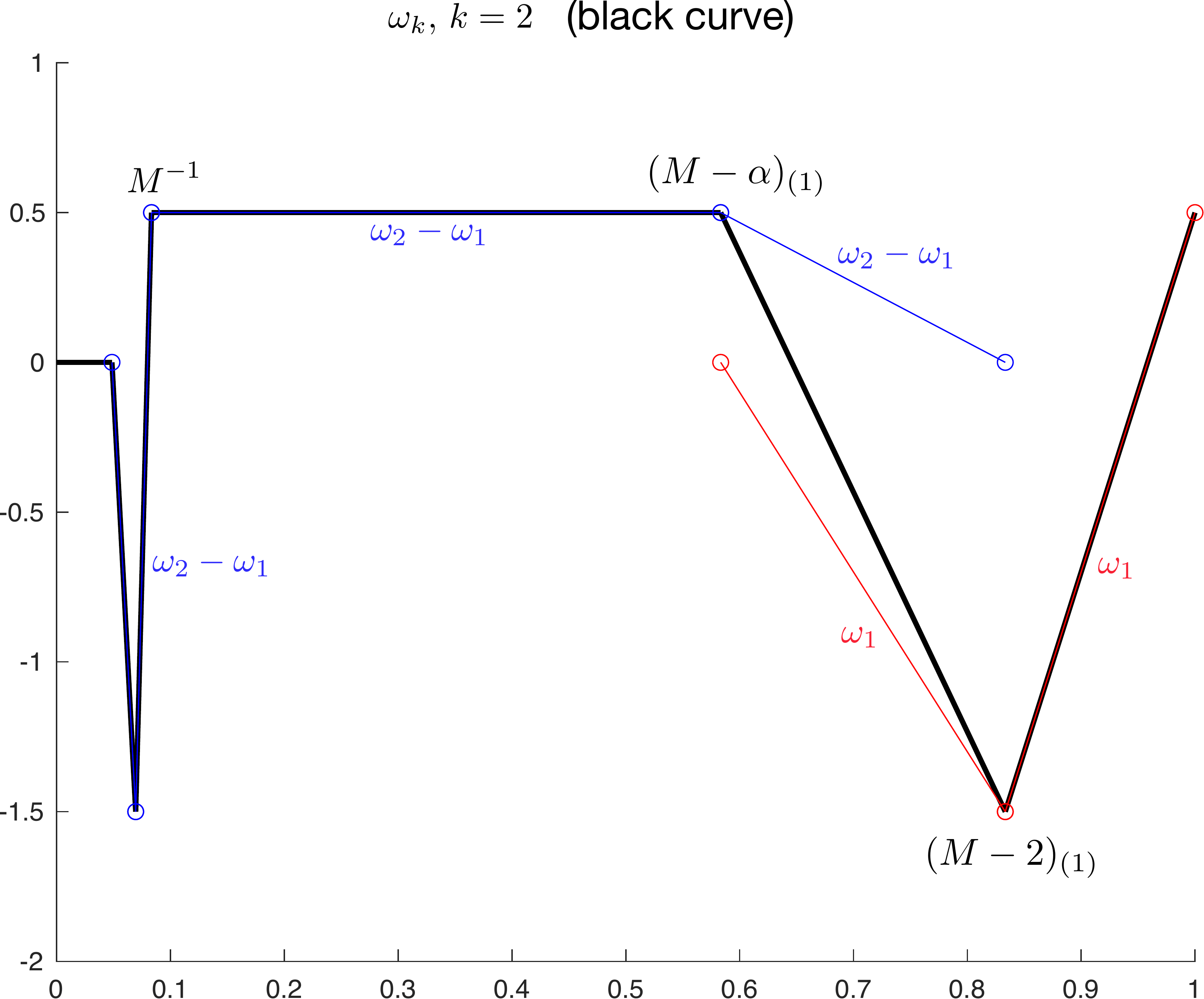}
	\caption{Decomposition of $\omega_k(x)=\omega_1(x) + \sum_{j=2}^k (\omega_j(x)-\omega_{j-1}(x))$, in the case $k=2$.}
	\label{fig_omegak}
\end{center}
\end{figure}

We first prove the necessary condition \eqref{W2cond} for $d_2$-local minimizers for $W_K$ and $\rho_K$ and those $x_0$ in the interval $(M^{-1},1-M^{-1})$, which corresponds to the inner interval of Figure \ref{fig_Wcantor}. This is the first step in a self-similar argument for the full condition \eqref{W2cond} for $\rho_K$.
\begin{proposition}\label{prop_M}
Let $W_K$ be given as above, and $\rho_K$ as given in \eqref{rhok}. If $M$ and $\alpha$ satisfy
\begin{equation}\label{prop_M_1}
\frac{1}{3}(M+2) < \alpha \le \frac{2}{5}(M-10)
\end{equation}
then for any integer $K \ge 1$,
\begin{equation}\label{prop_M_2}
(W_K*\rho_K )(x_0) - (W_K*\rho_K )(M^{-1}) \ge c(x_0)>0,\quad \forall x_0\in (M^{-1},1-M^{-1})\,.
\end{equation}
\end{proposition}

Notice that the possible choice of $\alpha$ is nonempty if $M$ is large enough. For example, $M=100$, $\alpha=35$.

We first give the following lemma for the convolution of a characteristic function with a compactly supported measure. Its proof is straightforward and thus omitted.
\begin{lemma}\label{lem_convo}
Let $\mu(x)$ be a non-negative measure supported on $I_1 = [a_1,b_1]$ and assume $\mu$ is symmetric around $(a_1+b_1)/2$. For another interval $I_2=[a_2,b_2]$ such that $a_2>0$ and $|I_2|\ge |I_1|$, $\mu*\chi_{I_2}$ is a non-negative function supported on $[a_1+a_2,b_1+b_2]$, given by
\begin{equation}
(\mu*\chi_{I_2})(x) = \left\{\begin{split} & \psi(x-(a_1+a_2)),\quad a_1+a_2 \le x \le b_1+a_2 \\
& |\mu|,\quad b_1+a_2 \le x \le a_1+b_2 \\
& \psi((b_1+b_2)-x),\quad a_1+b_2 \le x \le b_1+b_2
\end{split}\right.
\end{equation}
where $\psi(y):=\int_{a_1}^{a_1+y} \mu(y_1)\rd{y_1}$, $|\mu| = \psi(b_1-a_1)$. 
\end{lemma}
The three pieces of the above function $\mu*\chi_{I_2}$ will be referred to as piece 1, piece 2 and piece 3 respectively.

\begin{proof}[Proof of Proposition \ref{prop_M}]


Since $\rho_K$ is a steady state for $W_K$ by Theorem \ref{thm_Wkss}, we have $(W_K'*\rho_K)(M^{-1})=0$ since $M^{-1}\in\supp\rho_K$. Then we write
\begin{equation}\label{Wx0m1}\begin{split}
(W_K*\rho_K )(x_0)-(W_K*\rho_K )(M^{-1}) = & \int_{M^{-1}}^{x_0} (W_K'*\rho_K )(x_1)\rd{x_1} = \int_{M^{-1}}^{x_0} \int_{M^{-1}}^{x_1}(W_K''*\rho_K )(x)\rd{x}\rd{x_1} \\
= & \int_{M^{-1}}^{x_0} (x_0-x)(W_K''*\rho_K )(x)\rd{x}
\end{split}\end{equation}
and for $x\in (M^{-1},1/2)$ which is outside $\supp\rho_K $,
\begin{equation}\label{eq:decomp}\begin{split}
(W_K''*\rho_K )(x) = & (W_1''*\rho_K )(x) + ((\omega_2'-\omega_1')*\rho_K )(x) + \sum_{k=3}^K ((\omega_k'-\omega_{k-1}')*\rho_K )(x).
\end{split}\end{equation}
We deal with each of these three terms in the next three steps starting from the last one.

\

{\bf STEP 1}: show that $ \int_{M^{-1}}^{x_0} (x_0-x)((\omega_k'-\omega_{k-1}')*\rho_K )(x)\rd{x}>0$ for any $3\le k \le K$.

We first decompose $\rho_K $ as
\begin{equation}
\rho_K  = \rho_K \chi_{I_{1,1}} + \sum_{j=2}^{k-1} \rho_K  \chi_{I_{j,2^{j-1}-2}} + \rho_K  \chi_{I_{k-1,2^{k-2}-1}}\,.
\end{equation}
Notice that for $x\in [M^{-1},\frac{1}{2}]$, we have $\dist(I_{1,1},x) \ge \frac{1}{2}-M^{-1}$, $\dist(I_{j,2^{j-1}-2},x) \ge (M-1)_{(j)},\,2\le j\le k-1$,  and $\supp(\omega_k'-\omega_{k-1}')\cap\mathbb{R}_+ \subseteq [(M-\alpha)_{(k)}, (M-2)_{(k-1)}]$. Therefore, we can check that the support of the density parts translated by $x$ are to the right of the support of $\omega_k'-\omega_{k-1}'$, leading to 
\begin{equation}
(\rho_K \chi_{I_{1,1}})*(\omega_k'-\omega_{k-1}'))(x) = (\rho_K \chi_{I_{j,2^{j-1}-2}})*(\omega_k'-\omega_{k-1}'))(x) =  0,\quad 2\le j \le k-1\,.
\end{equation}
Then we further decompose
\begin{equation}
\rho_K  \chi_{I_{k-1,2^{k-2}-1}} = \rho_K  \chi_{I_{k,2^{k-1}-2}} + \rho_K  \chi_{I_{k,2^{k-1}-1}}
\end{equation}
and then the term with $\bar{\chi}_{[(M-\alpha)_{(k)} , (M-2)_{(k)}]}$ in \eqref{omega} does not interact with $\rho_K  \chi_{I_{k,2^{k-1}-2}}$ for the same reason as above. Therefore
\begin{equation}\begin{split}
(\omega_k'-\omega_{k-1}')*\rho_K  = & M^k \bar{\chi}_{[(M-2)_{(k)},M_{(k)}]} * (\rho_K  \chi_{I_{k,2^{k-1}-2}})  \\
& -\frac{1}{4}\cdot\frac{2}{\alpha-2}M^{k-1}\bar{\chi}_{[(M-\alpha)_{(k-1)} , (M-2)_{(k-1)}]} * (\rho_K  \chi_{I_{k,2^{k-1}-2}}) \\
& + M^k \bar{\chi}_{[(M-2)_{(k)},M_{(k)}]} * (\rho_K  \chi_{I_{k,2^{k-1}-1}}) \\
& -\frac{3}{4}\cdot\frac{2}{\alpha-2}M^k \bar{\chi}_{[(M-\alpha)_{(k)} , (M-2)_{(k)}]} * (\rho_K  \chi_{I_{k,2^{k-1}-1}}) \\
&  -\frac{1}{4}\cdot\frac{2}{\alpha-2}M^{k-1}\bar{\chi}_{[(M-\alpha)_{(k-1)} , (M-2)_{(k-1)}]}* (\rho_K  \chi_{I_{k,2^{k-1}-1}})\\
:= & \cJ_1 + \cJ_2 + \cJ_3 + \cJ_4 + \cJ_5\,.
\end{split}\end{equation}
See Figure \ref{fig_J} for an illustration. 

\begin{figure}
\begin{center}
	\includegraphics[width=0.48\textwidth]{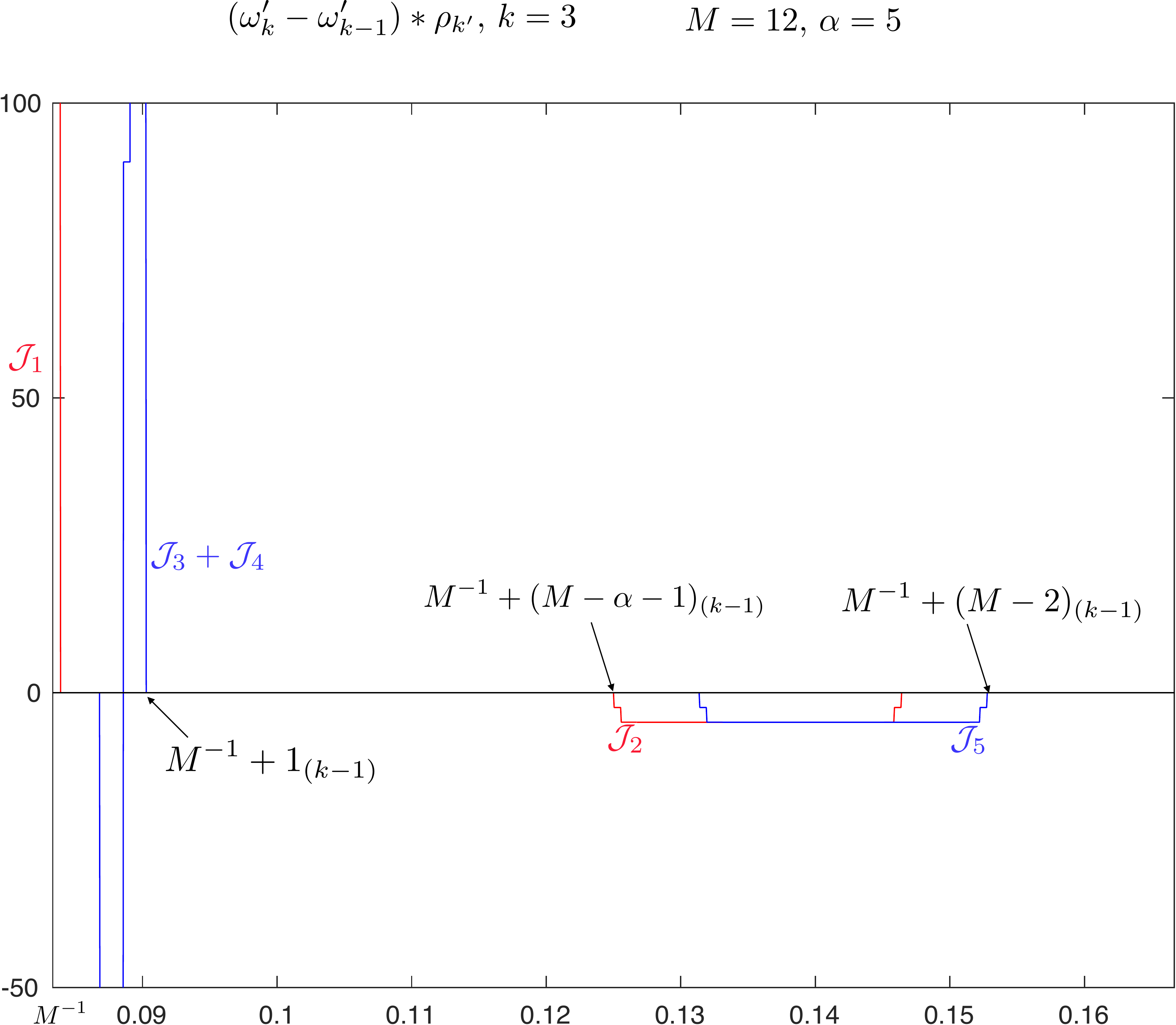}
	\includegraphics[width=0.48\textwidth]{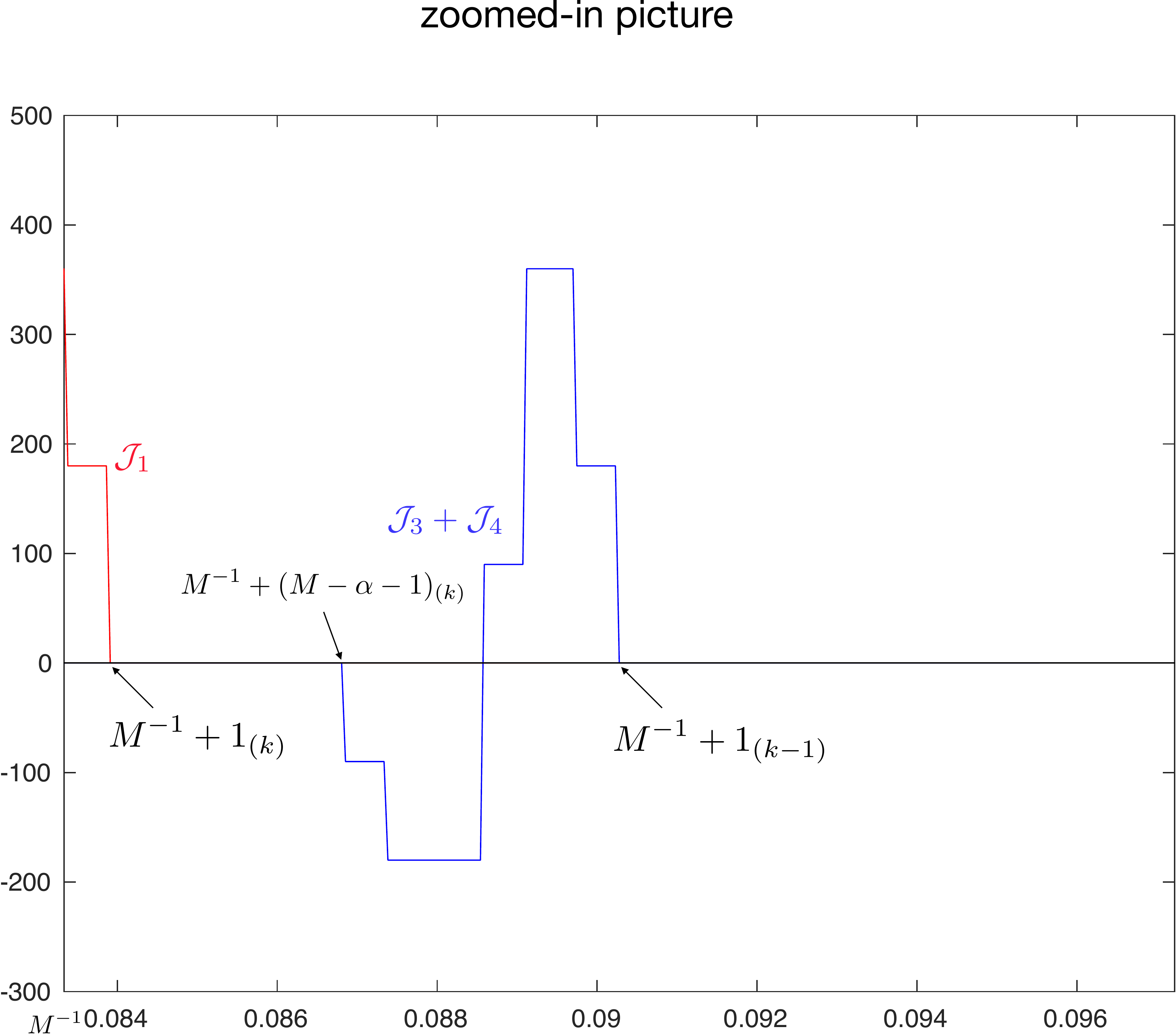}
	\caption{The decomposition of $(\omega_k'-\omega_{k-1}')*\rho_K $ into $\cJ_1,\dots,\cJ_5$.}
	\label{fig_J}
\end{center}
\end{figure}

Recall that we focus on $x\in [M^{-1},1/2]$, and the involved interval for $rho_K$ are $I_{k,2^{k-1}-2} = [M^{-1}-1_{(k-1)},M^{-1}-1_{(k-1)}+1_{(k)}]$ and $I_{k,2^{k-1}-1} = [M^{-1}-1_{(k)},M^{-1}]$.
Due to \eqref{rhok}, we get a localized accumulation function of $\rho_K $ for both intervals
\begin{equation}
\psi(y) = \int_{M^{-1}-1_{(k-1)}}^{M^{-1}-1_{(k-1)}+y}\rho_K (y_1)\rd{y_1} = \int_{M^{-1}-1_{(k)}}^{M^{-1}-1_{(k)}+y}\rho_K (y_1)\rd{y_1},\quad 0\le y \le 1_{(k)}
\end{equation}
as appeared in Lemma \ref{lem_convo}.
By symmetry, $\psi(y) + \psi(1_{(k)}-y)=\psi(1_{(k)}) = 2^{-k}$, which implies $\int_0^{1_{(k)}}\psi(y)\rd{y} = \frac{1}{2}\int_0^{1_{(k)}}(\psi(y)+\psi(1_{(k)}-y))\rd{y}=2^{-(k+1)}\cdot 1_{(k)}$. Then we analyze the five integrals $\cJ_1,\dots,\cJ_5$ separately.

{\bf Positive contribution from $\cJ_1$}:

$\supp\cJ_1 = [M^{-1},M^{-1}+1_{(k)}]$, and its expression only has piece 3.  Therefore
\begin{equation}\label{j1dec}
\cJ_1(x) = M^k\psi(M^{-1}+1_{(k)}-x)
\end{equation}
and
\begin{equation}
\int_{M^{-1}}^{M^{-1}+1_{(k)}}\cJ_1(x)\rd{x} = M^k2^{-(k+1)}\cdot 1_{(k)} = \frac{1}{2}\cdot 2^{-k}.
\end{equation}

{\bf Negative contribution from $\cJ_2$}:

$\supp\cJ_2 =  [M^{-1}+ (M-\alpha-1)_{(k-1)},M^{-1}+ (M-3)_{(k-1)} + 1_{(k)}]$, and its expression has all 3 full pieces. 
\begin{equation}\begin{split}
& \cJ_2(x) = -\frac{1}{4}\cdot\frac{2}{\alpha-2}M^{k-1}\\
& \cdot\left\{\begin{array}{ll} \psi(x-(M^{-1}+ (M-\alpha-1)_{(k-1)})), & M^{-1}+ (M-\alpha-1)_{(k-1)} \le x \le M^{-1}+ (M-\alpha-1)_{(k-1)}+1_{(k)} \\[3pt]
2^{-k}, & M^{-1}+ (M-\alpha-1)_{(k-1)}+1_{(k)} \le x \le M^{-1}+ (M-3)_{(k-1)} \\[3pt]
\psi(M^{-1}+ (M-3)_{(k-1)}+1_{(k)}-x), & M^{-1}+ (M-3)_{(k-1)} \le x \le M^{-1}+ (M-3)_{(k-1)}+1_{(k)}
\end{array}\right.
\end{split}\end{equation}
and
\begin{equation}
\int_{M^{-1}+ (M-\alpha-1)_{(k-1)}}^{M^{-1}+ (M-3)_{(k-1)} + M^{-k}}\cJ_2(y)\rd{y} = -\frac{1}{4}\cdot\frac{2}{\alpha-2}M^{k-1} \cdot 2^{-k}\cdot (\alpha-2)_{(k-1)} = -\frac{1}{2}\cdot 2^{-k}.
\end{equation}

{\bf Positive contribution from $\cJ_3$}:

$\supp\cJ_3 =  [M^{-1}+(M-3)_{(k)},M^{-1}+M_{(k)}]$, and its expression has all 3 full pieces. 
\begin{equation}
\cJ_2(x) =M^k\cdot \left\{\begin{array}{ll}  \psi(x-(M^{-1}+(M-3)_{(k)})), & M^{-1}+(M-3)_{(k)} \le x \le M^{-1}+(M-2)_{(k)} \\[3pt]
2^{-k}, & M^{-1}+(M-2)_{(k)} \le x \le M^{-1}+(M-1)_{(k)} \\[3pt]
\psi(M^{-1}+M_{(k)}-x), & M^{-1}+(M-1)_{(k)} \le x \le M^{-1}+M_{(k)}
\end{array}\right.
\end{equation}
and
\begin{equation}
\int_{M^{-1}+(M-3)_{(k)}}^{M^{-1}+M_{(k)}}\cJ_3(y)\rd{y} = M^k \cdot 2^{-k}\cdot 2_{(k)} = 2\cdot 2^{-k}.
\end{equation}

{\bf Negative contribution from $\cJ_4$}:

$\supp\cJ_4 =  [M^{-1}+(M-\alpha-1)_{(k)},M^{-1}+(M-2)_{(k)}]$, and its expression has all 3 full pieces. 
\begin{equation}\begin{split}
& \cJ_4(x) = -\frac{3}{4}\cdot \frac{2}{\alpha-2}M^k\\
& \cdot\left\{\begin{array}{ll}  \psi(x-(M^{-1}+(M-\alpha-1)_{(k)})), & M^{-1}+(M-\alpha-1)_{(k)} \le x \le M^{-1}+M^{-k}\cdot(M-\alpha) \\[3pt]
2^{-k}, & M^{-1}+M^{-k}\cdot(M-\alpha) \le x \le M^{-1}+M^{-k}\cdot(M-3) \\[3pt]
\psi(M^{-1}+(M-2)_{(k)}-x), & M^{-1}+M^{-k}\cdot(M-3) \le x \le M^{-1}+(M-2)_{(k)}
\end{array}\right.
\end{split}\end{equation}
and
\begin{equation}
\int_{M^{-1}+(M-\alpha-1)_{(k)}}^{M^{-1}+(M-2)_{(k)}}\cJ_4(y)\rd{y} = -\frac{3}{4}\cdot\frac{2}{\alpha-2}M^k \cdot 2^{-k}\cdot (\alpha-2)_{(k)} = -\frac{3}{2}\cdot 2^{-k}.
\end{equation}

{\bf Negative contribution from $\cJ_5$}:

$\supp\cJ_5 =  [M^{-1}+(M-\alpha)_{(k-1)}-1_{(k)},M^{-1}+(M-2)_{(k-1)}]$, and its expression has all 3 full pieces. 
\begin{equation}\begin{split}
& \cJ_5(x) = -\frac{1}{4}\cdot\frac{2}{\alpha-2}M^{k-1} \\
& \cdot \left\{\begin{array}{ll}  \psi(x-(M^{-1}+(M-\alpha)_{(k-1)}-1_{(k)})), & M^{-1}+(M-\alpha)_{(k-1)}-1_{(k)} \le x \le M^{-1}+(M-\alpha)_{(k-1)} \\[3pt]
2^{-k}, & M^{-1}+(M-\alpha)_{(k-1)} \le x \le M^{-1}+(M-2)_{(k-1)}-1_{(k)} \\[3pt]
\psi(M^{-1}+(M-2)_{(k-1)}-x), & M^{-1}+(M-2)_{(k-1)}-1_{(k)} \le x \le M^{-1}+(M-2)_{(k-1)} 
\end{array}\right.
\end{split}\end{equation}
and
\begin{equation}
\int_{M^{-1}+(M-\alpha)_{(k-1)}-1_{(k)}}^{M^{-1}+(M-2)_{(k-1)}}\cJ_5(y)\rd{y} = -\frac{1}{4}\cdot\frac{2}{\alpha-2}M^{k-1} \cdot 2^{-k}\cdot (\alpha-2)_{(k-1)} = -\frac{1}{2}\cdot 2^{-k} .
\end{equation}

{\bf STEP 1-1}: treat the range $M^{-1}\le x_0 \le M^{-1}+1_{(k)}$, and quantify the positive contribution from $\cJ_1$.

The assumption \eqref{prop_M_1} implies that
\begin{equation}\label{alpha1}
M>4,\quad \alpha \le \frac{M}{2}\,.
\end{equation}
Then it is clear that  
\begin{equation}
\supp\cJ_1 \cap \supp \cJ_i = \emptyset,\quad i=2,3,4,5.
\end{equation}
Therefore $ \int_{M^{-1}}^{x_0} (x_0-x)((\omega_k'-\omega_{k-1}')*\rho_K )(x)\rd{x}>0$ for $x_0\in \supp\cJ_1=[M^{-1},M^{-1}+1_{(k)}]$. Also, for $x_0>M^{-1}+1_{(k)}$, we have the positive contribution from $\cJ_1$ as
\begin{equation}\label{J1est}\begin{split}
\int_{M^{-1}}^{x_0} & (x_0-x)\cJ_1(x)\rd{x} 
=  \int_{M^{-1}}^{M^{-1}+1_{(k)}} (x_0-x)\cJ_1(x)\rd{x}\\
= & \int_{M^{-1}}^{M^{-1}+1_{(k)}} (M^{-1}+(\frac{1}{2})_{(k)}-x)\cJ_1(x)\rd{x} + (x_0-M^{-1}-(\frac{1}{2})_{(k)})\int_{M^{-1}}^{M^{-1}+1_{(k)}}\cJ_1(x)\rd{x} \\
= & \int_{M^{-1}}^{M^{-1}+1_{(k)}} (x-M^{-1}-(\frac{1}{2})_{(k)})\cJ_1(2M^{-1}+1_{(k)}-x)\rd{x} + \frac{1}{2}\cdot 2^{-k}\cdot (x_0-M^{-1}-(\frac{1}{2})_{(k)}) \\
= & \frac{1}{2}\int_{M^{-1}}^{M^{-1}+1_{(k)}} (M^{-1}+(\frac{1}{2})_{(k)}-x)(\cJ_1(x)-\cJ_1(2M^{-1}+1_{(k)}-x))\rd{x} + \frac{1}{2}\cdot 2^{-k}\cdot (x_0-M^{-1}-(\frac{1}{2})_{(k)}) \\
\ge & \frac{1}{2}\cdot 2^{-k}\cdot (x_0-M^{-1}-(\frac{1}{2})_{(k)})
\end{split}\end{equation}
where the inequality uses the decreasing property of $\cJ_1$ due to \eqref{j1dec} to get the positivity of the integrand.

{\bf STEP 1-2}: treat the range $M^{-1}+1_{(k)} \le x_0 \le M^{-1}+1_{(k-1)}=M^{-1}+M_{(k)}$.

For such $x_0$, \eqref{J1est} holds, and $x_0\notin \supp\cJ_2\cup\supp\cJ_5$ under the condition \eqref{alpha1}. Then the only negative contribution is from $\cJ_4$, whose support is $\supp\cJ_4 =  [M^{-1}+(M-\alpha-1)_{(k)},M^{-1}+(M-2)_{(k)}]$ and has integral $-\frac{3}{2}\cdot 2^{-k}$. Therefore it suffices to consider $x_0\in[M^{-1}+(M-\alpha-1)_{(k)},M^{-1}+M_{(k)}]$. Using the negativity and symmetry of $\cJ_4$, we have
\begin{equation}\begin{split}
\int_{M^{-1}}^{x_0}  (x_0-x)\cJ_4(x)\rd{x} \ge & \int_{M^{-1}}^{M^{-1}+M_{(k)}}  (M^{-1}+M_{(k)}-x)\cJ_4(x)\rd{x} = -\frac{3}{2}\cdot 2^{-k}(M^{-1}+M_{(k)}-x_c) \\
= & -\frac{3}{2}\cdot 2^{-k}(\frac{\alpha+3}{2})_{(k)}
\end{split}\end{equation}
where $x_c=M^{-1}+(M-\frac{\alpha+3}{2})_{(k)}$ is the center of $\supp \cJ_4$. 
For such $x_0$, \eqref{J1est} gives
\begin{equation}\label{J1int}
\int_{M^{-1}}^{x_0} (x_0-x)\cJ_1(x)\rd{x} \ge \frac{1}{2}\cdot 2^{-k}\cdot (M^{-1}+(M-\alpha-1)_{(k)}-M^{-1}-(\frac{1}{2})_{(k)}) = \frac{1}{2}\cdot 2^{-k}\cdot (M-\alpha-\frac{3}{2})_{(k)}
\end{equation}
Therefore, as long as
\begin{equation}\label{alphacond1}
M-\alpha-\frac{3}{2} \ge \frac{3}{2}(\alpha+3) \quad\Leftrightarrow\quad \alpha \le \frac{2}{5}(M-6)
\end{equation}
which is guaranteed by the assumption \eqref{prop_M_1}, we have $ \int_{M^{-1}}^{x_0} (x_0-x)(\cJ_1(x)+\cJ_4(x))\rd{x}\ge 0$ for $M^{-1}+1_{(k)} \le x_0 \le M^{-1}+1_{(k-1)}$.

{\bf STEP 1-3}: treat the range $M^{-1}+1_{(k-1)} \le x_0 \le 1/2$.

Notice that
\begin{equation}
\frac{\partial}{\partial x_0} \int_{M^{-1}}^{x_0} (x_0-x)\sum_{i=1}^5\cJ_i(x)\rd{x} = \int_{M^{-1}}^{x_0} \sum_{i=1}^5\cJ_i(x)\rd{x}
\end{equation}
For $M^{-1}+1_{(k-1)} \le x_0 \le 1/2$, $[M^{-1},x_0]$ contains all the positive contributions: $\supp\cJ_1\cup\supp\cJ_3$,  therefore
\begin{equation}
\int_{M^{-1}}^{x_0} \sum_{i=1}^5\cJ_i(x)\rd{x} \ge \int_{M^{-1}}^{1/2} \sum_{i=1}^5\cJ_i(x)\rd{x} = 0
\end{equation}
Therefore $\int_{M^{-1}}^{x_0} (x_0-x)\sum_{i=1}^5\cJ_i(x)\rd{x}$ is increasing in $x_0$ for $M^{-1}+1_{(k-1)} \le x_0 \le 1/2$, and its positivity follows from STEP 1-2.

\

{\bf STEP 2}: show that $ \int_{M^{-1}}^{x_0} (x_0-x)(W_1''*\rho_K )(x)\rd{x}>0$.

$\omega_1'$ is supported on $[(M-\alpha)_{(1)},1]\subseteq [1/2,1]$. Therefore for $x\in (M^{-1},1/2)$, the only possible contribution for $\omega_1'*\rho_K $ comes from $\rho_K \chi_{I_{1,1}}$, and then
\begin{equation}\begin{split}
W_1''*\rho_K  = & \frac{M}{M-2} + \omega_1'*(\rho_K \chi_{I_{1,1}})\\
= & \frac{M}{M-2} + M\bar{\chi}_{[(M-2)_{(1)},1]}*(\rho_K \chi_{I_{1,1}}) - \frac{3}{4}\cdot\frac{2}{\alpha-2}M\bar{\chi}_{[(M-\alpha)_{(1)},(M-2)_{(1)}]}*(\rho_K \chi_{I_{1,1}}) \\
= & \frac{M}{M-2} + \cK_1 + \cK_2.
\end{split}\end{equation}

Applying Lemma \ref{lem_convo} with $x$-axis reversed, $\cK_1\ge0$ is supported on $[M^{-1},2M^{-1}]$ (after intersecting with $[M^{-1},1/2]$), and only has piece 1, with expression given by
\begin{equation}
\cK_1(x) = M\psi(2M^{-1}-x),\quad \psi(y) = \int_0^y \rho_K (y_1)\rd{y_1}.
\end{equation}
By the symmetry $\psi(y)+\psi(M^{-1}-y)=1/2$, 
\begin{equation}
\int_{M^{-1}}^{2M^{-1}}\cK_1(x)\rd{x} = \frac{1}{2}\cdot \frac{M}{2}\cdot M^{-1} = \frac{1}{4}.
\end{equation}

Notice that $\alpha\ge4$ by the assumption \eqref{prop_M_1}. Therefore $\cK_2\le 0$ is supported on $[M^{-1},\alpha M^{-1}]\subseteq [M^{-1},1/2]$, and has all 3 full pieces. Again, Lemma \ref{lem_convo} gives
\begin{equation}\label{K2}
\cK_2(x) = -\frac{3}{4}\cdot\frac{2}{\alpha-2}M\cdot \left\{\begin{array}{ll} \psi(x-M^{-1}), & M^{-1} \le x \le 2M^{-1} \\[3pt]
\frac{1}{2}, & 2M^{-1} \le x \le (\alpha-1)M^{-1} \\[3pt]
\psi(\alpha M^{-1}-x), & (\alpha-1)M^{-1} \le x \le \alpha M^{-1} 
\end{array}\right.
\end{equation}
and
\begin{equation}
\int_{M^{-1}}^{\alpha M^{-1}}\cK_2(x)\rd{x} = -\frac{3}{4}\cdot\frac{2}{\alpha-2}M\cdot\frac{1}{2}\cdot (\alpha-2)M^{-1} = -\frac{3}{4}.
\end{equation}

Notice that $W_1''*\rho_K $ is positive at $M^{-1}$, decreasing on $[M^{-1},2M^{-1}]$, constant on $[2M^{-1},(\alpha-1)M^{-1}]$, increasing on $[(\alpha-1)M^{-1},\alpha M^{-1}]$, being a positive constant on $[\alpha M^{-1},1/2]$. It is clear that $W_1'*\rho_K$ vanishes at $1/2$ by symmetry, and
\begin{equation}\begin{split}
(W_1'*\rho_K)(M^{-1}) = & \frac{M}{2M-4}\int_{[0,1]}\big(-\sgn(M^{-1}-y)+2(M^{-1}-y)\big)\rho_K(y)\rd{y} \\
& + \int_{I_{1,1}} \big((M-\frac{1}{2})+M(M^{-1}-y)\big)\rho_K(y)\rd{y} \\
= & \frac{M}{2M-4}\Big(2M^{-1}-2\cdot \frac{1}{2}\Big) + \frac{1}{2}\Big(M-\frac{1}{2}+1-M\cdot \Big(1-\frac{1}{2M}\Big)\Big) =0
\end{split}\end{equation}
using the fact that the center of mass of $\rho_K$ and $\rho_K\chi_{I_{1,1}}$ are $\frac{1}{2}$ and $1-\frac{1}{2M}$ respectively. Then we see that $W_1*\rho_K $ achieves its minimum on $[M^{-1},1/2]$ at either $M^{-1}$ or $1/2$. Therefore, to show that $(W_1*\rho_K )(x_0)=\int_{M^{-1}}^{x_0} (x_0-x)(W_1''*\rho_K )(x)\rd{x}>0$, we only need to check it for $x_0=1/2$.

Similar to \eqref{J1est} for $\cK_1$ and the symmetry of the integrand for $\cK_2$, we can estimate 
\begin{equation}
\int_{M^{-1}}^{1/2} \Big(\frac{1}{2}-x\Big)(W_1''*\rho_K )(x)\rd{x} \geq \frac{M}{M-2}\cdot \frac{1}{2}\Big(\frac{1}{2}-M^{-1}\Big)^2 + \frac{1}{4}\cdot \Big(\frac{1}{2} - \frac{3}{2}M^{-1}\Big) - \frac{3}{4}\cdot \Big(\frac{1}{2}-\frac{\alpha+1}{2} M^{-1}\Big)
\end{equation}
where we use the fact that the center of $\supp\cK_1$ and $\supp\cK_2$ are $\frac{3}{2}M^{-1}$ and $\frac{\alpha+1}{2} M^{-1}$ respectively.  Its positivity is equivalent to
\begin{equation}\label{alphacond2}
\alpha > \frac{1}{3}(M+2)
\end{equation}
which is guaranteed by \eqref{prop_M_1}.

Now we show that 
\begin{equation}\label{W1cx0}
(W_1*\rho_K)(x_0)-(W_1*\rho_K)(M^{-1})= \int_{M^{-1}}^{x_0} (x_0-x)(W_1''*\rho_K )(x)\rd{x}\ge c(x_0)>0
\end{equation}
for any $x_0\in (M^{-1},1/2)$, with $c(x_0)$ independent of $K$. In fact,  using the monotone properties of $W_1*\rho_K$ we obtained, together with the positivity of $\int_{M^{-1}}^{x_0} (x_0-x)(W_1''*\rho_K )(x)\rd{x}$ at $x_0=1/2$, it suffices to show \eqref{W1cx0} for $x_0$ near $M^{-1}$. Since $\psi$ is increasing on $[0,M^{-1}]$ with $\psi(y)+\psi(M^{-1}-y)=1/2$, we have $\psi(2M^{-1}-x)\ge \psi(x-M^{-1})$ for $x\in [M^{-1},\frac{3}{2}M^{-1}]$. Together with the fact that $\alpha\ge 4$, we see that $\cK_1(x)\ge \cK_2(x)$ for $x\in [M^{-1},\frac{3}{2}M^{-1}]$, which implies $(W_1''*\rho_K)(x)\ge \frac{M}{M-2}$. Therefore we see \eqref{W1cx0} with $c(x_0) = \frac{M}{M-2}\int_{M^{-1}}^{x_0} (x_0-x)\rd{x} = \frac{M}{2(M-2)}(x_0-M^{-1})^2>0$ for $x_0\in (M^{-1},\frac{3}{2}M^{-1}]$.

\

{\bf STEP 3}: show that $ \int_{M^{-1}}^{x_0} (x_0-x)((\omega_2'-\omega_1')*\rho_K )(x)\rd{x}>0$ for $K\ge 2$.

Compared to STEP 1, $\cJ_1,\cJ_3,\cJ_4$ appear in exactly the same form. The terms $\cJ_2,\cJ_5$ are different because they involve convolutions of $\rho_K $ with $\bar{\chi}_{[(M-\alpha)_{(1)},(M-2)_{(1)}]}$ and thus have contributions from $\rho_K \chi_{I_{1,1}}$, since $[(M-\alpha)_{(1)},(M-2)_{(1)}]\subseteq [1/2,1]$. The new term corresponding to $\cJ_2+\cJ_5$ is
\begin{equation}
\tilde{\cJ}_2 = -\frac{1}{4}\cdot\frac{2}{\alpha-2}M\bar{\chi}_{[(M-\alpha)_{(1)},(M-2)_{(1)}]} * \rho_K  \chi_{I_{1,1}} = \frac{1}{3}\cK_2 \le 0
\end{equation} 
where $\cK_2$ is defined in \eqref{K2}. Then STEP 1-3 can be repeated as before. For STEP 1-1, we have the same estimate \eqref{J1est} for $\cJ_1$ with $x_0>M^{-1}+1_{(2)}$. For $x_0\in(M^{-1},M^{-1}+1_{(2)}]$, 
we have the lower bound 
\begin{equation}
\int_{M^{-1}}^{x_0} (x_0-x)\cJ_1(x)\rd{x} = \int_{M^{-1}}^{M^{-1}+1_{(2)}} \max\{x_0-x,0\}\cJ_1(x)\rd{x} \ge \int_{M^{-1}}^{x_0} (x_0-x)\frac{M^2}{8}\rd{x}
\end{equation}
by replacing $\cJ_1(x)$ with its average on $[M^{-1},M^{-1}+1_{(2)}]$, using the decreasing property of $\cJ_1$ and $\max(x_0-x,0)$. We also have the lower bound 
\begin{equation}
\int_{M^{-1}}^{x_0} (x_0-x)\tilde{\cJ}_2(x)\rd{x} \ge -\int_{M^{-1}}^{x_0} (x_0-x)\frac{1}{3}\cdot \frac{3}{4}\cdot\frac{2}{\alpha-2}M\cdot\frac{1}{2}\rd{x} = -\int_{M^{-1}}^{x_0} (x_0-x)\frac{2}{4(\alpha-2)}M\rd{x}
\end{equation}
using \eqref{K2} for $\tilde{\cJ}_2=\frac{1}{3}\cK_2$ and the upper bound $\psi(y)\le 1/2$. Since $\alpha\ge 4$ and $M\ge 4$ by \eqref{prop_M_1}, we get $\int_{M^{-1}}^{x_0} (x_0-x)(\cJ_1(x)+\tilde{\cJ}_2(x))\rd{x}>0$, and thus $\int_{M^{-1}}^{x_0} (x_0-x)((\omega_2'-\omega_1')*\rho_K)(x)\rd{x}>0$.

For STEP 1-2 (i.e., $x_0\in [M^{-1}+1_{(2)},2M^{-1}]\subseteq [M^{-1},\alpha M^{-1}]$), we have an extra negative term $\tilde{\cJ}_2$. By \eqref{K2}, we have $\tilde{\cJ}_2(x) \ge -\frac{1}{4}\cdot\frac{2}{\alpha-2}M\cdot \frac{1}{2}$ for any $x\in [M^{-1}+1_{(2)},2M^{-1}]$. Therefore its contribution in $ \int_{M^{-1}}^{x_0} (x_0-x)((\omega_2'-\omega_1')*\rho_K )(x)\rd{x}$ can be estimated by
\begin{equation}\begin{split}
\int_{M^{-1}}^{x_0} (x_0-x)\tilde{\cJ}_2(x)\rd{x} \ge &  \int_{M^{-1}}^{x_0} (x_0-x)\Big(-\frac{1}{4}\cdot\frac{2}{\alpha-2}M\cdot \frac{1}{2}\Big)\rd{x}
=  -\frac{1}{4(\alpha-2)}M\cdot \frac{1}{2}(x_0-M^{-1})^2 \\
\ge & -\frac{M}{8(\alpha-2)}(2M^{-1}-M^{-1})^2 =  -\frac{1}{2}\cdot 2^{-2}\cdot \Big(\frac{M}{\alpha-2}\Big)_{(2)}
\end{split}\end{equation}
where in the last equality we rewrite it in a similar form as the positive contribution \eqref{J1int}. Therefore, compared to the condition \eqref{alphacond1}, we have $ \int_{M^{-1}}^{x_0} (x_0-x)((\omega_2'-\omega_1')*\rho_K )(x)\rd{x}>0$ for $x_0\in [M^{-1}+1_{(2)},2M^{-1}]$ as long as a more restrictive condition
\begin{equation}\label{alphacond30}
M-\alpha-\frac{3}{2} \ge \frac{3}{2}(\alpha+3) + \frac{M}{\alpha-2}
\end{equation}
is satisfied. We claim that \eqref{alphacond30} is a consequence of \eqref{prop_M_1}. In fact, \eqref{prop_M_1} implies that $\alpha\ge 24$. Therefore $\frac{M}{\alpha-2}\le \frac{24}{22}\cdot\frac{M}{\alpha} \le \frac{24}{22}\cdot 3<4$ by \eqref{alphacond2}. Using this, we see that \eqref{alphacond30} can be guaranteed by $\alpha \le \frac{2}{5}(M-10)$ from \eqref{prop_M_1}.

Finally notice that the strictly positive contribution from $W_1$ (as in STEP 2) appears in any case of $K\ge 1$. Therefore we get \eqref{prop_M_2} with $c(x_0)$ independent of $K$.
\end{proof}

Next we use Proposition \ref{prop_M} and combined with self-similar arguments to obtain \eqref{W2cond} for $W_k*\rho_k$, that is the necessary condition for $d_2$-local minimizers.

\begin{proposition}\label{prop_rhok}
Assume $M$ and $\alpha$ satisfy \eqref{prop_M_1}, and let $k\ge 1$. Then $W_k*\rho_k$ is constant on $\supp\rho_k$, denoted as $c_k$, and for any $x_0\not\in \supp \rho_j$, there exists $c(x_0)>0$ such that 
\begin{equation}\label{prop_rhok_1}
(W_k*\rho_k)(x_0) - c_k \ge c(x_0),\quad \forall k\ge j\,.
\end{equation}
\end{proposition}

See Figure \ref{fig_Wcantor} for an illustration.

\begin{figure}
\begin{center}
	\includegraphics[width=0.8\textwidth]{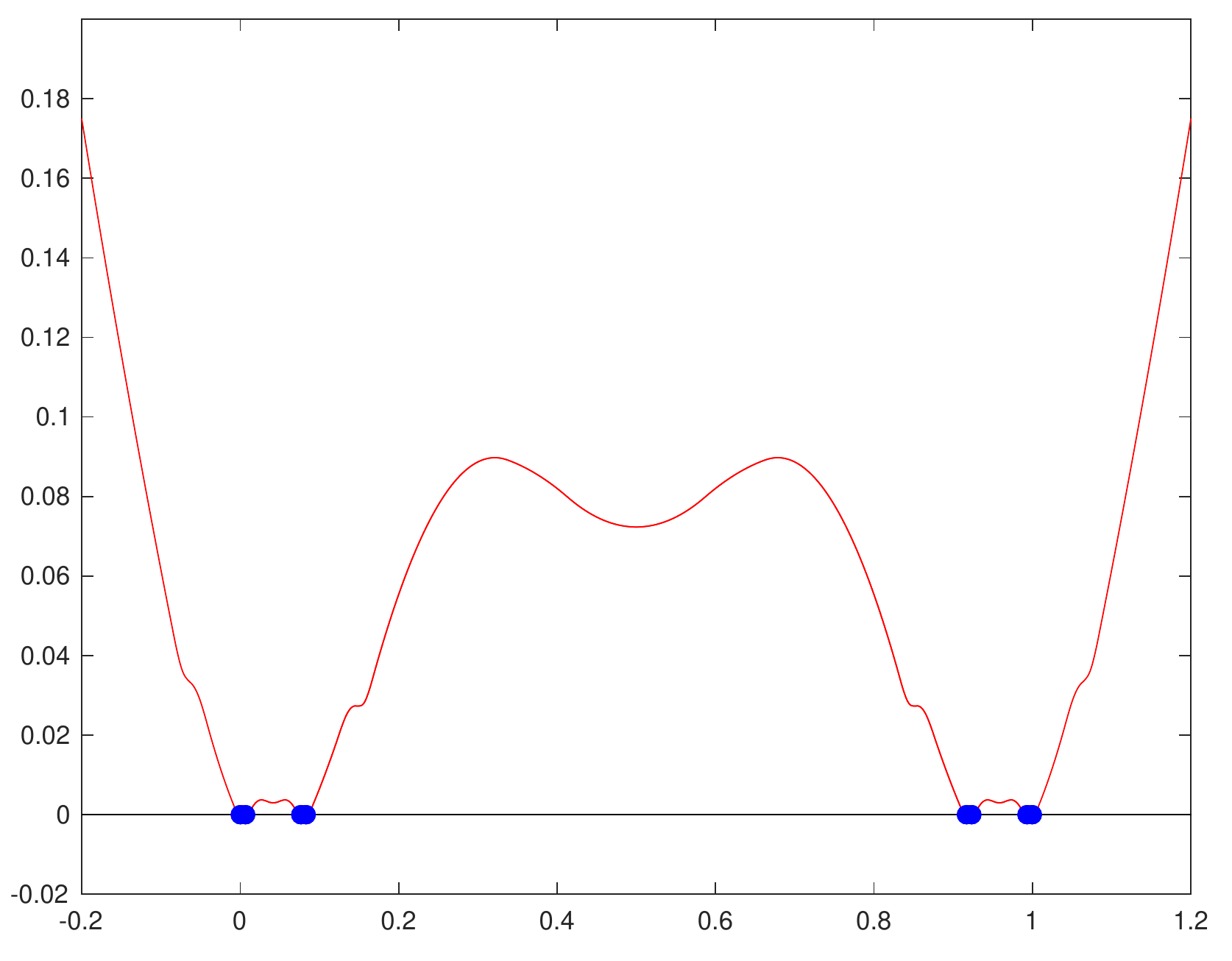}
	\caption{The potential $W_4*\rho_4$, with $M=12$ and $\alpha=5$. The blue dots indicate $\supp\rho_4$. One can see that $W_4*\rho_4$ is constant on $\supp\rho_4$ and larger elsewhere. Notice that $M$ and $\alpha$ do not satisfy the condition \eqref{prop_M_1}. Nevertheless, the conclusion of Proposition \ref{prop_rhok} still holds. }
	\label{fig_Wcantor}
\end{center}
\end{figure}

\begin{proof}

{\bf STEP 1}: Inside $[0,1]$.

We start by noticing the self-similar relations
\begin{equation}\label{selfsim}
\rho_k(x) = \frac{M}{2}\rho_{k-1}(M x),\quad \omega_k'(x) = M\omega_{k-1}'(M x),\quad \forall x\in [0,M^{-1}].
\end{equation}

We first use induction on $k$ to prove 
\begin{equation}\label{prop_rhok_2}
\mbox{ $W_k*\rho_k$ is constant $c_k$ on $\supp\rho_k$, and  }(W_k*\rho_k)(x) > c_k,\, \forall x\in [0,1]\backslash \supp\rho_k.
\end{equation}
For the case $k=1$, Theorem \ref{thm_Wkss} implies that $\rho_1$ is a steady state for $W_1$, i.e., $W_1*\rho_1$ is constant on $I_{1,0}$ and $I_{1,1}$, and these two constants are the same by symmetry. Combined with Proposition \ref{prop_M}, we obtain \eqref{prop_rhok_2} for $k=1$.

Suppose \eqref{prop_rhok_2} is true for $k-1$.
We first apply Proposition \ref{prop_M} to see that $(W_k*\rho_k)(x) > (W_k*\rho_k)(M^{-1})=c_k$ for $x\in(M^{-1},1-M^{-1})$. Notice that $\supp\rho_k\subseteq I_{1,0}\cup I_{1,1}$ and $\rho_k$ is symmetric about $1/2$.
By this symmetry, we can reduce ourselves to the interval $x\in I_{1,0}=[0,M^{-1}]$ to prove that \eqref{prop_rhok_2} holds for
$x\in[0,1]$, since the same conclusion will be true for $I_{1,1}=[1-M^{-1},1]$, and thus \eqref{prop_rhok_2} is proved for $k$.

Take $x\in I_{1,0}=[0,M^{-1}]$, and then we have
\begin{equation}\begin{split}
(\omega_k'*(\rho_k\chi_{I_{1,0}}))(x) = & \int_0^{M^{-1}}M\omega_{k-1}'(M (x-y))\cdot\frac{M}{2}\rho_{k-1}(M y)\rd{y} = \frac{M}{2}\int_0^1\omega_{k-1}'(M x-y_1)\rho_{k-1}(y_1)\rd{y_1} \\
= & \frac{M}{2}(\omega_{k-1}'*\rho_{k-1})(M x)
\end{split}\end{equation}
by \eqref{selfsim} and the change of variable $y_1=M y$. Combined with
\begin{equation}\begin{split}
((W_k''-\omega_k')*(\rho_k\chi_{I_{1,0}}))(x) = &\frac{M}{M-2}((-\delta+1)*(\rho_k\chi_{I_{1,0}}))(x) = -\frac{M}{M-2}\rho_k(x) + \frac{M}{M-2}\int_{I_{1,0}}\rho_k(y)\rd{y}\\
= & -\frac{M}{M-2}\rho_k(x)+\frac{M}{2M-4}\,,
\end{split}\end{equation}
we get
\begin{equation}\label{Wk10}
(W_k''*(\rho_k\chi_{I_{1,0}}))(x) = \frac{M}{2}(\omega_{k-1}'*\rho_{k-1})(M x)-\frac{M^2}{2M-4}\rho_{k-1}(M x)+\frac{M}{2M-4}\,.
\end{equation}

Since $|x-I_{1,1}|\subseteq [(M-2)_{(1)},M_{(1)}]$ on which $W_k''=\frac{M}{M-2}+M$, we have
\begin{equation}
(W_k''*(\rho_k\chi_{I_{1,1}}))(x) = \Big(\frac{M}{M-2}+M\Big)\int_{I_{1,1}}\rho_k(y)\rd{y}=\frac{M}{2M-4}+\frac{M}{2}\,.
\end{equation}
Adding with \eqref{Wk10}, we get
\begin{equation}
(W_k''*\rho_k)(x) = \frac{M}{2}(\omega_{k-1}'*\rho_{k-1})(M x)-\frac{M^2}{2M-4}\rho_{k-1}(M x)+\frac{M^2}{2M-4} = \frac{M}{2}(W_{k-1}''*\rho_{k-1})(M x).
\end{equation}
Together with $(W_k'*\rho_k)(0)=(W_{k-1}'*\rho_{k-1})(0)=0$ from Theorem \ref{thm_Wkss} and integrating twice, we conclude that
\begin{equation}\label{Wrhok0}
(W_k*\rho_k)(x)-(W_k*\rho_k)(0) = \frac{1}{2M}\big((W_{k-1}*\rho_{k-1})(M x)-(W_{k-1}*\rho_{k-1})(0)\big) .
\end{equation}
Notice that $\{Mx:x\in\supp\rho_k\cap I_{1,0}\}=\supp\rho_{k-1}$. Therefore, the induction hypothesis implies that $W_k*\rho_k$ is constant on $\supp\rho_k \cap [0,M^{-1}]$, calling it $c_k$, and $(W_k*\rho_k)(x)>c_k$ for any $x\in [0,M^{-1}]\backslash \supp\rho_k$.

For $x_0\in [0,1]\backslash \supp\rho_j$, to see that the difference $(W_k*\rho_k)(x_0)-c_k$ can be bounded from below uniformly in $k\ge j$, we notice that iteratively applying \eqref{Wrhok0} and its symmetric counterparts gives \begin{equation}
(W_k*\rho_k)(x_0)-(W_k*\rho_k)(0) = \Big(\frac{1}{2M}\Big)^{j'}\big((W_{k-j'}*\rho_{k-j'})(x_1)-(W_{k-j'}*\rho_{k-j'})(0)\big) 
\end{equation} 
for some $0\le j' \le j$ and $x_1\in [M^{-1},1-M^{-1}]$, for any $k\ge j$. Then the lower bound of  $(W_k*\rho_k)(x_0)-(W_k*\rho_k)(0)$ is given by \eqref{prop_M_2} for $x_1$, independent of $k$.

{\bf STEP 2}: Outside $[0,1]$.
Then we prove
\begin{equation}\label{prop_rhok_3}
(W_k*\rho_k)(x_0) > c_k,\quad \forall x_0\in (-\infty,0)\cup (1,\infty).
\end{equation}
It suffices to treat $x_0\in (-\infty,0)$ by symmetry. 

If $x_0\in (M^{-1}-\frac{1}{2},0)$, then we may write
\begin{equation}
(W_k*\rho_k)(x_0)-(W_k*\rho_k)(0)=\int_{x_0}^0 (x-x_0)(W_k''*\rho_k)(x)\rd{x}
\end{equation}
similar to \eqref{Wx0m1}. In this interval, we make use of the decomposition \eqref{eq:decomp} and analyse each term separately.

By STEP 1 of the proof of Proposition \ref{prop_M} and the symmetry, we get
\begin{equation}
\int_{x_0}^0 (x-x_0)((\omega_j'-\omega_{j-1}')*\rho_k)(x)\rd{x} > 0,\quad 3\le j \le k.
\end{equation}
For the contribution from $W_1'' = \frac{M}{M-2}(-\delta+1)+\omega_1'$ (corresponding to STEP 2 of the proof of Proposition \ref{prop_M}) where $\supp\omega_1'=[(M-\alpha)_{(1)},1]\subseteq[1/2,1]$, we have
\begin{equation}
(W_1''*\rho_k)(x) = \frac{M}{M-2}+(\omega_1'*(\rho_k\chi_{I_{1,1}}))(x) >0
\end{equation}
since $\dist(x_0,I_{1,1})\ge (M-1)_{(1)}$ and thus the last convolution only uses the nonzero values of $\omega_1'$ on $[(M-1)_{(1)},1]$ on which $\omega_1'=M>0$. Therefore
\begin{equation}
\int_{x_0}^0 (x-x_0)(W_1''*\rho_k)(x)\rd{x} > 0.
\end{equation}
For the contribution from $\omega_2'-\omega_1'$, compared to STEP 3 of the proof of Proposition \ref{prop_M}, we have the same terms $\cJ_1,\cJ_3,\cJ_4$ by symmetry, since they only involve $\rho_k\chi_{1,0}$. There is no contribution from $\rho_k\chi_{1,1}$ since $\supp(\omega_2'-\omega_1')\subseteq [(M-\alpha)_{(2)},(M-2)_{(1)}]$ and $\dist(x_0,I_{1,1})\ge (M-1)_{(1)}$. Therefore the negative term $\tilde{\cJ}_2$ is absent, and we get 
\begin{equation}
\int_{x_0}^0 (x-x_0)((\omega_2'-\omega_1')*\rho_k)(x)\rd{x} > 0.
\end{equation}
from STEP 3 of the proof of Proposition \ref{prop_M}. Therefore we conclude \eqref{prop_rhok_3} for  $x_0\in (M^{-1}-\frac{1}{2},0)$. The difference $(W_k*\rho_k)(x_0)-c_k$ can be bounded from below uniformly in $k$ because the positive contribution from $W_1$ appears in any case of $k\ge 1$.

If $x_0\le M^{-1}-\frac{1}{2}$, we will analyze $(W_k'*\rho_k)(x_0)$. We have
\begin{equation}
\Big(\frac{M}{2M-4}(-\sgn(x)+2x)*\rho_k\Big)(x_0)=\frac{M}{2M-4}\Big(1+2\Big(x_0-\frac{1}{2}\Big)\Big)=\frac{M}{M-2}x_0 \le -\frac{M}{M-2}\cdot \Big(\frac{1}{2}-\frac{1}{M}\Big) = -\frac{1}{2}
\end{equation}
using the fact that the center of mass of $\rho_k$ is $\frac{1}{2}$. For $\omega_k*\rho_k$, first notice that $\dist(x_0,I_{1,1}) > 1-M^{-1}+\frac{1}{2}-M^{-1}>1$, we have
\begin{equation}
(\omega_k*(\rho_k\chi_{I_{1,1}}))(x_0) = -\frac{1}{2}\int_{I_{1,1}}\rho_k\rd{x} = -\frac{1}{4}
\end{equation}
since $\omega_k=-\frac{1}{2}$ on $(-\infty,-1]$. Then notice that $x_0-x\in (-\infty,M^{-1}-\frac{1}{2}]$ for any $x\in I_{0,1}$. On $[\frac{1}{2}-M^{-1},\infty)\subseteq [M^{-1},\infty)$, the expression for $\omega_k$ shows that $\omega_k\ge -\frac{3}{2}$ on it. Therefore, by its odd property, $\omega_k\le \frac{3}{2}$ on $(-\infty,M^{-1}-\frac{1}{2}]$. Therefore
\begin{equation}
(\omega_k*(\rho_k\chi_{I_{0,1}}))(x_0) \le  \frac{3}{2}\int_{I_{1,1}}\rho_k\rd{x} = \frac{3}{4}\,.
\end{equation}
Combined with the above three estimates, we see that $(W_k'*\rho_k)(x_0)$ for any $x_0\le M^{-1}-\frac{1}{2}$. Combined with \eqref{prop_rhok_3} for $x_0\in (M^{-1}-\frac{1}{2},0)$, we get \eqref{prop_rhok_3} for $x_0\le M^{-1}-\frac{1}{2}$.
\end{proof}

Define the limiting potential $W_\infty$ by the pointwise limit
\begin{equation}
W_\infty'(x)=\lim_{k\rightarrow\infty} W_k'(x),\quad \forall x\ne 0.
\end{equation}
We remind the reader that the weak limit of $\rho_k$ is denoted by $\rho_\infty$, that is the uniform distribution on the Cantor set $S=\bigcap_{k=0}^\infty S_k$. The main theorem of this section asserts that $\rho_\infty$ is also a steady state for $W_\infty$ satisfying the $d_2$-local minimizer condition.

\begin{theorem}\label{thm_M}
For $M$ and $\alpha$ satisfying the condition \eqref{prop_M_1} in Proposition \ref{prop_M}, there holds
\begin{equation}\label{thm_M_1}
(W_\infty*\rho_\infty)(x)=c_\infty,\, \forall x\in\supp\rho_\infty,\quad (W_\infty*\rho_\infty)(x)> c_\infty,\, \forall x\notin\supp\rho_\infty
\end{equation}
for some constant $c_\infty$. Also, $W_\infty*\rho_\infty\in C^{1,\gamma}_{loc}$ for some $\gamma>0$. Therefore, $W_\infty'*\rho_\infty$ is defined pointwisely, satisfying $W_\infty'*\rho_\infty=0$ on $\supp\rho_\infty$, i.e., $\rho_\infty$ is a steady state for the interaction energy associated to the potential $W_\infty$.
\end{theorem}

\begin{proof}
It is clear that $W_\infty\in L^\infty_{loc}$ and therefore $W_\infty*\rho_\infty$ is well-defined since $\rho_\infty$ is compactly supported. We first claim that $W_k*\rho_k\rightarrow W_\infty*\rho_\infty$ pointwisely as $k\rightarrow\infty$. In fact, we write $(W_k*\rho_k)(x)- (W_\infty*\rho_\infty)(x) = ((W_k*\rho_k)(x)- (W_k*\rho_\infty)(x)) + ((W_k*\rho_\infty)(x)- (W_\infty*\rho_\infty)(x))$ and control the two terms separately.

For the first term, notice that
\begin{equation}\begin{split}
|(W_k*\rho_k)(x)- (W_k*\rho_\infty)(x)| = & \left|\int_\R W_k(x-y)(\rho_k(y)-\rho_\infty(y))\rd{y}\right| \\
\le & \sum_{l=0}^{2^k-1} \left|\int_{I_{k,l}} W_k(x-y)(\rho_k(y)-\rho_\infty(y))\rd{y}\right| \\
\end{split}\end{equation}
since $\rho_k$ and $\rho_\infty$ are supported inside $\supp\rho_k=\bigcup_{l=0}^{2^k-1} I_{k,l}$. Notice that we have the mass conservation property $\int_{I_{k,l}} (\rho_k(y)-\rho_\infty(y))\rd{y} = 0$. Therefore, denoting $y_0$ as the left endpoint of $I_{k,l}$,
\begin{equation}\begin{split}
\left|\int_{I_{k,l}} W_k(x-y)(\rho_k(y)-\rho_\infty(y))\rd{y}\right| = & \left|\int_{I_{k,l}} (W_k(x-y)-W_k(x-y_0))(\rho_k(y)-\rho_\infty(y))\rd{y}\right| \\
\le & \|W_k'\|_{L^\infty([x-1,x])} |I_{k,l}|\int_{I_{k,l}}|\rho_k(y)-\rho_\infty(y)|\rd{y} \\
\le & C(x) M^{-k}2^{-k},
\end{split}\end{equation}
where $C(x)$ denotes a constant depending on $x$, and it is independent of $k$ by the construction of $W_k$ as a locally Lipschitz function. Summing over $l$, we get
\begin{equation}\begin{split}
|(W_k*\rho_k)(x)- (W_k*\rho_\infty)(x)| \le C(x) M^{-k}.
\end{split}\end{equation}

For the second term, notice that $\|W_k'-W_{k-1}'\|_{L^\infty} \le C$ and $|\supp(W_k'-W_{k-1}')| \le CM^{-k}$  by construction. Therefore, by adding proper constants on $W_k$, we have $\|W_k-W_{k-1}\|_{L^\infty} \le CM^{-k}$. Then summing over $k,k+1,\dots$ gives $\|W_k-W_\infty\|_{L^\infty} \le CM^{-k}$. Therefore, we conclude
\begin{equation}
|(W_k*\rho_\infty)(x)- (W_\infty*\rho_\infty)(x)| \le \|W_k-W_\infty\|_{L^\infty}\cdot \int_\mathbb{R}\rho_\infty(y)\rd{y} \le CM^{-k}
\end{equation}
and the claimed convergence is proved.

Recalling Proposition \ref{prop_rhok} and applying this convergence to $x_1,x_2\in \supp\rho_\infty$, we see that $(W_\infty*\rho_\infty)(x_1)=(W_\infty*\rho_\infty)(x_2)$ since $\supp\rho_\infty\subseteq \supp\rho_k$ for any $k$. Applying this convergence to $x_1\in \supp\rho_\infty, x_2\notin \supp\rho_\infty$, we see that  $(W_\infty*\rho_\infty)(x_1)\le (W_\infty*\rho_\infty)(x_2) - c(x_2)$ with $c(x_2)>0$. This proves \eqref{thm_M_1}.

Next we prove the (local) H\"older continuity of the velocity field $u:=-W_\infty'*\rho_\infty$. Fix $R>0$ large. By construction, in:inbox 
\begin{equation}
\|W_\infty'\|_{L^\infty([-R,R])}\le C,\quad\|W_\infty''\|_{L^\infty([-R,R]\backslash(-\kappa,\kappa))}\le \frac{C}{\kappa}
\end{equation}
for any $\kappa>0$. 

Take $0<\epsilon<1/2$, $x\in [-(R-2),R-2]$ and write
\begin{equation}\begin{split}
u(x) & -u(x+\epsilon) =  \int_{[0,1]} (W_\infty'(x-y+\epsilon)-W_\infty'(x-y))\rho_\infty(y)\rd{y} \\
= & \int_{|x-y|\ge\kappa} (W_\infty'(x-y+\epsilon)-W_\infty'(x-y))\rho_\infty(y)\rd{y} + \int_{|x-y|<\kappa} (W_\infty'(x-y+\epsilon)-W_\infty'(x-y))\rho_\infty(y)\rd{y}
\end{split}\end{equation}
where $\kappa>2\epsilon$ is to be chosen. Since $x\in [-(R-2),R-2]$, $\supp\rho_\infty\subseteq[0,1]$ and $\epsilon<1/2$, all the arguments in $W_\infty'$ above are inside $[-R,R]$. Then we estimate the first integral by
\begin{equation}
\left|\int_{|x-y|\ge\kappa} (W_\infty'(x-y+\epsilon)-W_\infty'(x-y))\rho_\infty(y)\rd{y}\right| \le \epsilon \|W_\infty''\|_{L^\infty([-R,R]\backslash(-\kappa/2,\kappa/2)} \int_{[0,1]} \rho_\infty(y)\rd{y} \le \frac{C\epsilon}{\kappa}
\end{equation}
and the second integral by
\begin{equation}
\left|\int_{|x-y|<\kappa} (W_\infty'(x-y+\epsilon)-W_\infty'(x-y))\rho_\infty(y)\rd{y}\right| \le 2\|W_\infty'\|_{L^\infty([-R,R])}\int_{|x-y|<\kappa}\rho_\infty(y)\rd{y} \le C2^{-k}
\end{equation}
where $k=\floor{-\log_M \frac{2\kappa}{M-2}}$, and the last inequality follows from the facts that $I_{k,l}$ and $I_{k,l'}$ has distance at least $(M-2)_{(k)}$, and $(x-\kappa,x+\kappa)$ can intersect $I_{k,l}$ for at most one value of $l$ if $2\kappa \le (M-2)_{(k)}$.

Therefore
\begin{equation}
|u(x)-u(x+\epsilon)| \le C(\frac{\epsilon}{\kappa} + \kappa^{\frac{\ln 2}{\ln M}})
\end{equation}
Then take $\kappa = \max\{\epsilon^{1/(1+\frac{\ln 2}{\ln M})},2\epsilon\}$, we get the H\"older continuity of $u$ with $\gamma=1-1/(1+\frac{\ln 2}{\ln M})$.
\end{proof}

\section{Numerical simulations}

In this section, we numerically illustrate the main result about almost fractal behavior of the support of steady states in Theorem \ref{thm_cons1} of Section 8 and Theorem \ref{lem_concave} in Section 7. We show several numerical simulations with the potential $W$ constructed in \eqref{Wcons}. We consider a particle gradient flow in 2D:
\begin{equation}\label{xiode}
\dot{\bx}_i = -\frac{1}{N}\sum_{j\ne i} \nabla W(\bx_i-\bx_j),\quad i=1,\dots,N
\end{equation}
where $W$ is defined by 
\begin{equation}
W(\bx) = c_{d,\alpha}|\bx|^{\alpha-d} + C_2\frac{|\bx|^2}{2}- c_W\sum_{k=1}^K \lambda^{(\alpha-d)k} \exp(-\frac{|\bx|^2}{2\lambda^{2k}})
\end{equation}
We take the parameters 
\begin{equation}
N=2000,\quad \alpha = 3,\quad K = 7,\quad c_W = 0.25,\quad C_2 = 0.2\,.
\end{equation}
To solve \eqref{xiode}, we take the initial data as random points in $[0,0.5]^2$ with uniform distribution, and apply the fourth order Runge-Kutta method with time step $\Delta t=0.01$ to solve \eqref{xiode} numerically, and stop at final time $T=200$. See Figures \ref{fig_frac_1} and \ref{fig_frac_2} for the results at $T=200$, for the choices $\lambda=0.15$ and $\lambda=0.1$ respectively. Both figures show fractal behavior of the particle distributions. In Figure \ref{fig_frac_1} one can see 7 layers of the fractal structures, as marked on the pictures. In Figure \ref{fig_frac_2} one can only see 4 layers because the number of particles $N$ is not sufficiently large, and at Layer 4 each cluster has only two or three particles, which are not enough to resolve the next layer.

\begin{figure}
\begin{center}
	\includegraphics[width=0.95\textwidth]{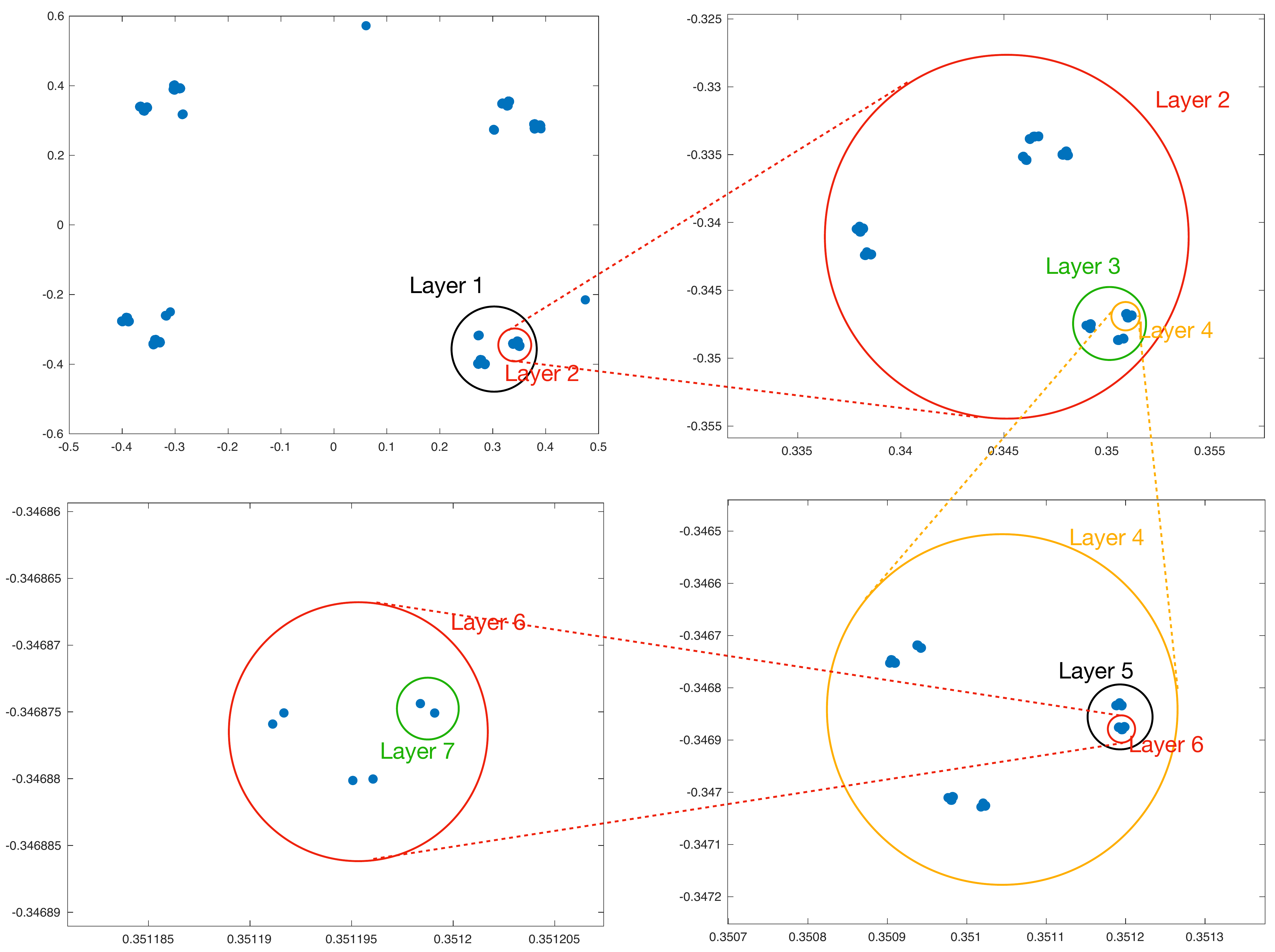}
	\caption{Large time simulation of \eqref{xiode} with $\lambda=0.15$.}
	\label{fig_frac_1}
\end{center}
\end{figure}
\begin{figure}
\begin{center}
	\includegraphics[width=0.95\textwidth]{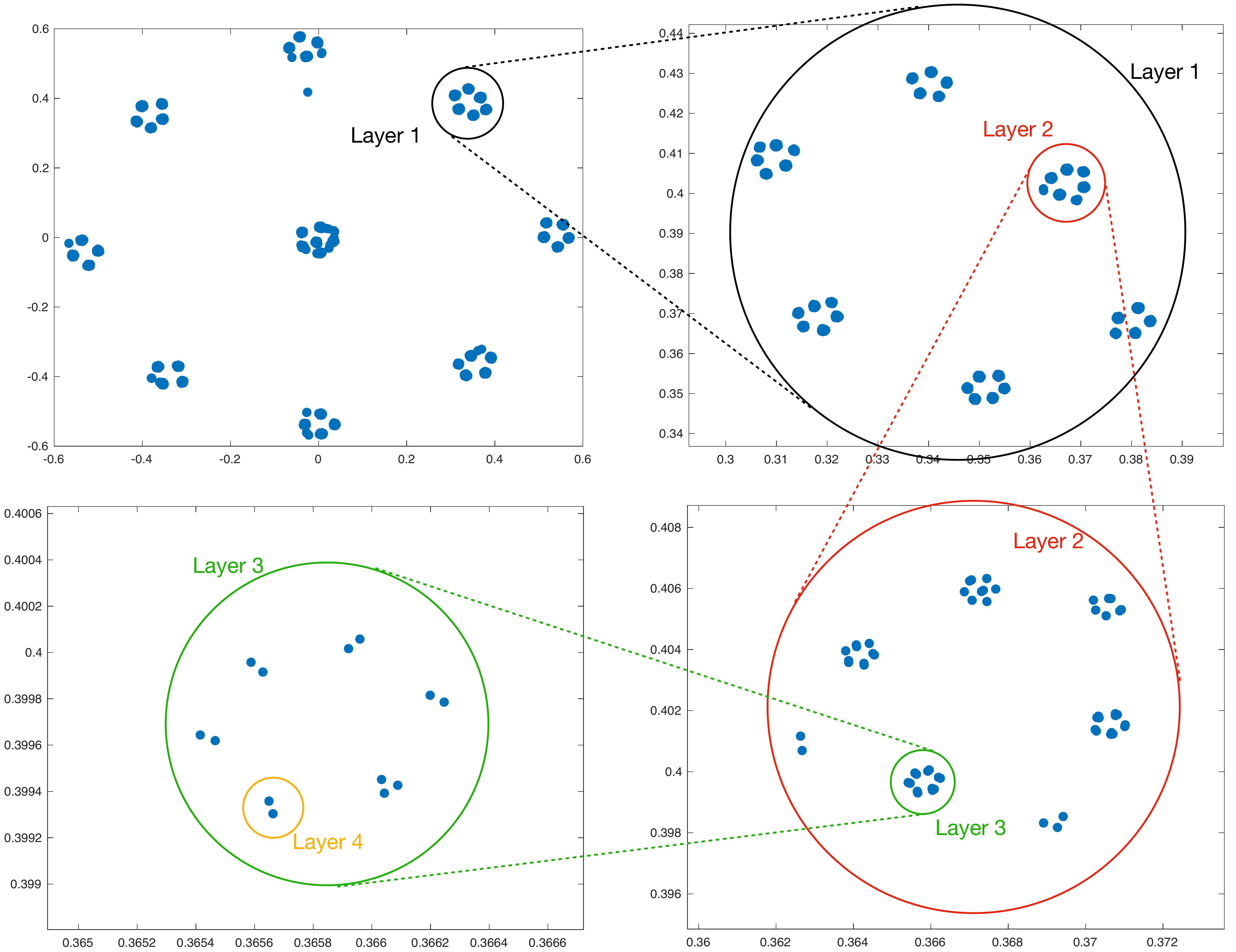}
	\caption{Large time simulation of \eqref{xiode} with $\lambda=0.1$.}
	\label{fig_frac_2}
\end{center}
\end{figure}

\section*{Acknowledgements}
JAC and RS were supported by the Advanced Grant Nonlocal-CPD (Nonlocal PDEs for Complex Particle Dynamics: Phase Transitions, Patterns and Synchronization) of the European Research Council Executive Agency (ERC) under the European Union's Horizon 2020 research and innovation programme (grant agreement No. 883363). JAC was also partially supported by the EPSRC grant number EP/T022132/1. RS was supported in part by NSF and ONR grants DMS1613911 and N00014-1812465.

\section*{Appendix: FLIC property of some 1D power-law potentials}

In this appendix we prove the FLIC property of the 1D potential $-|x|^b/b$ for $0\le b < 1$.
\begin{theorem}
Let the spatial dimension $d=1$. Let $W(x)=-|x|^b/b$  with $0<b<1$, or $W(x)=-\ln|x|$ (denoting $b=0$ in this case). Let $\mu\ne 0$ be a compactly supported signed measure on $\mathbb{R}$ with $\int |\mu| \rd{x}<\infty$ and $\int \mu\rd{x}=0$. Then
\begin{equation}
    \int_{\mathbb{R}} (W*\mu)(x)\mu(x)\rd{x} = c \int_{\mathbb{R}} |\xi|^{-b-1}|\hat{\mu}(\xi)|^2\rd{\xi} > 0.
\end{equation}
\end{theorem}

\begin{proof}
By a mollification argument, we may assume that $\mu$ is a smooth function. Let $m(x) = \int_{(-\infty,x)} \mu(y)\rd{y}$. Then $m$ is compactly supported smooth function. Denote $H$ as the Hilbert transform.

We first treat the case $0<b<1$. In this case
\begin{equation}
    \int_{\mathbb{R}} (W*\mu)(x)\mu(x)\rd{x} = -\int_{\mathbb{R}} (W'*\mu)(x)m(x)\rd{x} = -\int_{\mathbb{R}} H[W'*\mu](x)H[m](x)\rd{x},
\end{equation}
since $\int f g \rd{x} = \int H[f]H[g]\rd{x}$ for any real functions $f,g\in L^2$, and $W'*\mu=(-|\cdot|^{b-1}\sgn(\cdot))*\mu$ is in $L^2$ since $\mu$ is smooth, compactly supported and mean-zero. Notice that 
\begin{equation}
    H[W'](x) = H[-|\cdot|^{b-1}\sgn(\cdot)](x) = -\frac{1}{\pi}\pv \int_{\mathbb{R}} |y|^{b-1}\sgn(y)\frac{1}{x-y}\rd{y}
\end{equation}
is well-defined for any $x\ne 0$ since the last integral has a well-defined principle value near $y=x$ and decays like $|y|^{b-2}$ at infinity. Also, for any $\lambda\ne 0$,
\begin{equation}
    H[W'](\lambda x)  = -\frac{1}{\pi}\pv \int_{\mathbb{R}} |y|^{b-1}\sgn(y)\frac{1}{\lambda x-y}\rd{y} = |\lambda|^{b-1}H[W'](x),
\end{equation}
by a change of variable $y=\lambda x$. This shows
\begin{equation}
    H[W'](x) = c_1|x|^{b-1}, \quad c_1 = -\frac{1}{\pi}\pv \int_{\mathbb{R}} |y|^{b-1}\sgn(y)\frac{1}{1-y}\rd{y}.
\end{equation}
Therefore, we conclude
\begin{equation}\begin{split}
    \int_{\mathbb{R}} (W*\mu)(x)\mu(x)\rd{x} = &  -\int_{\mathbb{R}} (H[W']*\mu)(x)H[m](x)\rd{x} = -c_1 \int_{\mathbb{R}} (|\cdot|^{b-1}*\mu)(x)H[m](x)\rd{x} \\ = & -\frac{c_1c_2}{2\pi}\int_{\mathbb{R}} |\xi|^{-b}\hat{\mu}(\xi)\,i \,\sgn(\xi) \bar{\hat{m}}(\xi)\rd{\xi} 
    =  \frac{c_1c_2}{2\pi} \int_{\mathbb{R}} |\xi|^{-b-1}|\hat{\mu}(\xi)|^2\rd{\xi},
\end{split}\end{equation}
with $c_2>0$, using the fact that $|\cdot|^{b-1}*\mu$ is $L^2$ and $\hat{\mu}(\xi) = 2\pi i \xi \hat{m}(\xi)$.

Taking $\mu(x) = \delta(x)-\delta(x-1)$, it is clear that $\int (W*\mu)(x)\mu(x)\rd{x} = 2W(0)-2W(1)>0$. Therefore the coefficient $\frac{c_1c_2}{2\pi}>0$. This finishes the proof of the case $0<b<1$.

In the case $b=0$, the conclusion follows from
\begin{equation}\begin{split}
    \int_{\mathbb{R}} (W*\mu)(x)\mu(x)\rd{x} = & -\int_{\mathbb{R}} (W*\mu)'(x)m(x)\rd{x} = \pi\int_{\mathbb{R}} H[\mu](x)m(x)\rd{x} \\
    = & \pi \int_{\mathbb{R}} (-i)\sgn(\xi)\hat{\mu}(\xi) \bar{\hat{m}}(\xi)\rd{\xi}  = \frac{1}{2}\int_{\mathbb{R}} \frac{1}{|\xi|}|\hat{\mu}(\xi)|^2\rd{\xi}.
\end{split}\end{equation}

\end{proof}

\bibliographystyle{abbrv}
\bibliography{biblio}

\end{document}